\documentclass[reqno]{amsart}
\usepackage{bm}
\usepackage{amsmath,amssymb}
\usepackage[numbers]{natbib}
\usepackage{hyperref}
\usepackage{graphicx}

\newcommand{\vx}{\bm{x}}
\newcommand{\vy}{\bm{y}}

\newcommand{\vp}{\bm{p}}

\newcommand{\vn}{\bm{n}}

\newcommand{\vvE}{\vec{\bm{E}}}
\newcommand{\vvPi}{\vec{\bm{\Pi}}}
\newcommand{\vvx}{\vec{\bm{x}}}

\newcommand{\vvy}{\vec{\bm{y}}}
\newcommand{\vvp}{\vec{\bm{p}}}
\newcommand{\vvr}{\vec{\bm{r}}}

\newcommand{\mE}{\mathbb{E}}
\newcommand{\mG}{\mathbb{G}}
\newcommand{\mI}{\mathbb{I}}
\newcommand{\mP}{\mathbb{P}}
\newcommand{\mQ}{\mathbb{Q}}
\newcommand{\mH}{\mathbb{H}}

\newcommand{\bGa}{\bm{\alpha}}

\newcommand{\cA}{\mathcal{A}}

\newcommand{\cI}{\vec{\mathcal{I}}_{\text{KM}}}
\newcommand{\cO}{\mathcal{O}}
\newcommand{\bbI}{\mathbb{I}_{\text{KM}}}
\newcommand{\tr}{\top}

\newcommand{\mes}{\operatorname{mes}}

\newcommand{\md}{\mathrm{d}}

\newcommand{\mi}{\mathrm{i}}

\newcommand{\real}{\mathbb{R}}
\newcommand{\complex}{\mathbb{C}}

\renewcommand{\Re}{\text{Re}\,}

\newcommand{\Mb}[1]{\left[{#1}\right]}
\newcommand{\Mcb}[1]{\left\{{#1}\right\}}

\newtheorem{proposition}{Proposition}[section]

\newtheorem{remark}{Remark}[section]
\newtheorem{lemma}{Lemma}[section]
\DeclareMathOperator{\cond}{cond}

\newcommand{\diver}{\operatorname{div}}
\newcommand{\bcurl}{\operatorname{\bf curl}}

\begin{document}
\title{Imaging polarizable dipoles}
\author{Maxence Cassier \and Fernando Guevara Vasquez}
\address{Mathematics Department, University of Utah, Salt Lake City UT
84112}
\email{cassier@math.utah.edu}
\address{Mathematics Department, University of Utah, Salt Lake City UT
84112}
\email{fguevara@math.utah.edu}
\begin{abstract}
We present a method for imaging the polarization vector of an electric
dipole distribution in a homogeneous medium from measurements of the
electric field made at a passive array. We study an electromagnetic
version of Kirchhoff imaging and prove, in the Fraunhofer asymptotic
regime, that range and cross-range resolution estimates are identical to
those in acoustics. Our asymptotic analysis provides error estimates for
the cross-range dipole orientation reconstruction and shows that the
range component of the dipole orientation is lost in this regime. A
naive generalization of the Kirchhoff imaging function is afflicted by
oscillatory artifacts in range, that we characterize and correct.  We
also consider the active imaging problem which consists in imaging both
the position and polarizability tensors of small scatterers in the
medium using an array of collocated sources and receivers. As in the
passive array case, we provide resolution estimates that are consistent
with the acoustic case and give error estimates for the cross-range
entries of the polarizability tensor. Our theoretical results are
illustrated by numerical experiments.
\end{abstract}

\maketitle

\tableofcontents

\section{Introduction}

The behavior of small scatterers in a homogeneous medium and subject to an
electromagnetic field can be well described by a polarizability tensor
(see e.g.  \cite{Novotny:2012:PNO}), which characterizes the effective
response of the scatterer to the incident electric field. We present here
an imaging method that produces tensor valued images of the
polarizability tensor in a medium, by using measurements made at an
array of collocated sources and receivers. This is {\em active imaging}
because the array generates waves to probe the medium and could have
applications to radar \cite{Cheney:2009:FRI}. We also consider
the {\em passive imaging} problem which consists in imaging the location
and polarization vector of small electromagnetic sources within the
medium, from measurements of their emitted electrical field on a
(passive) array of receivers.

The imaging technique we use here is an analogue of Kirchhoff imaging
(see e.g. \cite{Blei:2013:MSI}) in electromagnetics. This a
well-understood technique in acoustics for both passive and active
imaging. To our knowledge Kirchhoff imaging has not been thoroughly
studied for electromagnetics. In acoustics, the Kirchhoff image
resolution in a plane parallel to the array (i.e. the cross-range) is
given by the Rayleigh criterion $\lambda L / a$, where $\lambda$ is the
wavelength, $L$ is the distance between the array and the imaging plane
and $a$ is array aperture (or diameter). The resolution in depth (in the
range direction, perpendicular to the array) is governed by the
frequency bandwidth $B$ of the measurements and is given by $c/B$, where
$c$ is wave velocity in the medium. For a mathematical derivation of
these resolution results see e.g. \cite{Borcea:2007:OWD}.

One could view the active imaging problem we consider as an inverse
medium problem. In the full aperture case (i.e. when the array
completely surrounds the medium) the inverse medium problem is known to
have a unique solution for the Maxwell equations
\cite{Ola:1993:IBV,Colton:2013:IAE} and linearized approaches (i.e.
based on the Born approximation) exist, see e.g.
\cite{Devaney:2012:MFI}. The passive problem for electromagnetics has
been studied in \cite{Marengo:1999:ISP}. To our knowledge, this work
presents the first error estimates for reconstruction of polarization
vectors (for the passive problem) and polarizability tensors (for the
active problem).

Another way of approaching the active imaging problem for small
inclusions is to study the behavior of the measurements as the size of
the inclusions is driven to zero. This was done for the conductivity
equation in \cite{Ced::ICI:1998} and was used to image small
inhomogeneities in a known conductive medium in e.g.
\cite{Ced::ICI:1998,Bruhl:DIT:2003}.  It was demonstrated that
information about the shape of the inclusions (namely an ellipsoid
approximation) can be obtained with this method \cite{Ced::ICI:1998}.
This asymptotic expansion has been carried out for other equations, e.g.
for the Maxwell equations \cite{Vog::AFP:2000}.  The asymptotic
expansion can actually be carried further, revealing the so-called
generalized polarizability tensors \cite{Ammari::2003} which encode more
information about the shape of the inclusions \cite{Ammari::2007}.  This
approach has been applied to acoustics \cite{Ammari::2013} and in the
context of electromagnetism to detect small dielectric inhomogeneities
characterized by a contrast in permittivity and permeability (see
\cite{Am::AFP:2001, Ammari::2003:ESD, Vog::AFP:2000} for the asymptotic
analysis and \cite{Ammari::2014:TCE} for the imaging application). We
emphasize that in our approach we are not after the geometric
properties of each inclusion separately. We only seek the first
polarization tensor (or matrix\footnote{Since all tensors we consider
are matrices, we interchangeably use ``tensor'' and ``matrix''.}) at an imaging point, whether there is a
scatterer or not, and study the resolution and accuracy of this image.

Our paper is organized as follows. The passive imaging problem is
studied in section~\ref{sec:passive}. The active case is considered in
section~\ref{sec:active}. We conclude with a brief summary and future in
section~\ref{sec:future}.

\section{The passive imaging problem}
\label{sec:passive}
It is convenient to use the dyadic Green function to
represent Maxwell equation solutions, so we recall it in
section~\ref{sec:dyadic:green}. Then in section~\ref{sec:math:form} we
define the passive imaging problem and the data we use to image. Our
analysis relies on the Fraunhofer asymptotic regime which is detailed
in section~\ref{sub.asympfraun}. The generalization of the Kirchhoff
imaging function to a vector valued image is explained in
section~\ref{sec:kirch}. The remaining sections include the analysis of
different aspects of this imaging function. The resolution estimate of the position in a
plane parallel to the array (cross-range direction) is done in
section~\ref{sec:crossrange:passive}. The extraction of the dipole
polarization vector data from the vector valued Kirchhoff image is done
in section~\ref{sec:crossrangep}. Section~\ref{sec:range:passive}
includes the resolution of the position in the range direction, i.e. the direction
perpendicular to the array, and section~\ref{sec:rangepol:passive}
studies the variation in range of the image and shows how to mitigate
oscillations in range of the reconstructed polarization vector. We
complement the theory with numerical experiments in
section~\ref{sec:numexp:passive2}.

\subsection{The dyadic Green function}
\label{sec:dyadic:green}
We consider an electromagnetic homogeneous medium characterized by its dielectric permittivity $\varepsilon$ and magnetic permeability $\mu$ which define its constant speed of propagation
$c=(\varepsilon \mu)^{-1/2}$. We denote by $\omega$ the angular
frequency and $k=\omega/c$ the wavenumber and by $\mG(\vvx,\vvy;k)$  the
matrix valued Green function associated with the time-harmonic Maxwell
equations, i.e. the solution to
\begin{equation}\label{eq.greentensor}
\bcurl_{\vvx}  \, \bcurl_{\vvx}  \, \mG(\vvx,\vvy;k)-  k^2
\mG(\vvx,\vvy;k)= \mI  \, \delta_{\vvy},
\end{equation}
where $\mI$ stands for the $3\times 3$ identity matrix, $ \delta_{\vvy}$
is the Dirac distribution at $\vvx=\vvy$ and $\bcurl_{\vvx}$ for the
curl operator of a $3\times3$ matrix, defined as the curl of its three
columns with respect to the $\vvx$ variable.  This so-called dyadic
Green function $\mG(\vvx,\vvy;k)$, is given by (see e.g.
\cite{Novotny:2012:PNO})
\begin{equation}\label{eq.defgreendyadic}
 \mG(\vvx,\vvy;k) = G(\vvx,\vvy;k) \Mb{ (1+m(kr)) \mI - (1+3m(kr))
 \frac{\vvr \vvr^\tr}{r^2} },
\end{equation}
where $\vvr = \vvx - \vvy$, $r = \|\vvr\|$,  $m(kr)\equiv (\mi kr -
1)/(kr)^2$ and $G(\vvx,\vvy;k)$ is the acoustic Green function in three
dimensions
\[
 G(\vvx,\vvy;k) = \frac{\exp[\mi k\|\vvx-\vvy\|]}{4\pi \|\vvx-\vvy\|}.
\]
The dyadic Green function $\mG(\vvx,\vvy;k)$ is diagonalizable in an
orthonormal basis because it is a normal
matrix (i.e. $\mG(\vvx,\vvy;k)$ commutes with $\mG(\vvx,\vvy;k)^* =
\overline{\mG}(\vvx,\vvy;k)$). Its eigenvalues are $\lambda_1(k,r)= -2
\,m(kr) \, G(\vvx,\vvy;k)$ (with multiplicity 1 and corresponding
eigenvector $\vvr/r$) and $\lambda_2(k,r)=\, (1+m(kr)) \,
G(\vvx,\vvy;k)$ (with multiplicity 2 and eigenspace being the orthogonal
of $\vvr$).
We note that the largest eigenvalue in magnitude of $\mG(\vvx,\vvy;k)$
is $\lambda_2(k,r)$ when $kr \gg 1$, and thus the Green function
becomes increasingly ill-conditioned  as $kr \to \infty$:
\[
\cond(\mG(\vvx,\vvy;k))=\frac{ \|\lambda_2(k,r) \|}{ \|\lambda_1(k,r)
\|}=\left\|\frac{m(kr)+1}{2m(kr)}\right\|=\frac{kr}{2}+\cO(1).
\]
\subsection{Mathematical formulation of the passive imaging problem}
\label{sec:math:form}

We denote by $a$ the characteristic size of an array $\cA$ of
collocated transmitters and receivers.  For example a square array in
the $z=0$ plane of side $a$ is given by $\cA = \{ \vx_r \in
\real^2~|~ \| \vx_r \|_\infty \leq a/2 \}$. Note that we drop the arrow
for two dimensional vectors, i.e. $\vx_r \in \real^2$ denotes the two first
components of $\vvx_r \in \real^3$.

We consider a family of $N$ radiating electric dipoles located at
$\vvy_1, \vvy_2,\ldots, \vy_N$ with respective polarization vectors (also called dipole moments)
$\vvp_1, \vvp_2,\ldots, \vvp_N$ (see figure \ref{fig.med}). 

\begin{figure}[!ht]
\begin{center}
 \includegraphics[width=0.65\textwidth]{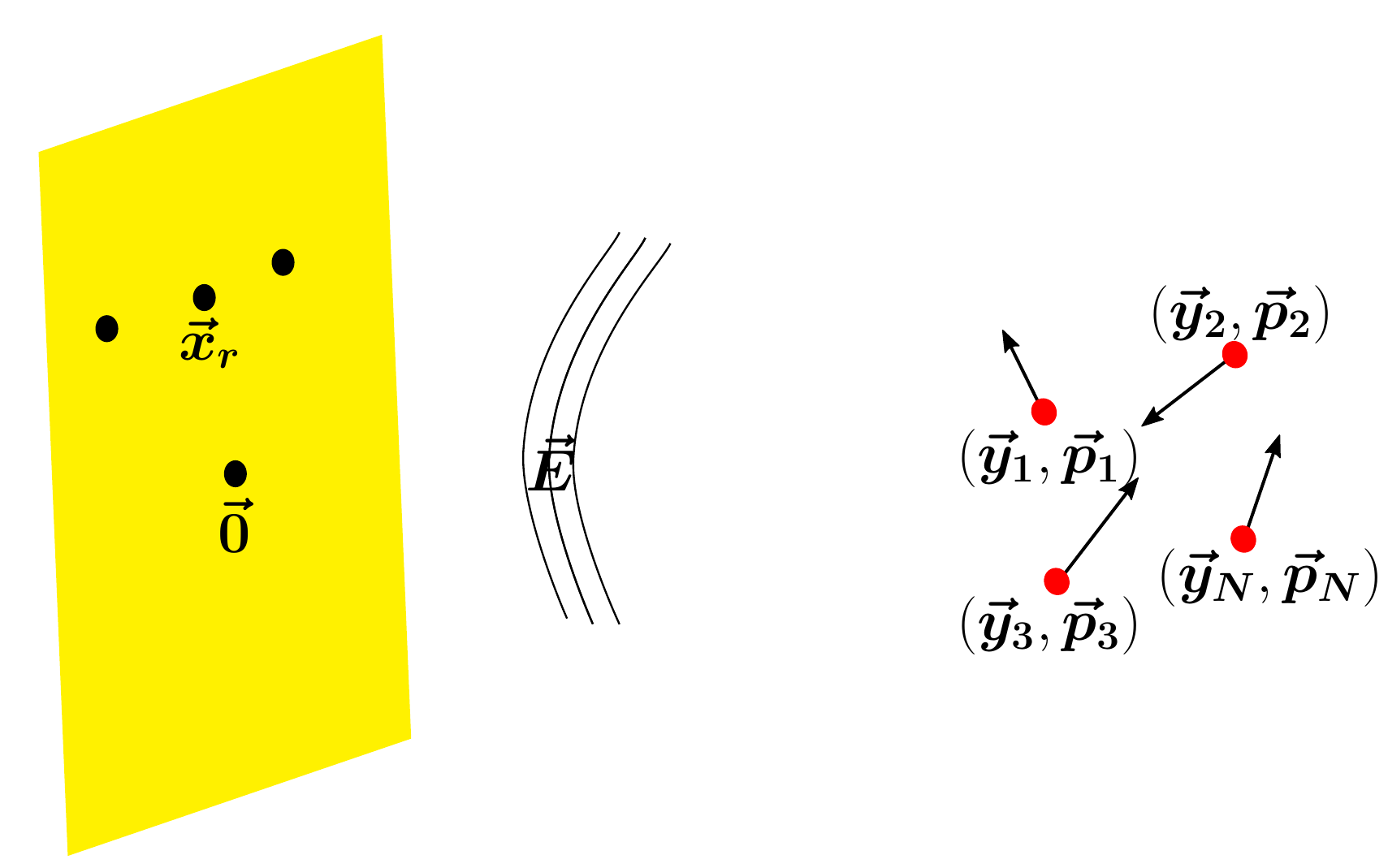}
\end{center}
 \caption{Description of the passive imaging problem.}
 \label{fig.med}
\end{figure}

The electric field (see \cite{Novotny:2012:PNO}) emitted by
this family of dipoles is a solution to the following time-harmonic equation on $\real^3$
\[
\bcurl  \, \bcurl  \, \vvE-  k^2 \vvE=  \mu  \, \omega^2 \sum_{n=1}^{N}   \vvp_n \, \delta_{\vvy_n}.
\]
and thus can be expressed by virtue of \eqref{eq.greentensor} in terms
of dyadic Green functions
\[
 \vvE(\vvx;k) = \mu \, \omega^2\sum_{j=1}^{N} \mG(\vvx,\vvy_j;k)\,  \vvp_j.
\]
The {\em passive imaging problem} consists in finding the positions
$\vvy_1, \ldots, \vvy_N$ and polarizations $\vvp_{1}, \ldots, \vvp_{N}$
of  the electric dipoles from the data
\begin{equation}\label{eq.data}
 \vvPi(\vx_r;k) =  \mu \, \omega^2\sum_{j=1}^{N} \mG((\vx_r,0),\vvy_j;k)\,  \vvp_j ,
\end{equation}
for $\vx_r \in \cA \subset \real^2$, in other words full measurements of
the electrical field on the array $\cA$.  This is the best case
scenario, as it assumes we can measure all three components of the
electric field at each point of the array.

\subsection{The dyadic Green function in the Fraunhofer asymptotic
regime}
\label{sub.asympfraun}

Let $\vvy=(\vy, L+\eta)$ and $\vvx_r=(\vx_r,0)$ be respectively an
imaging point and a receiver. The characteristic propagation distance
here is $L$. We assume that the scatterers lie in a known imaging window
(see figure \ref{fig.Frau}) of characteristic size $b$ in cross-range
and $h$ in range, i.e. $\| \vy \|= \cO(b) ~\text{and}~  | \eta |=\cO
(h)$. We assume we are in the Fraunhofer asymptotic regime, i.e. that
the following scalings hold (see e.g.
\cite{Borcea:2008:EII,Born:1959:POE}).
\begin{itemize}
\item  $kL \gg 1$, (high frequency or large propagation distance)
\item
Fresnel number $\Theta_a \equiv \displaystyle \frac{k a^2}{L}\ll k L$,
i.e. small aperture: $a\ll L$,
\item Fresnel number $\Theta_b \equiv \displaystyle \frac{k b^2}{L}\ll kL$,
i.e. small imaging window in cross-range: $b\ll L,$
\item Fresnel number $\Theta_h \equiv \displaystyle \frac{k h^2}{L}\ll kL$, i.e.  small
imaging window in range $h\ll L$.
\end{itemize}
Moreover we assume that
\[
 \Theta_{b} \ll 1,  ~ 
  ~ 1 \ll \Theta_{a} \ll
 \frac{L^2}{a^2 }, \mbox{ and }   kh=\cO( 1 ),
\]
which amount to assume that the imaging window is small compared to the
array aperture, i.e. $b\ll a$.

\begin{figure}[!ht]
\begin{center}
 \includegraphics[width=0.65\textwidth]{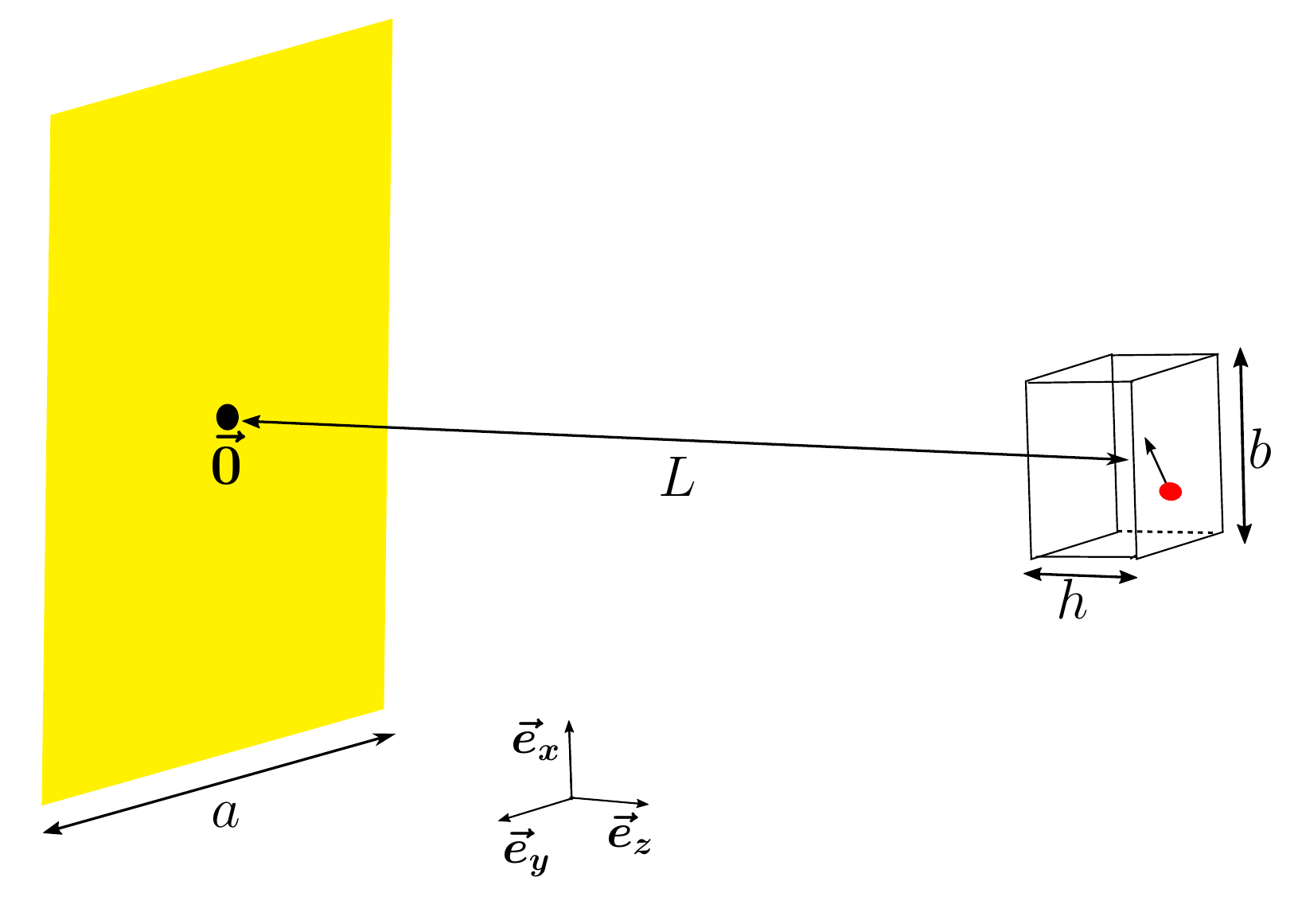}
\end{center}
 \caption{Characteristic lengths involved in the Fraunhofer asymptotic
 regime.}
 \label{fig.Frau}
\end{figure}

We first approximate the acoustic Green function $G(\vvx_r,\vvy;k)$ as
in \cite{Borcea:2008:EII} by expanding the distance $\|\vvx_r-\vvy\|$:
\begin{eqnarray}\label{eq.ampl}
\displaystyle \|\vvx_r-\vvy\|&=& L  \left(\frac{\|\vx_r-\vy\|^2}{L^2}+\left(1+\frac{\eta}{L}\right)^2\right)^{\frac{1}{2}}\nonumber \\
&=& L \left(1+\cO\left(\frac{a^2}{L^2}\right)\right)\nonumber  \\
&=& L  \left(1+o(1)\right),
\end{eqnarray}
since $kh=\cO( 1 )$ implies that
$\cO(h/L)=\cO\left(1/(kL)\right)=\cO\left(a^2/(\Theta_a
L^2)\right)=o\left(a^2/L^2\right)$.

Expanding the phase leads to
\begin{eqnarray}\label{eq.phase}
\displaystyle k \|\vvx_r-\vvy\|&=&k L \left(1+
\frac{\|\vx_r\|^2}{L^2}+\frac{2\vx_r\cdot\vy}{L^2}+\frac{2\eta}{L}
+\cO\left(\frac{b^2}{L^2}\right)+\cO\left(\frac{h^2}{L^2}\right)\right)^\frac{1}{2}\nonumber \\
&=& k L+\frac{k\|\vx_r\|^2}{2
L}+\frac{k\vx_r\cdot\vy}{L}+k\eta+\cO\left(\Theta_b\right)+\cO\left(\frac{a^2\Theta_a}{L^2}\right) \nonumber\\
&=& k L+\frac{k\|\vx_r\|^2}{2 L}+\frac{k\vx_r\cdot\vy}{L}+k\eta +o(1),
\end{eqnarray}
since $kh=\cO (1)$ implies that $\displaystyle \cO(h\Theta_a/ L )=\cO(a^2/L^2)=o(a^2\Theta_a/L^2)$ and $\cO(\Theta_h)=\cO(h/L)=o(a^2/L^2)$. 

Using (\ref{eq.ampl}) and (\ref{eq.phase}), we get the following
expansion for the acoustic Green function
\begin{eqnarray}\label{eq.greenpara}
G(\vvx_r, \vvy;k)&=&\widetilde{G}(\vvx_r, \vvy;k) \left( 1+\cO\left(\frac{a^2\Theta_a}{L^2}\right)+\cO\left(\Theta_b\right)\right) \nonumber\\
&=&  \widetilde{G}(\vvx_r, \vvy;k)  ~ (1+o(1)),
\end{eqnarray}
where the  Fraunhofer or paraxial approximation of the acoustic Green
function is
\[
\widetilde{G}(\vvx_r, \vvy;k) \equiv \frac{1}{4\pi L}  ~  \exp\Mb{\mi ( k
L+\frac{k\|\vx_r\|^2}{2 L}+\frac{k\vx_r\cdot\vy}{L}+k\eta )}.
\]
Using that $m(kr)= \cO(1/kL)$, the dyadic Green function
\eqref{eq.defgreendyadic} becomes
\begin{equation}\label{eq.dyadic}
\mG(\vvx_r, \vvy;k)=G(\vvx_r, \vvy;k)  \left(\mP(\vvx_r,\vvy)+
\cO\left(\frac{1}{kL}\right) \right),
\end{equation}
where we used the orthogonal projector $\mP(\vvx_r,\vvy)$ defined by
\[
\mP(\vvx_r,\vvy)=\mI-\frac{(\vvx_r-\vvy)(\vvx_r-\vvy)^{\tr}}{\|\vvx_r-\vvy\|^2}.
\]
Finally, combining the relations (\ref{eq.greenpara}) and (\ref{eq.dyadic}), we obtain the asymptotic of the dyadic Green function in the Fraunhofer regime:
\begin{eqnarray}\label{eq.dyadicbis}
\mG(\vvx_r, \vvy;k)&=&  \widetilde{G}(\vvx_r, \vvy;k)\left( 1+\cO\left(\frac{a^2\Theta_a}{L^2}\right)+\cO\left(\Theta_b\right)\right)  \left(\mP(\vvx_r,\vvy)+ \cO\left(\frac{1}{kL}\right) \right) \nonumber\\
&=& \widetilde{G}(\vvx_r, \vvy;k) \left(\mP(\vvx_r,\vvy)+\cO\left(\frac{a^2\Theta_a}{L^2}\right)+ \cO\left(\Theta_b\right)\right) \nonumber\\
&=& \widetilde{G}(\vvx_r, \vvy;k) \left(\mP(\vvx_r,\vvy)+o(1)\right).
\end{eqnarray}

\subsection{The Kirchhoff imaging function in electromagnetics}
\label{sec:kirch}
In acoustics, the Kirchhoff imaging function is, up to conjugation,
time reversing the data recorded at the array and propagating it to the
medium \cite{Blei:2013:MSI,Borcea:2007:OWD,Pra:94:ETR}. For the Maxwell
equations we consider the analogous imaging function
\begin{equation}
  \cI(\vvy;k) = (\mu \omega^2)^{-1} \int_{\cA} \md\vx_r\,
  \overline{\mG((\vx_r,0),\vvy;k)}~\vvPi(\vx_r;k).
\end{equation}
In contrast with the acoustic case, the image we obtain is a vector
field. Also the factor $(\mu \omega^2)^{-1}$ is there to offset a
similar factor in the data \eqref{eq.data} and simplifies our asymptotic
analysis. The imaging function with data \eqref{eq.data} is then
\begin{equation}\label{eq.image}
 \cI(\vvy;k)=   \sum_{j=1}^N \mH(\vvy,\vvy_j;k) \vvp_{j},
\end{equation}
where  the $3\times 3$ matrix $\mH(\vvy',\vvy;k)$ is 
\begin{eqnarray}\label{eq.defmatA}
\mH(\vvy,\vvy';k)=  \int_{\cA} \md\vx_r \,
 \overline{\mG(\vvx_r,\vvy;k)}  ~ \mG(\vvx_r,\vvy';k).
\end{eqnarray}
The image at $\vvy$ of a dipole located at $\vvy_*$  with polarization
$\vp_*$ is thus $\mH(\vvy,\vvy_*; k) \vp_*$. The matrix
$\mH(\vvy,\vvy_*; k)$ is thus, in some sense, a matrix valued point
spread function for the Kirchhoff imaging function.

To analyze the Fraunhofer asymptotics of 
\eqref{eq.image}, we introduce for
$\vvy=(\vy,L+\eta)$ and $\vvy'= (\vy',L+\eta') \in
\real^3$ the $3\times 3$ matrix $\widetilde{\mH}(\vvy,\vvy';k)$ defined by
\begin{eqnarray}\label{eq.matIcross}
\widetilde{\mH}(\vvy,\vvy';k)&=&  \int_{\cA}\md\vx_r\, \overline{\widetilde{G}(\vvx_r, \vvy;k)}
\widetilde{G}(\vvx_r, \vvy';k)\mP(\vvx_r,\vvy)  \mP(\vvx_r,\vvy')  \nonumber  \\
&=&\frac{\exp[\mi k(\eta'-\eta)]}{(4\pi L)^2}\int_{\cA} \md\vx_r\,
\exp\Mb{\mi k \left(\frac{\vx_r\cdot(\vy'-\vy)}{L}\right)}
\mP(\vvx_r,\vvy)  \mP(\vvx_r,\vvy'). \nonumber\\
\end{eqnarray}

\begin{proposition}\label{prop.kirchfraun}
The Kirchhoff imaging function (\ref{eq.image}) associated with data
\eqref{eq.data} is, in the Fraunhofer regime,
\begin{equation}\label{eq.decompimage}
\cI(\vvy;k)=   \sum_{j=1}^N\left[
\widetilde{\mH}(\vvy,\vvy_j;k)  + \cO\left(\frac{a^4\Theta_a}{L^4}\right)+ \cO\left(\frac{a^2\Theta_b}{L^2} \right) \right] \vvp_{j}.
\end{equation}
\end{proposition}

\begin{proof}
Using the asymptotic of the dyadic Green function \eqref{eq.dyadicbis}, one can rewrite the matrices $\mH(\vvy,\vvy_j;k)$ in the  imaging function \eqref{eq.image} as
\begin{equation}\label{eq.asymtAg}
\mH(\vvy,\vvy_j;k)=\widetilde{\mH}(\vvy,\vvy_j;k)  + \mE(\vvy,\vvy_j;k) 
\end{equation}
with a matrix $\mE(\vvy,\vvy_j;k)$ whose asymptotic is given by:
\begin{eqnarray*}
\mE(\vvy,\vvy_j;k)&=&\frac{1}{(4\pi L)^2}\int_{\cA} \md\vx_r
\left(\cO\left(\frac{a^2\Theta_a}{L^2}\right)+
\cO\left(\Theta_b\right)\right) \\
&=&  \cO\left(\frac{a^4\Theta_a}{L^4}\right)+ \cO\left(\frac{a^2\Theta_b}{L^2}\right).
\end{eqnarray*}
This last asymptotic follows from two facts. First the matrix norm
of the orthogonal projectors
$\mP(\vvx_r,\vvy)$ and  $\mP(\vvx_r,\vvy')$ is uniformly bounded in
$\cA$ (for any matrix norm choice). Second, the area of the array
$\cA$ is $\cO(a^2)$.
\end{proof}

\subsection{Cross-range estimation of positions}
\label{sec:crossrange:passive}

We start by analyzing the spatial resolution of the Kirchhoff imaging
function \eqref{eq.image}. This is done by  looking at the decay properties of the point
spread function $\mH(\vvy,\vvy_*;k)$ for $\vvy$ away from a fixed dipole location
$\vvy_*$. This decay comes from the (approximate) orthogonality of the
acoustic Green functions $G(\vvx_r,\vvy_*;k)$ and $G(\vvx_r,\vvy;k)$ as
functions of $\vx_r$ in $L^2(\cA)$ and when $\vvy$ and $\vvy_*$ are
``well separated'', see \cite{Borcea:2007:OWD,Cass:2014:SFU,Pra:94:ETR}. 

Since the imaging function is linear in the data we consider (without
loss of generality) the case of a single obstacle located at
$\vvy_* = (\vy_*,L + \eta_*)$. This is valid if we assume that the
dipole polarizations have the same order of magnitude (i.e.
$\|\vvp_i\|=\cO(1), \, i=1,\ldots, N$) and this implicitly assumed in the
following. The resolution analysis in this section is done in the
cross-range plane $z=L+\eta_*$ passing through the single obstacle.
The range analysis is done later in section~\ref{sec:range:passive}. The
main result of this section is summarized by the following proposition.

\begin{proposition}[Imaging
function decrease in cross-range] \label{eq.propcrossrang}
The Kirchhoff imaging function \eqref{eq.image} of a dipole located at
$\vvy_*=(\vy_*, L+\eta_*)$ and evaluated at  $\vvy=(\vy,
L+\eta_{*})$ satisfies for $\vvy=(\vy,L+\eta_*)$ with $\vy\neq \vy_*$
\begin{equation}\label{eq.asymp1}
 \|\cI(\vvy;k)\| =   \frac{ a^2}{L^2}  \, \left( \cO\Big(  \frac{ L}{
 a \, k   \left\|\vy-\vy_* \right\|}\Big) +o(1) \right),
\end{equation}
where the $o(1)$ is explicitly given by $\cO(a^2\Theta_a/L^2)+\cO(\Theta_b)$.
When the shape of the array $\cA$ is a disk of radius $a$, one has 
\begin{eqnarray}\label{eq.asymp2}
\cI(\vvy_*;k)=\frac{  a^2 }{16 \pi L^2} \begin{pmatrix} 1 & 0 & 0\\ 0 & 1 & 0\\  0 &0 &0  \end{pmatrix}  \vvp_*+o\left(\frac{a^2}{L^2}\right).
\end{eqnarray}
\end{proposition}

Concretely, the asymptotics \eqref{eq.asymp1} and \eqref{eq.asymp2} in
proposition \ref{eq.propcrossrang} means that the image decays as we
move away from the dipole location $\vvy_*$. The image is asymptotically
zero compared to its value at $\vvy_*$, when $\left\|\vy-\vy_* \right\|$
is large with respect to $L/ak$. Hence the size of the
focal spot  in the cross-range is
given by the Rayleigh resolution $L/(ak)$, as is the case in acoustics
(see e.g. \cite{Blei:2013:MSI,Borcea:2007:OWD}).

\begin{remark}
In the particular case where the dipole polarization is aligned with the
$z-$axis, the image at the dipole location is
$\cI(\vvy_*;k)=o(a^2/L^2)$ (see \eqref{eq.asymp2}). We estimated  in
\eqref{eq.asymp1} that the imaging function is small compared to
$a^2/L^2$ for imaging points far way from the dipole location. Therefore
it is not clear from this analysis if the image reveals the position
of such a dipole. This case is considered later in
section~\ref{sec:range:passive}, when we consider multi-frequency
images.
\end{remark}

The proof of proposition \ref{eq.propcrossrang} is a direct consequence
of two results. The first one shows the decay of property of the image
function using a stationary phase argument
(lemma~\ref{lem.crossrange2}). The second one, explicitly calculates
the value of the image at the dipole location using a circular array
(lemma \ref{lem.crossrange1}).

\begin{lemma}\label{lem.crossrange2}
For $\vvy=(\vy,L+\eta)$, $\vvy'=(\vy',L+\eta')$ with  $\vvy \neq \vvy'$, the matrix $\widetilde{\mH}(\vvy,\vvy';k)$ defined by
\eqref{eq.matIcross} satisfies 
\begin{equation}\label{eq.decreascrossrange1}
\left\|\widetilde{\mH}(\vvy,\vvy';k)\right\|
 =  \frac{a^2}{L^2}\, \cO\left(   \frac{ L}{  a \, k \left\|\vy-\vy'
 \right\| } \right),
\end{equation}
where the constant in the $\cO$ notation depends only on the shape of
the array $\cA$ and the depth of the imaging window.
\end{lemma}

\begin{proof}
To avoid writing a phase term many times, we introduce the matrix 
\begin{equation}\label{eq.tildeIj}
\widetilde{\mH}'(\vvy,\vvy';k)=\exp[-\mi
k(\eta-\eta')]\mH(\vvy,\vvy';k),
\end{equation}
which satisfies $\| \widetilde{\mH}'(\vvy,\vvy';k) \|=\|
\mH(\vvy,\vvy';k) \|$. 
To prove (\ref{eq.decreascrossrange1}), it is more convenient to use the
rescaled array $\widetilde{\cA} = a^{-1} \cA$. By the change of
variable $\widetilde{\vx}_r= \vx_r/a$, we can rewrite the matrix $
\widetilde{\mH}'(\vvy,\vvy';k) $ defined by (\ref{eq.matIcross}) and (\ref{eq.tildeIj}) as
\begin{equation}\label{eq.matIjcross2}
 \widetilde{\mH}'(\vvy,\vvy';k)=\frac{a^2}{(4\pi L)^2}
 \int_{\widetilde{\cA}} \md \widetilde{\vx}_r\, \exp\Mb{\frac{ \mi k a}{L}  \,
 \widetilde{\vx}_r\cdot(\vy'-\vy)} \mP(a \widetilde{\vvx}_r,\vvy)  \mP(a
 \widetilde{\vvx}_r,\vvy'),
\end{equation}
where $\widetilde{\vvx}_r=(\widetilde{\vx}_r,0)$.
Then, we follow the method  proposed in \cite{Wong:2014:AAI}, based on an integration by parts,
to study the asymptotic of multivariate oscillating integrals. We denote by $f$ the scalar function:  
 $$ f(\widetilde{\vx}_r)=\frac{k a}{L}\, [\widetilde{\vx}_r\cdot(\vy'-\vy)],$$
and by $ \mathbb{M}(\cdot,\vvy,\vvy')$ and $
\mathbb{K}(\cdot,\vvy,\vvy')$ the $3\times3$ matrix-valued functions
defined componentwise by
$$
\mathbb{M}_{p,q}(\widetilde{\vvx}_r,\vvy,\vvy')=\big(\mP(a
\widetilde{\vvx}_r,\vvy)  \mP(a \widetilde{\vvx}_r,\vvy' )\big)_{p,q} \
\mbox{ and } \
\mathbb{K}_{p,q}(\widetilde{\vvx}_r,\vvy,\vvy')=\|\nabla_{
\widetilde{x}_r} \mathbb{M}_{p,q}(\widetilde{\vvx}_r,\vvy,\vvy')\|,
$$
for $p,q\in\{1,2,3\}$ on the rescaled array $\widetilde{\cA}$. 
Using the identity:
$$
\diver_{\widetilde{\vx}_r}\left(
\frac{\nabla{f}_{\widetilde{\vx}_r}}{\|\nabla{f}_{\widetilde{\vx}_r}\|^2}
\mathbb{M}_{p,q} e^{\mi f}\right)
=
e^{\mi f}\diver_{\widetilde{\vx}_r}\left(
\frac{\nabla{f}_{\widetilde{\vx}_r}}{\|\nabla{f}_{\widetilde{\vx}_r}\|^2}
\mathbb{M}_{p,q}\right)
+
\mi \, e^{\mi f}  \mathbb{M}_{p,q},
$$
and the divergence theorem, we can rewrite the entries of $\widetilde{\mH}'(\vvy,\vvy';k)$ as
$$
\widetilde{\mH}'_{p,q}(\vvy,\vvy';k)
=
\frac{-\mi a^2}{(4\pi L)^2\|\nabla{f}_{\widetilde{\vx}_r}\|}
\left[ \int_{\partial \widetilde{\cA} } \md \gamma \,
\frac{\nabla{f}_{\widetilde{\vx}_r}}{\|\nabla{f}_{\widetilde{\vx}_r}\| }\cdot \vn \,\mathbb{M}_{p,q} e^{\mi f} 
- \int_{\widetilde{\cA} } \md \widetilde{\vx}_r \,e^{\mi f}
\diver_{\widetilde{\vx}_r}\left(
\frac{\nabla{f}_{\widetilde{\vx}_r}}{\|\nabla{f}_{\widetilde{\vx}_r}\|}\mathbb{M}_{p,q}\right)
\right].
$$
Replacing $\nabla{f}_{\widetilde{\vx}_r}$ by its constant value $(ka/L)
(\vy'- \vy)$, expressing the divergence of the second integral and  then
using the Cauchy-Schwarz inequality we obtain
\begin{equation}\label{eq.decreascrossrange2}
\| \widetilde{\mH}'(\vvy,\vvy';k)_{pq}\|\leq \frac{(4\pi)^{-2}  a }{ k
\,L \left\|\vy'-\vy \right\| }\left( \int_{\partial \widetilde{\cA} }
\md \gamma \,
|\mathbb{M}_{p,q}(\widetilde{\vvx}_r,\vvy,\vvy')|+
\int_{\widetilde{\cA}} \md \widetilde{\vx}_r \mathbb{K}_{pq}(\widetilde{\vvx}_r,\vvy,\vvy' )\right).
\end{equation}
As the entries of an orthogonal projection are dominated by $1$, we have
$$
 |\mathbb{M}_{p,q}(\widetilde{\vvx}_r,\vvy,\vvy')| \leq \sum_{l=1}^3  |\mathbb{P}_{p,l}(a\widetilde{\vvx}_r,\vvy)| \, |\mathbb{P}_{l,q}(a\widetilde{\vvx}_r,\vvy') | \leq 3
$$
which leads immediately to \begin{equation}\label{eq.domborderterm}
\int_{\partial \widetilde{\cA} } \md \gamma \,
|\mathbb{M}_{p,q}(\widetilde{\vx}_r,\vy,\vy')| \leq 3  \mes \partial
\widetilde{\cA},
\end{equation}
where $\mes \partial \widetilde{\cA}$ is the perimeter of
$\widetilde{\cA}$.  A direct calculation gives 
\begin{equation*}\label{eq.domgrad}
|\nabla \mathbb{P}_{p,q} ( \,a \widetilde{\vvx}_r,\vvy)|\leq 2 a
|L+\eta|^{-1},
\end{equation*}
for all points $\vvy=(\vy, L+\eta) \in \mathbb{R}^3$ with $L+\eta \neq 0$.
Applying this inequality and the fact that the entries of an orthogonal
projector are bounded by $1$ we get
\begin{eqnarray*}
 \mathbb{K}_{p,q}(\widetilde{\vvx}_r,\vvy,\vvy')& \leq& \sum_{k=1}^3  |\nabla \mathbb{P}_{p,k}(a \widetilde{\vvx}_r,\vvy)| \, |\mathbb{P}_{k,q}(\widetilde{\vvx}_r,\vvy') | + | \mathbb{P}_{p,k}(a \widetilde{\vvx}_r,\vvy)| \, |\nabla \mathbb{P}_{k,q}(\widetilde{\vvx}_r,\vvy') | \\
& \leq & 12\, a\, \max( |L+\eta|^{-1}, |L+\eta'|^{-1} ).
\end{eqnarray*}
This last inequality leads immediately to
 \begin{equation}\label{eq.domgradint}
\int_{\widetilde{\cA}} \md \widetilde{\vx}_r
\mathbb{K}_{p,q}(\widetilde{\vvx}_r,\vvy,\vvy' )\leq 12\, a\,
\mes(\widetilde{\cA}) \max( |L+\eta|^{-1}, |L+\eta'|^{-1} )=\cO\left( \frac{a}{L} \right)=o(1),
\end{equation}
where $\mes(\widetilde{\cA})$ is the area of $\widetilde{\cA}$.
Finally, the asymptotic \eqref{eq.decreascrossrange1} for the matrix
$\mH(\vvy,\vvy';k)$ follows from the relations
\eqref{eq.decreascrossrange2},  \eqref{eq.domborderterm},
\eqref{eq.domgradint} and using the $\sup$ norm for the entries of a
matrix. By equivalence of norms, it also holds for other matrix norms.
\end{proof}

We now continue with lemma~\ref{lem.crossrange1} and study the
asymptotics of $\widetilde{\mH}(\vvy_*,\vvy_*;k)$. This is a special
case because the integral \eqref{eq.matIcross} is not oscillatory, is
independent of $k$ and simplifies to
$$\widetilde{\mH}(\vvy_*,\vvy_*;k)=\frac{1}{(4\pi L)^2}  \int_{\cA}
\md\vx_r\, \mP(\vvx_r,\vvy_*).$$ We compute explicitly this integral when
the array is a disk, but this could be generalized to other simple
geometries.  
\begin{lemma} \label{lem.crossrange1} When the
array $\cA$ is a disk of radius $a$, we have the asymptotic
\begin{equation}\label{eq.asymp}
\widetilde{\mH}(\vvy_*,\vvy_*;k)= \frac{  a^2 }{16 \pi L^2} \begin{pmatrix} 1 & 0 & 0\\ 0 & 1 & 0\\  0 &0 &0  \end{pmatrix} + \cO\left(\frac{a^4}{L^4}\right)+ \cO\left(\frac{a^2 b}{L^3} \right).
\end{equation}
\end{lemma}
For the proof of lemma~\ref{lem.crossrange1}, see appendix
\ref{app:cross:passive}.

\begin{proof}[Proof of proposition~\ref{eq.propcrossrang}]
The proof is a direct consequence of proposition \ref{prop.kirchfraun},
lemmas \ref{lem.crossrange1} and \ref{lem.crossrange2} and the fact that
the polarization vectors $\vvp_i$ are assumed to have order one length.
\end{proof}

\subsection{Cross-range estimation of the polarization vectors} 
\label{sec:crossrangep}
We now extract from the Kirchhoff imaging function for $N$ dipoles
\eqref{eq.image} information about the polarization vectors $\vvp_1,
\ldots, \vvp_N$. Assuming the dipole positions $\vvy_1,\ldots, \vvy_N$
are known, the polarization vectors $\vvp_i$ must satisfy
\begin{equation}\label{eq.identityvvpi}
\mH(\vvy_i,\vvy_i;k) \vvp_i=\cI(\vvy_i;k)-\sum_{j\neq
i}^{N}\mH(\vvy_i,\vvy_j;k)\vvp_j,~i=1,\ldots,N.
\end{equation}
If the dipoles are distant enough from each other, lemma
\ref{lem.crossrange2} guarantees that the coupling terms $\sum_{j\neq
i}^{N}\mH(\vvy_i,\vvy_j;k)\vvp_j$ remain small. Thus we estimate the
$\vvp_i$ by solving the linear system,
\begin{equation}\label{eq.systptot}
\mH(\vvy_i,\vvy_i;k) \vvp=\cI(\vvy_i;k),~i=1,\ldots,N.
\end{equation}
Unfortunately, the systems \eqref{eq.systptot} are ill-conditioned. Indeed, for
a circular array of radius $a$, the Fraunhofer asymptotic of the matrix
$\mH(\vvy_i,\vvy_i;k)$ (proposition \ref{prop.kirchfraun} and lemma
\ref{lem.crossrange1}) gives
\begin{eqnarray}\label{eq.asymptA}
\mH( \vvy_i,\vvy_i;k)&=& \widetilde{\mH}(\vvy_i,\vvy_i;k)+  \cO\left(\frac{a^4\Theta_a}{L^4}\right)+ \cO\left(\frac{a^2\Theta_b}{L^2}\right)\\
&=&  \frac{ a^2 }{16 \pi L^2} \begin{pmatrix} 1 & 0 & 0\\ 0 & 1 & 0\\  0 &0 &0  \end{pmatrix} + \cO\left(\frac{a^2 b}{L^3}\right)+ \cO\left(\frac{a^4\Theta_a}{L^4}\right)+ \cO\left(\frac{a^2\Theta_b}{L^2}\right). \nonumber\,\,
\end{eqnarray}
Thus the matrix $\mH(\vvy_i,\vvy_i;k)$ is asymptotically a singular
matrix and one cannot expect retrieving the $z-$component $p_{i,z}$ of
$\vvp_i=(\vp_i,p_{i,z})$. The same asymptotic \eqref{eq.asymptA} reveals that the
$2\times 2$ block of $\mH(\vvy_i,\vvy_i;k)$ (corresponding to the two
components $\vp_i$ of $\vvp_i$) is invertible, with condition number
close to $1$. 

We now derive an error estimate on $\vp_i$. To this end we denote by 
$\mathcal{I}_{\text{KM}}(\vvy;k)$ the first two components of the imaging function
$\cI(\vvy;k)$ and rewrite
\begin{equation}\label{eq.matrixbloc}
\mH(\vvy',\vvy;k)=\begin{pmatrix} \mH_{1:2,1:2}(\vvy',\vvy;k)&
\mH_{1:2,3}(\vvy',\vvy;k)\\ \mH_{3,1:2}(\vvy',\vvy;k) &
\mH_{3,3}(\vvy',\vvy;k)\end{pmatrix} \in \complex^{3\times 3},
\end{equation}
where $\mH_{1:2,1:2}(\vvy',\vvy;k)$, is the $2\times 2$ first principal
submatrix of $\mH(\vvy',\vvy;k)$, etc$\ldots$

\begin{proposition}\label{prop.polacrossactiv}
Let $\cA$ be a circular array $\cA$ of radius $a$. In the Fraunhofer asymptotic regime, the $2\times 2$ system
\begin{equation}\label{eq.systpbloc}
\mH_{1:2,1:2}(\vvy,\vvy;k) \,\vp=\mathcal{I}(\vvy;k)
\end{equation}
is invertible. The solution $\vp$ (which depends on the imaging point
$\vvy$) approximates the true cross-range polarization vector $\vp_i$
according to the following estimates
\begin{itemize}
\item If the imaging point coincides with a dipole, i.e. $\vvy=\vvy_i$,
\begin{equation}\label{eq.crossrangeplarecover}
\|\vp-\vp_i\| = \cO \left(  \frac{L} {a\, k \min \limits_{j\neq
i}\|\vy_i-\vy_j\|} \right)+  \cO\left(\frac{ b}{L}\right)+ \cO\left(\frac{a^2\Theta_a}{L^2}\right)+ \cO\left(\Theta_b \right).
\end{equation}
\item If the imaging point does not coincide with a dipole in range,
i.e. if $\vy\neq \vy_j$ for all $j=1,\ldots,N$,
\begin{equation}\label{eq.drecraspola}
\|\vp\| = \cO \left(  \frac{L} {a\, k \min
\limits_{j=1,\ldots,N}\|\vy-\vy_j\|} \right)+ \cO\left(\frac{a^2\Theta_a}{L^2}\right)+ \cO\left(\Theta_b \right).
\end{equation}
\end{itemize}
\end{proposition}

\begin{proof}
The asymptotic expansion \eqref{eq.asymptA} of $\mH(\vvy,\vvy;k)$ yields
immediately that for the induced matrix 2-norm, 
$$
\mH_{1:2,1:2}(\vvy,\vvy;k)= \frac{ a^2 }{16 \pi L^2} \mI+o\left(\frac{a^2}{L^2}\right),
$$
Hence, as soon as
\begin{equation}
16   \frac{ \pi^2 L^2 }{   a^2} \, \|  \mH_{1:2,1:2}(\vvy,\vvy;k)- \mI \|<1,
 \label{eq:neumann}
\end{equation}
we have that the matrix $\mH_{1:2,1:2}(\vvy,\vvy;k)$ is invertible.
This is the case in the Fraunhofer regime because the left hand side of
\eqref{eq:neumann} is $o(1)$ and one gets that $ \| \mH_{1:2,1:2}(\vvy,\vvy;k)^{-1} \|$ is of order $L^2/a^2$.
Hence, in this regime, the system \eqref{eq.systpbloc} admits a unique solution $\vp\in \complex^2$.
When $\vvy=\vvy_i$, 
we use the block decomposition \eqref{eq.matrixbloc} of
$\mH(\vvy_i,\vvy_i;k)$ and the identity \eqref{eq.identityvvpi} to get
\begin{equation}\label{eq.idendityvpi}
\mH_{1:2,1:2}(\vvy_i,\vvy_i;k) \vp_i=\mathcal{I}_{\text{KM}}(\vvy_i;k) -
\mH_{1:2,3}(\vvy_i,\vvy_i;k) \, p_{i,z} - \Big[ \sum_{j\neq
i}^{N}\mH(\vvy_i,\vvy_j;k)\vvp_j \Big]_{1:2}.
\end{equation}
Using relations \eqref{eq.idendityvpi},
\eqref{eq.matrixbloc} and  \eqref{eq.systpbloc} (which define the systems
satisfied by $\vp_i$ and $\vp$) we get that:
\[
\|\vp-\vp_i\| \leq  \| \mH_{1:2,1:2}(\vvy_i,\vvy_i;k)^{-1} \|  \, \left[       
\Big\| \big[
\sum_{j\neq i}^{N}\mH(\vvy_i,\vvy_j;k)\vvp_j \big]_{1:2} \Big\| + \|
\mH_{1:2,3}(\vvy_i,\vvy_i;k)\, p_{i,z}\|\right].
\]
Finally, we apply lemma \ref{lem.crossrange2} and the asymptotic
formulas \eqref{eq.asymtAg} and \eqref{eq.asymptA} to dominate the right
side of the last inequality:
\[
\begin{aligned}
\|\vp-\vp_i\|& = C\,  \frac{ L^2 }{a^2  } \left[ \frac{  a^2}{L^2}  \cO
\left(  \frac{L} {a\, k \min \limits_{i\neq j}\|\vy_i-\vy_j\|}\right)
+\cO\left(\frac{ a^2 b}{L^3}\right)+ \cO\left(\frac{a^4\Theta_a}{L^4}\right)+ \cO\left(\frac{a^2\Theta_b}{L^2} \right) \right]\\
& = \cO \left(  \frac{L} {a\, k \min \limits_{i\neq
j}\|\vy_i-\vy_j\|}\right) + \cO\left(\frac{ b}{L}\right)+
\cO\left(\frac{a^2\Theta_a}{L^2}\right)+ \cO\left(\Theta_b \right).
\end{aligned}
\]
We now move to the case $\vvy\neq \vvy_i$. From \eqref{eq.systpbloc}
we immediately get
$$
\|\vp\|\leq  \| \mH_{1:2,1:2}(\vvy,\vvy;k)^{-1} \| \|
\mathcal{I}_{\text{KM}}(\vvy;k) \|.
$$
Then, using proposition \ref{eq.propcrossrang} which expresses the
decay of the Fraunhofer image $\| \mathcal{I}_{\text{KM}}(\vvy;k) \|$ in the case of
one dipole and the linearity of $\mathcal{I}_{\text{KM}}(\vvy;k)$ with respect to
the number of dipoles,  we immediately get \eqref{eq.drecraspola}.
\end{proof}

The last proposition means that one obtains a good estimation of
the two components $\vp_i$ of $\vvp_i$ by solving the linear system
(\ref{eq.systpbloc}) at $\vvy=\vvy_i$. A part from terms that are small
in the Fraunhofer regime, the error is due to the presence
of the other dipoles  and involves the cross-range distances
$\|\vy_i-\vy_j\|$. Thus the error in  estimating the polarization vector
is small when the dipoles are well separated, i.e. with distances large
compared to the cross-range resolution $L/(ak)$. The
asymptotic \eqref{eq.drecraspola} tells us that $\|\vp\|$ is
also a good imaging function for the dipole's position since it decreases
as  reciprocal of the distance to the closest dipole.  We expect
$\|\vp\|$ to have the same resolution $L/(ak)$ as the Fraunhofer imaging
function $ \cI(\vvy;k)$ (see proposition \ref{eq.propcrossrang}).

\subsubsection{Numerical experiments}
\label{sec:numexp:passive1}
We illustrate our cross-range resolution analysis in figures
\ref{fig.illcond} and \ref{fig.goodcond}. Our goal is to contrast the
difference between solving for all 3 components of the polarization
vector for a single dipole and solving only for the cross-range
components. We consider a regime corresponding to microwaves in
vacuum, i.e. propagation velocity $c=3\times10^8$ m/s, frequency
$f_0=2.4$GHz and wavelength $\lambda_0=0.125$ m. The array $\cA$ is a square centered
at the origin and located within the plane $z=0$. Its side aperture is $a=20
\lambda_0$ and is composed of $40 \times 40$ equally spaced, point-like
receivers. The single dipole is located at
$\vvy_* = (0,0,L=100 \lambda_0)$  and has polarization vector $\vvp_*
=(1+2\mi,1-1\mi, 1+1\mi)$. 

In figure \ref{fig.illcond}, the reconstruction of both the position and
polarization vector is performed by solving the ill-conditioned linear
system \eqref{eq.systptot}. Here $\|\vvp\|$ is a poor approximation to
$\|\vvp_*\|=3$ near the dipole location, so much so that the image
appears as that of two nearby dipoles. Since the directions
$\operatorname{Re}(\vp)$ and $\operatorname{Im}(\vp)$ are still well
recovered, this means that $p_z$ is not well reconstructed. The focal
spot is broader compared to that in figure~\ref{fig.goodcond}, which
indicates that $\|\vvp\|$ does not give a precise reconstruction of the
position.  The recovery of the directions $\operatorname{Re}(\vp)$ and
$\operatorname{Im}(\vp)$ is also worse than in figure~\ref{fig.goodcond}
because they do not decrease as fast when one moves away from the
dipole.

In figure \ref{fig.goodcond}, the reconstruction of both the position
and cross-range polarization $\vp$ is performed by solving the
well-conditioned linear system \eqref{eq.systptot}. At the dipole
position we have $\|\vp\| = 2.5$, which is close to the true value
$\|\vp_*\|$ and the focal spot is nicely centered about the dipole. The
size of the focal spot is consistent with the Rayleigh criterion, i.e.
$\lambda_0 L/a=5 \lambda_0$. This confirms the asymptotic  expression
\eqref{eq.drecraspola} which tells us that in the case of a single
dipole, $\|\vp\|$ has a cross-range resolution $L/(ak)$ similar to the
Kirchhoff imaging function $\cI(\vvy;k)$. Moreover, one gets a stable
reconstruction of $\operatorname{Re}(\vp)$ and $\operatorname{Im}(\vp)$
near the dipole position.

\begin{figure}[htp]
\begin{center}
\includegraphics[width=6.05cm]{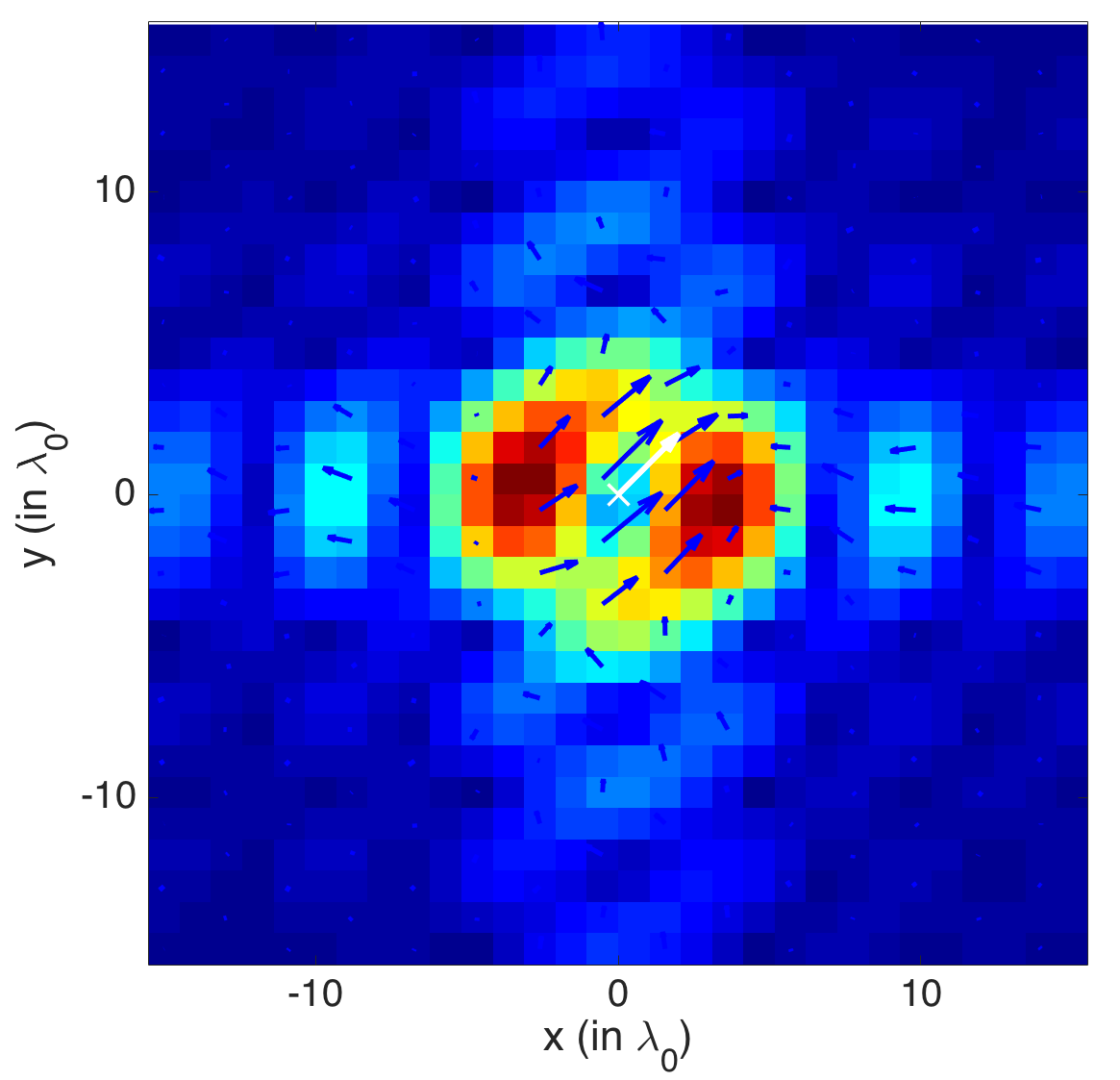}
\includegraphics[width=6.42cm]{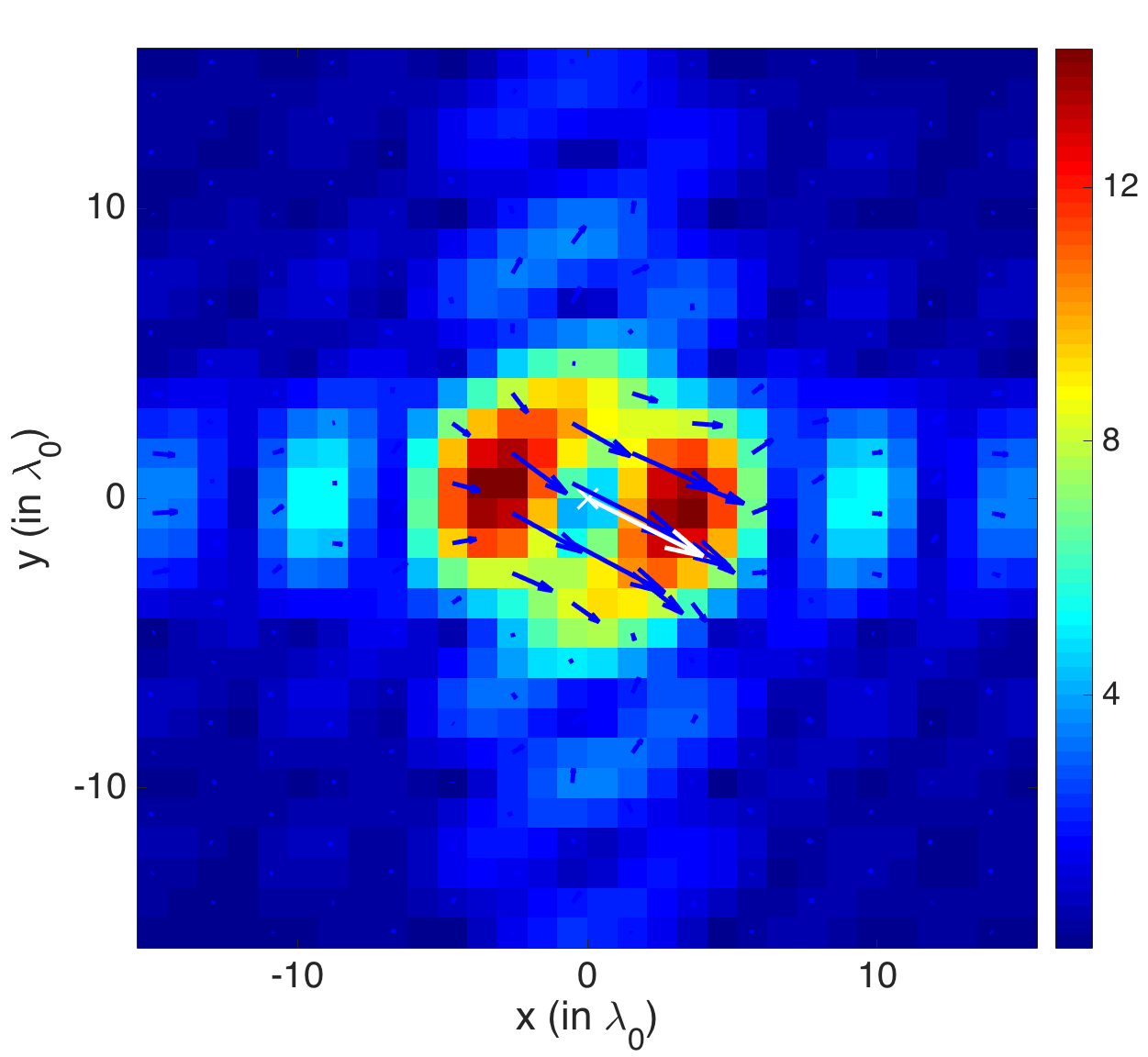}
\end{center}
\caption{Recovery of all three components of the polarization vector for
a single dipole (white cross) in the $z=L$ plane by solving the
(ill-conditioned) system \eqref{eq.systptot}. The color scale represents
$\|\vvp\|$ (i.e. the norm of all three polarization vector components). The left image contains the real part of the reconstructed
(blue) and true (white) cross-range polarizations. The right image is
similar for the imaginary part.}
\label{fig.illcond}
\end{figure}

\begin{figure}[htp]
\begin{center}
\includegraphics[width=6.05cm]{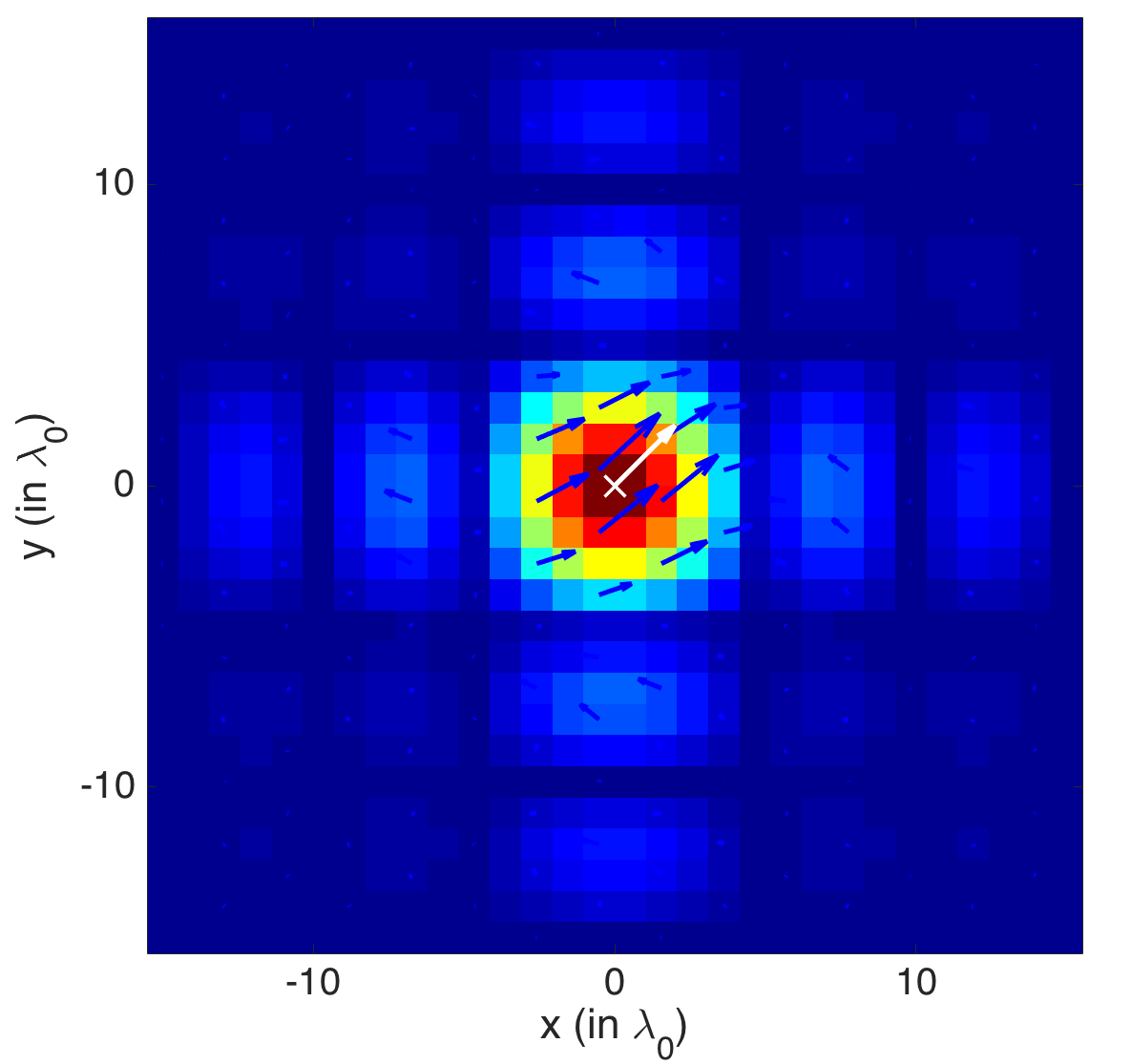}
\includegraphics[width=6.42cm]{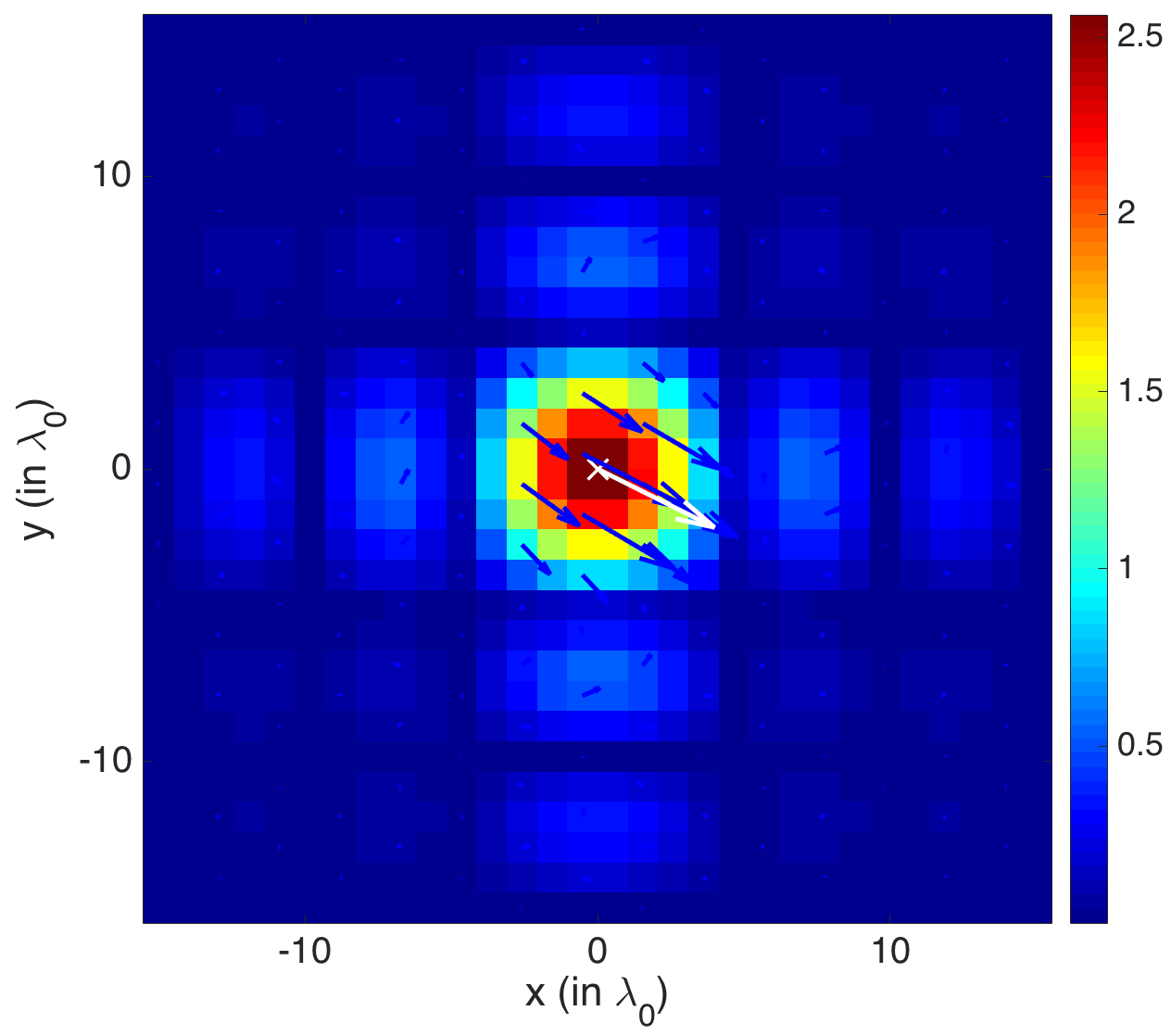}
\end{center}
\caption{Recovery of the cross-range components of the polarization
vector for a single dipole (white cross) in the $z=L$ plane by solving the
system \eqref{eq.systpbloc}. The color scale represents $\|\vp\|$ (i.e.
the norm of the cross-range component of the polarization vector). The
left image contains the real part of the reconstructed (blue) and true (white)
cross-range polarizations. The right image is similar for the imaginary
part.}
\label{fig.goodcond}
\end{figure}

\subsection{Range estimation of the position}
\label{sec:range:passive}
Similarly to the acoustic case, Kirchhoff migration can resolve the
location of a reflector in depth (range) by integrating $\cI(\vvy;k)$
over a frequency band \cite{Blei:2013:MSI}. Here we consider the band
$[\omega_0-B/2, \omega_0+ B/2]$ with central frequency $\omega_0$ and
bandwidth $B$. By linearity, we consider only the case of a single
dipole located at $\vvy_*=(\vy_*,L+\eta_*)$ and polarization vector
$\vvp_*$.  For the analysis, we assume that the Fraunhofer
asymptotic regime (see section~\ref{sub.asympfraun}) holds uniformly
with $k$ in the whole frequency band $[\omega_0-B/2, \omega_0+ B/2]$. We
study the following imaging function
\begin{equation}\label{eq.kirchhoffmultiplefrequency}
\cI(\vvy)=\int_{|\omega-\omega_0|<B/2} \md \omega   \,
\cI\left(\vvy;\frac{\omega}{c}\right).
\end{equation}
We suppose that the cross-range position $\vy_*$ of the obstacle
is known and we evaluate the imaging function
\eqref{eq.kirchhoffmultiplefrequency} at points of the form $\vvy=(\vy_*,L+\eta)$.
Hence, using the asymptotic \eqref{eq.decompimage}  of
$\cI(\vvy;\omega/c)$ yields (for $N=1$) 
\begin{eqnarray}\label{eq.Imagmf}
\cI(\vvy)&=&\int_{|\omega-\omega_0|<B/2}\md \omega \,
\left[\widetilde{\mH}\big(\vvy,\vvy_*; \frac{\omega}{c}\big)  + \cO\left(\frac{a^4\Theta_a}{L^4}\right)+ \cO\left(\frac{a^2\Theta_b}{L^2} \right) \right] \vvp_*, \\
&=& \int_{|\omega-\omega_0|<B/2}\md \omega    \frac{\exp[\mi \omega(\eta_*-\eta)/c]}{(4\pi L)^2}\int_{\cA} \md\vx_r\, \mP(\vvx_r,\vvy)  \mP(\vvx_r,\vvy_*)  \, \vvp_* + \, B \,  o\left(\frac{a^2}{L^2}\right). \nonumber
\end{eqnarray}

The following proposition shows that as the range distance $|\eta -
\eta_*|$  between the imaging point and the dipole is large compared to
$c/B$, the norm of the imaging function becomes small. This is similar
to the range resolution estimate in acoustic \cite{Blei:2013:MSI,Borcea:2007:OWD}.

\begin{proposition}[Imaging function decrease in the range direction]
\label{eq.proprang}
When the array $\cA$ is a disk of radius $a$,
the Kirchhoff imaging function \eqref{eq.image}  of the dipole $\vvy_{*}=(\vy_{*},
L+\eta_{*})$ satisfies for all $\vvy=(\vy_*,L+\eta)$  with $\eta \neq
\eta_*$,
\begin{equation}\label{eq.asymprange1}
 \|\cI(\vvy)\| =    B  \frac{ a^2}{L^2} \left(  \cO\Big(    \frac{ c}{
 B   \left|\eta-\eta_{*} \right|} \Big)+o(1)\right).
\end{equation}
At the dipole location we have
\begin{eqnarray}\label{eq.asymprange2}
\cI(\vvy_{*})=   B \left[ \frac{  a^2 }{16 \pi L^2} \begin{pmatrix} 1 & 0 & 0\\ 0 & 1 & 0\\  0 &0 &0  \end{pmatrix}  \vvp_*+o\left(\frac{a^2}{L^2}\right)\right].
\end{eqnarray}
\end{proposition}
\begin{proof}
When $\eta\neq \eta_*$,  we get
$$
\|\cI(\vvy) \| \leq  C\frac{ a^2 }{ L^2}  \left[\Big|
\int_{|\omega-\omega_0|<B/2}\md \omega    \exp[\mi\omega(\eta_*-\eta)/c]
\Big|+B o(1)\right]
$$
for some $C>0$, because orthogonal projectors have induced matrix 2-norm
equal to one and the polarization vector is assumed to be $\cO(1)$.  By
evaluating the integral in frequency of the latter expression we
obtain
\begin{equation}\label{eq.sinscardinal}
 \int_{|\omega-\omega_0|<B/2}\md \omega    \exp[\mi\omega(\eta_*-\eta)/c] = B \exp[\mi \omega_0 (\eta_*-\eta)/c ] \operatorname{sinc}\left(\frac{B (\eta_*-\eta)}{2 \, c}\right)
\end{equation}
where $\operatorname{sinc}(x) \equiv \sin(x) / x$. Hence we have that
$$
\|\cI(\vvy) \| = B  \frac{ a^2}{L^2}  \left(  \cO\Big(    \frac{c}{ B   \left|\eta-\eta_* \right|} \Big)+o(1)\right).
$$

The expression \eqref{eq.Imagmf} of the imaging function  $\cI(\cdot)$
evaluated at the dipole location $\vvy_*$ is
$$
\cI(\vvy_*)=\left[\int_{|\omega-\omega_0|<B/2}\md \omega \, \right]  \left[ \int_{\cA} \md\vx_r\,   \mP(\vvx_r,\vvy_*) \right]  \, \vvp_*+ B \,  o\left(\frac{a^2}{L^2}\right),
$$
and by lemma~\ref{lem.crossrange1} satisfies the following asymptotic
$$
\cI(\vvy_*)=\displaystyle  B \left[ \frac{  a^2 }{16 \pi L^2} \begin{pmatrix} 1 & 0 & 0\\ 0 & 1 & 0\\  0 &0 &0  \end{pmatrix}  \vvp_*+o\left(\frac{a^2}{L^2}\right)\right].
$$
\end{proof}

\subsection{Polarization vector recovery in the range direction}
\label{sec:rangepol:passive}

As we discussed in section~\ref{sec:crossrangep}, we cannot expect to
stably recover the range component $p_z$ of a dipole's polarization
vector $\vvp$ from the Kirchhoff image $\cI$. A straightforward
integration in frequency of the linear system approach of
section~\ref{sec:crossrangep} gives a good estimate of the cross-range
polarization vector $\vp$ if we knew the range position of the dipole.
Unfortunately, moving in depth gives oscillatory artifacts in $\vp$. We
characterize these artifacts and show how to eliminate them.

\subsubsection{Analysis of polarization vector image in range}
For the analysis, we consider once again a family of
dipoles with positions  $\vvy_1,\ldots, \vvy_N$ and polarizations
$\vvp_1, \ldots, \vvp_N$. To recover the cross-range components of the
polarization vectors, we integrate \eqref{eq.systpbloc} over the
frequency band and solve the following linear system
\begin{equation}\label{eq.systmultifreq}
\Big[\int_{|\omega_0 - \omega| < B/2} \md \omega\,\mH_{1:2,1:2}\Big(\vvy,\vvy;\frac{\omega}{c}\Big)  \Big]\, \vp= \mathcal{I}_{\text{KM}}(\vvy),
\end{equation}
where $\mathcal{I}_{\text{KM}}(\vvy)$ denotes the two first components of the
imaging function $\cI(\vvy)$.  Naturally, the solution of this system
leads to a stable reconstruction of both the position $\vvy_i$ and the
polarization $\vp_i$ of each dipole in the cross-range. Indeed, it is
straightforward to check that integrating the system
\eqref{eq.systpbloc} over the frequency band does not change: (a) its
invertibility in the Fraunhofer regime, (b) its condition number
being close to one and (c) the resolution estimates
\eqref{eq.crossrangeplarecover} and \eqref{eq.drecraspola}.

Now we study the behavior in range of this procedure to recover the
cross-range component of the polarization vector image $\vp$. Here we
isolate the effect of range by considering the case where all the
dipoles have same cross-range, i.e. $\vvy_i=(\vy_*,L+\eta_i)$ for $i=1,
\ldots, N$. The following proposition shows that the resolution of
$\|\vp\|$ in the range direction is $c/B$ (as in acoustics, see e.g.
\cite{Blei:2013:MSI,Borcea:2007:OWD}).  Furthermore at the dipole position
$\vvy_i$, one recovers the polarization vector $\vp_i$ provided that
the range distance between the different dipoles $|\eta-\eta_j|$ is
large with respect to the range resolution $c/B$.

\begin{proposition}\label{prop.polarange}
When the array $\cA$ is a disk of radius $a$ and the dipoles are all
aligned in the range direction of $\cA$, the cross-range polarization
image $\vp$ (obtained by solving the linear system
\eqref{eq.systmultifreq} and depending of the imaging point $\vvy$)
satisfies the two following estimates:
\begin{itemize}
\item If the imaging point is the dipole location, i.e. $\vvy=\vvy_i$ we
have
\begin{equation}\label{eq.rangeplarecover}
\|\vp-\vp_i\| = \cO \Big(  \frac{c} {B \min \limits_{i\neq j}|\eta_i-\eta_j|}\Big)+o(1),
\end{equation}
\item If the imaging point range is different from any of the dipole
ranges, $\vvy=(\vy_*,L+\eta)\neq \vvy_j$ (for all $j=1,\ldots,N$),
\begin{equation}\label{eq.drecraspolarange}
\|\vp\| =   \cO \Big(  \frac{c} {B \min \limits_{j=1,\ldots,N}|\eta-\eta_j|} \Big)+o(1).
\end{equation}
\end{itemize}
\end{proposition}

\begin{proof}
We introduce for convenience the matrix
$$
\mH_B(\vvy, \vvy')=\int_{|\omega - \omega_0|<B/2}\md \omega\, \mH \Big(\vvy,\vvy;\frac{\omega}{c} \Big).
$$
When $\vvy=\vvy_i$, the cross-range component $\vp_i$ of the
polarization satisfies
\begin{equation}\label{eq.defvpi}
[\mH_B(\vvy_i,\vvy_i)]_{1:2,1:2}\, \vp_i =\mathcal{I}_{\text{KM}}(\vvy_i) -
[\mH_B(\vvy_i,\vvy_i)]_{1:2,3} \, p_{i,z} - \Big[ \sum_{j=1,j\neq
i}^{N}\mH_B(\vvy_i,\vvy_j)\vvp_j \Big]_{1:2},
\end{equation}
which is the integral over the frequency band of \eqref{eq.idendityvpi}.
Thus from the systems \eqref{eq.defvpi} and \eqref{eq.systmultifreq}
satisfied by $\vp_i$ and $\vp$ one gets
\[
\|\vp-\vp_i\| \leq  \|[\mH_B(\vvy_i,\vvy_i)]_{1:2,1:2}^{-1} \|  \,
\left(\Big\| \big[
\sum_{j\neq i}^{N}\mH_B(\vvy_i,\vvy_j)\vvp_j \big]_{1:2} \Big\| + |
[\mH_B(\vvy_i,\vvy_i)]_{1:2,3}\, p_{i,z}|\right).
\]
By proceeding as in the proof of proposition~\ref{prop.polacrossactiv},
one can show by integrating over the frequency band that $\|[\mH_B(\vvy_i,\vvy_i)]_{1:2,1:2}^{-1}\|
=\cO(L^2/(a^2 B))$. Since the dipoles are aligned, the contribution from the other dipoles can be
controlled using proposition~\ref{eq.proprang}, giving 
\[ 
\Big\| \big[
\sum_{j\neq i}^{N}\mH_B(\vvy_i,\vvy_j)\vvp_j \big]_{1:2} \Big\| =
 \frac{Ba^2}{L^2} \left( \cO\Big(\frac{c}{ \min
\limits_{j\neq i}|\eta_i-\eta_j|}\Big)+o(1)\right).
\] 
Moreover, the asymptotic
\eqref{eq.asymptA} bounds the error due to only
looking at the cross-range component of the system:
$|[\mH_B(\vvy_i,\vvy_i)]_{1:2,3}\, p_{i,z}|=o(Ba^2/L^2)$. Combining the
last three estimates gives the desired result
\eqref{eq.rangeplarecover}. 

Finally, the asymptotics \eqref{eq.drecraspolarange} follows immediately
from $\|[\widetilde{\mH}_B(\vvy,\vvy)]_{1:2,1:2}^{-1}\| = \cO(L^2/(a^2
B))$ and the decay rate in range of the imaging function
\eqref{eq.asymprange1}.
\end{proof}

\subsubsection{Depth oscillation artifact and its suppression}
\label{sec:suppression}
It is known in acoustics that the reflection coefficient of a point
scatterer can only be recovered up to a complex phase, see e.g.
\cite{Nov::15:ISI}. A similar phenomenon is observed here:
if the range position of a scatterer is not known perfectly, the estimation 
of the cross-range polarization $\vp$ oscillates in range.
This can be easily seen by considering a single dipole at location
$\vvy_*$ and with polarization vector $\vvp_*$. Rewriting the imaging
function $\cI(\vy_*,L+\eta)$ for points with same cross-range as the
dipole gives together with \eqref{eq.Imagmf} and \eqref{eq.sinscardinal}
\begin{equation}
\label{eq:sincphase}
\cI(\vvy) =  \frac{B e^{\mi \, \omega_0 (\eta_*-\eta)/c } }{(4 \pi
L)^2}\, \operatorname{sinc}\left(\frac{B (\eta_*-\eta)}{2 \, c}\right)
\int_{\cA} \md\vx_r\, \mP(\vvx_r,\vvy)  \mP(\vvx_r,\vvy_*)  \, \vvp_* +
\, B \,  o\left(\frac{a^2}{L^2}\right).
\end{equation}
Clearly the presence of the complex exponential $\exp[\mi \, \omega_0
(\eta_*-\eta)/c ]$ and the $\operatorname{sinc}$ causes the image
$\cI(\vy_*,L+\eta)$ to oscillate in $\eta$. This oscillation is not
taken into account if we solve the linear system
\eqref{eq.systmultifreq} for $\vp$, indeed the system matrix does not
oscillate but the right hand side does. We point out that in the case
$\eta = \eta_*$, there are no such oscillations, which explains why this
error is not present in the error analysis for $\vp$ assuming a known
dipole position \eqref{eq.rangeplarecover}. Also the oscillations are
relevant because their length scale is close to $c/B$, the resolution in
depth.

To deal with this artifact, we estimate $\vp$ by solving
\eqref{eq.systmultifreq} and then we fix the phase of one component of
$\vp$. In the following, we have arbitrarily chosen to enforce that the
first component be real and positive, that is $\arg{p_x} =0$. This can
be achieved by post-processing the solution of
\eqref{eq.systmultifreq} by the operation $(\overline{p}_x/|p_x|)\vp$.
This operation is problematic for small $|p_x|$, in which case we can
use the $y$ component or we could also regularize using:
$(\overline{p}_x/(|p_x|+\delta))\vp$, for a small $\delta>0$.

\subsubsection{Numerical illustration of depth oscillation suppression}
Here we illustrate the depth oscillation and its correction in the same
setup as that in figures \ref{fig.illcond} and  \ref{fig.goodcond}.  In
figure \ref{fig.phase} we display $\|\vp\|$ in color scale and
$\Re(p_x)$ with arrows the plane $y=0$. The central frequency is
$f_0=2\pi \omega_0$ and the bandwidth $B/(2\pi)$ are both equal to
$2.4$GHz.  Thus, the ratio $2\pi c/B=\lambda_0$ gives in both images the
size of the focal spot in the range direction $z$. On the left, we
display $\operatorname{Re}(p_x)$ without any phase correction. The
dipole position and magnitude are accurately imaged, but the dipole polarization vector
oscillates in range direction inside the focal spot with a period
$c/f_0=\lambda_0$.  Hence the reconstruction is unstable. In the right
figure, we apply the correction $(\overline{p}_x/|p_x)\vp$ to both the
reconstructed $\vp$ and the true $\vp_*$ (white arrows). The correction
suppresses the phase oscillation and gives a stable reconstruction of
$\vp_*$, up to a complex sign.
\begin{figure}[h!]
\begin{center}
\includegraphics[width=6.1cm]{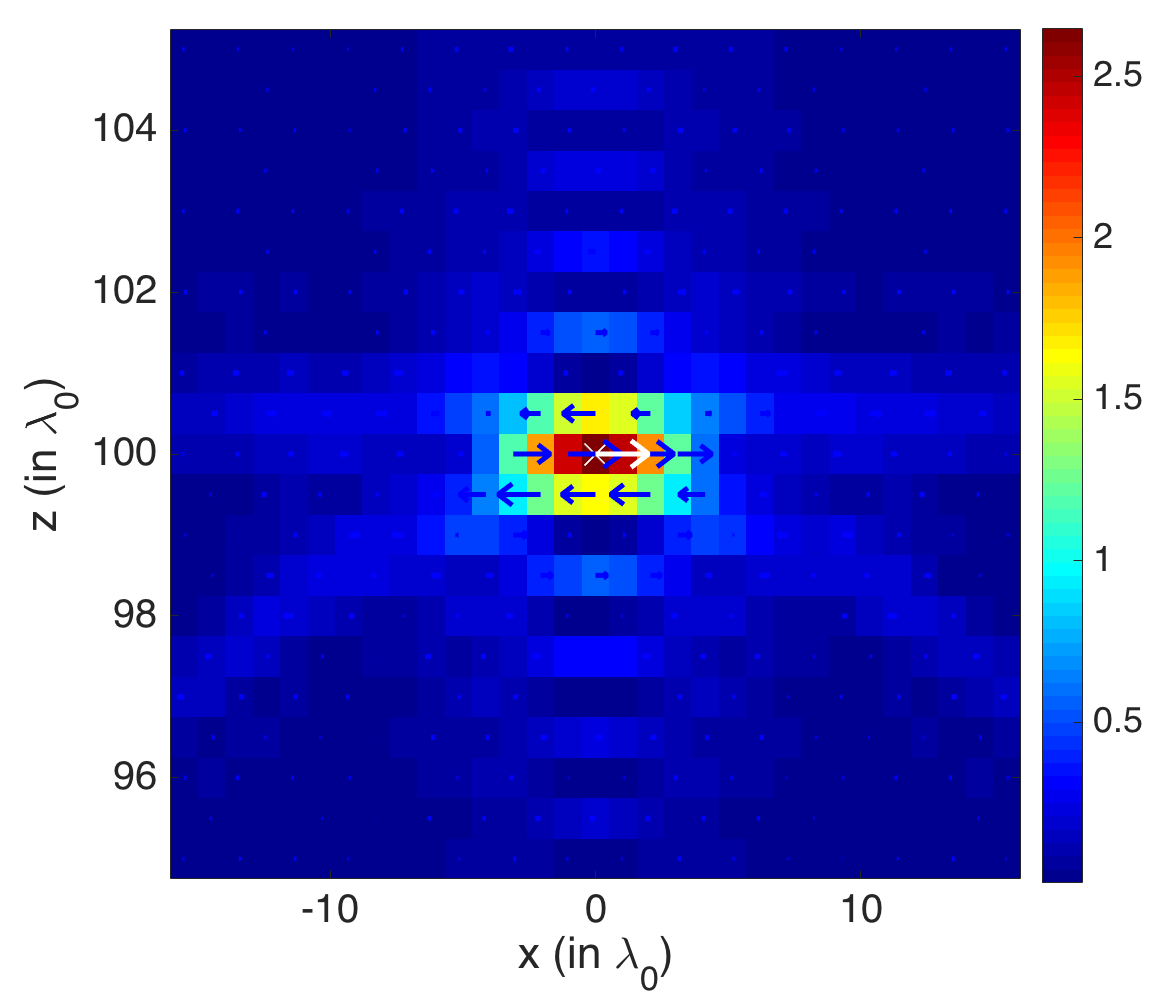}
\includegraphics[width=6.3cm]{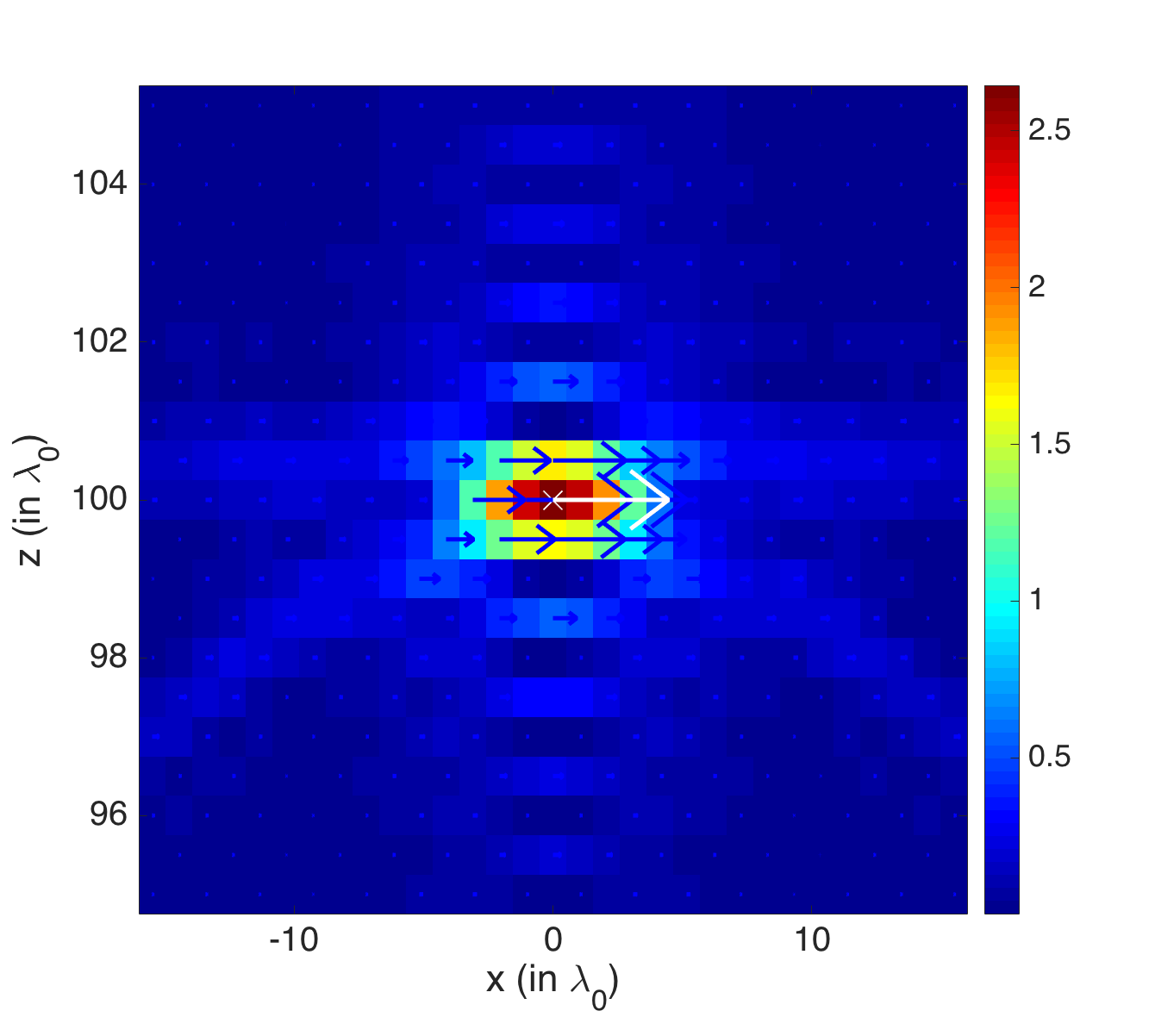}
\end{center}
\caption{Image of $\|\vp\|$ in the plane $y=0$  (color scale, both
figures). The blue arrows represent $\operatorname{Re}(p_x)$ without
phase correction (left) and with phase correction (right). The white
arrow represents $\operatorname{Re}(p_{x,*})$, with the same correction
applied.}
\label{fig.phase}
\end{figure}

\subsection{Numerical experiments for several dipoles}
\label{sec:numexp:passive2}
In figures \ref{fig.crosrrange1}, \ref{fig.crosrrange2} and
\ref{fig.range1}, we consider the case of three dipoles placed at
$\vvy_1 = (-7 \lambda_0,\, 7\lambda_0 ,\, L)$, 
$\vvy_2 = (7 \lambda_0,\, 7\lambda_0 ,\, L)$ and
$\vvy_3 = (2\lambda_0,-2\lambda_0 , \, L+7\lambda_0)$ 
with respective polarization vectors
$\vvp_1=(2 , 1-2\mi , 1-\mi)$,
$\vvp_2=(-2 , 2-2\mi  ,1+\mi)$, and
$\vvp_3=(1, 2+2\mi, 1-\mi)$.
The cross-range
polarization vector components have norms $\|\vp\| \approx 3$, $3.5$ and
$3$, respectively. The array $\cA$, the bandwidth $B$ and the central
frequency $f_0$  are identical to the ones used in figure
\ref{fig.phase}. We visualize with the same convention as in the
previous figures the reconstruction of the positions and the
polarization vectors obtained by solving the linear system
\eqref{eq.systmultifreq} and applying the phase correction of
section~\ref{sec:suppression}.  Figures \ref{fig.crosrrange1} and
\ref{fig.crosrrange2} illustrate the reconstruction and the resolution
of $\|\vp\|$ in the cross-range of each dipole. Once again, the focal
spot is given by the Rayleigh criterion: $\lambda_0 L/a=5 \lambda_0$. In
each case, we observe a stable reconstruction of $\|\vp\|$ and of the
complex vector $(\overline{p}_x/|p_x|)\vp$. Figure \ref{fig.range1}
illustrates the range resolution of each dipole. The size of the focal
spot is again of order $2\pi c/B=\lambda_0$. We note a stable
reconstruction of $\operatorname{Re}((\overline{p}_x/|p_x|) p_x)$ with
no oscillations in the range direction (as by convention
$\operatorname{Im}((\overline{p}_x/|p_x|) p_x)=0$, it is not represented
in our figures). Even if $(\overline{p}_x/|p_x|)
p_y$ is reconstructed accurately by our method, we chose not to display
it in figure~\ref{fig.range1} because the axis are $xz$.

\begin{figure}[htp]
\begin{center}
\includegraphics[width=5.9cm]{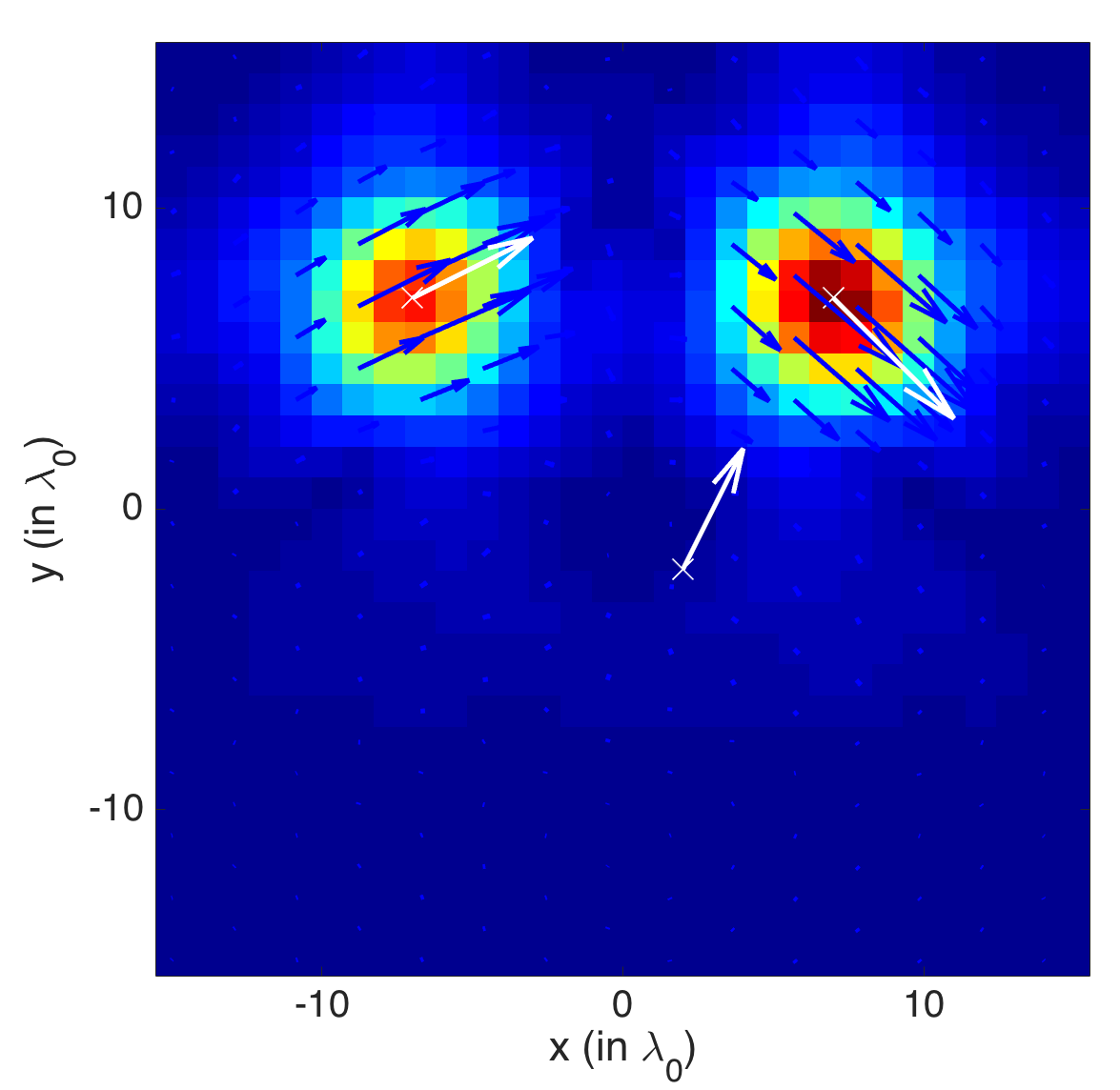}
\hspace{0.2cm}
\includegraphics[width=6.3cm]{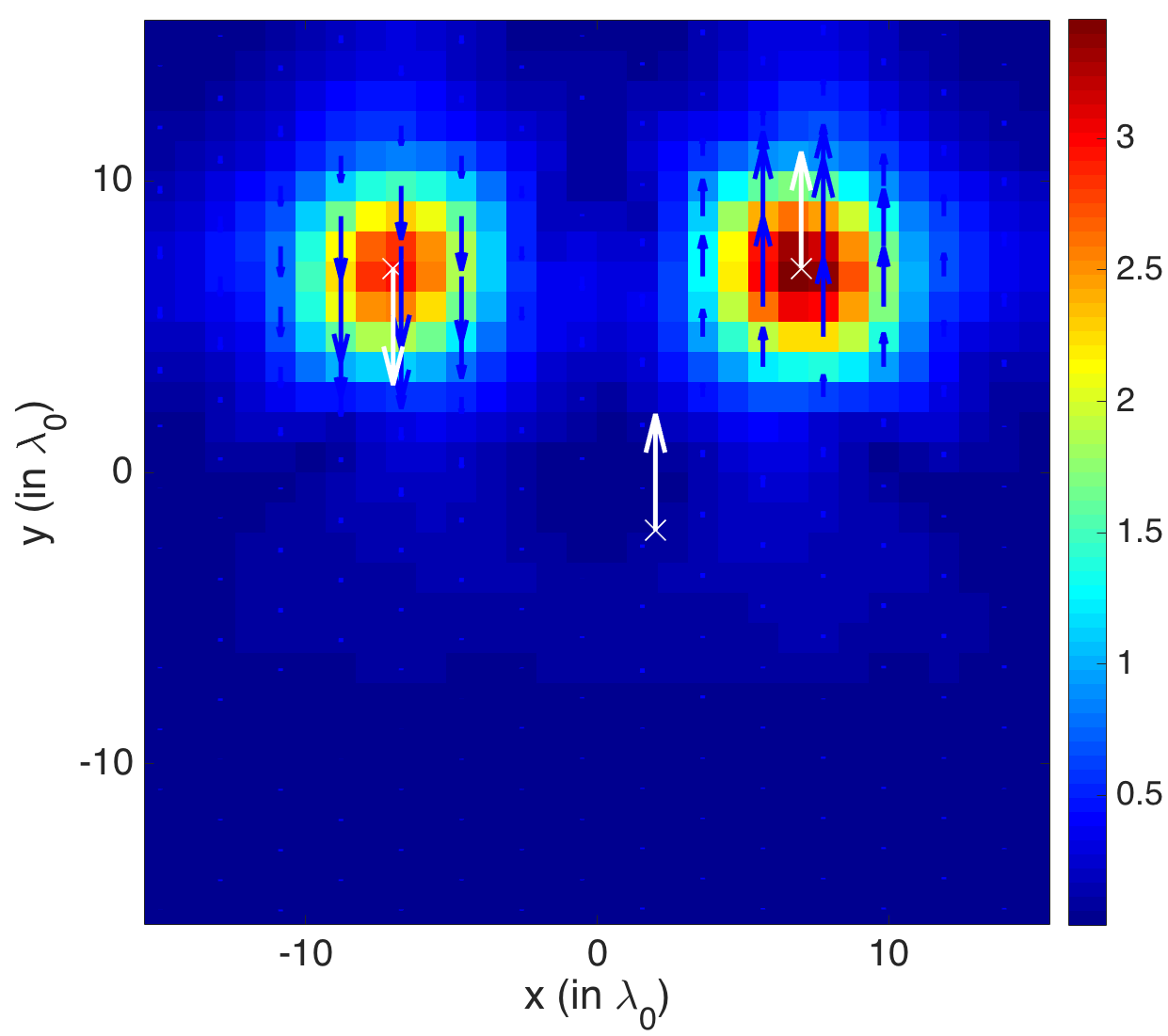}
\end{center}
\caption{Image of $\|\vp\|$ in the plane $z=0$ (color scale on both figures). The blue
arrows on the left represent 
$\operatorname{Re}((\overline{p}_x/|p_x|)\vp)$  and on the right $\operatorname{Im}((\overline{p}_x/|p_x|) \vp)$. The white arrows
represent the corresponding true quantities.} \label{fig.crosrrange1}
\end{figure}

\begin{figure}[htp]
\begin{center}
\includegraphics[width=5.9cm]{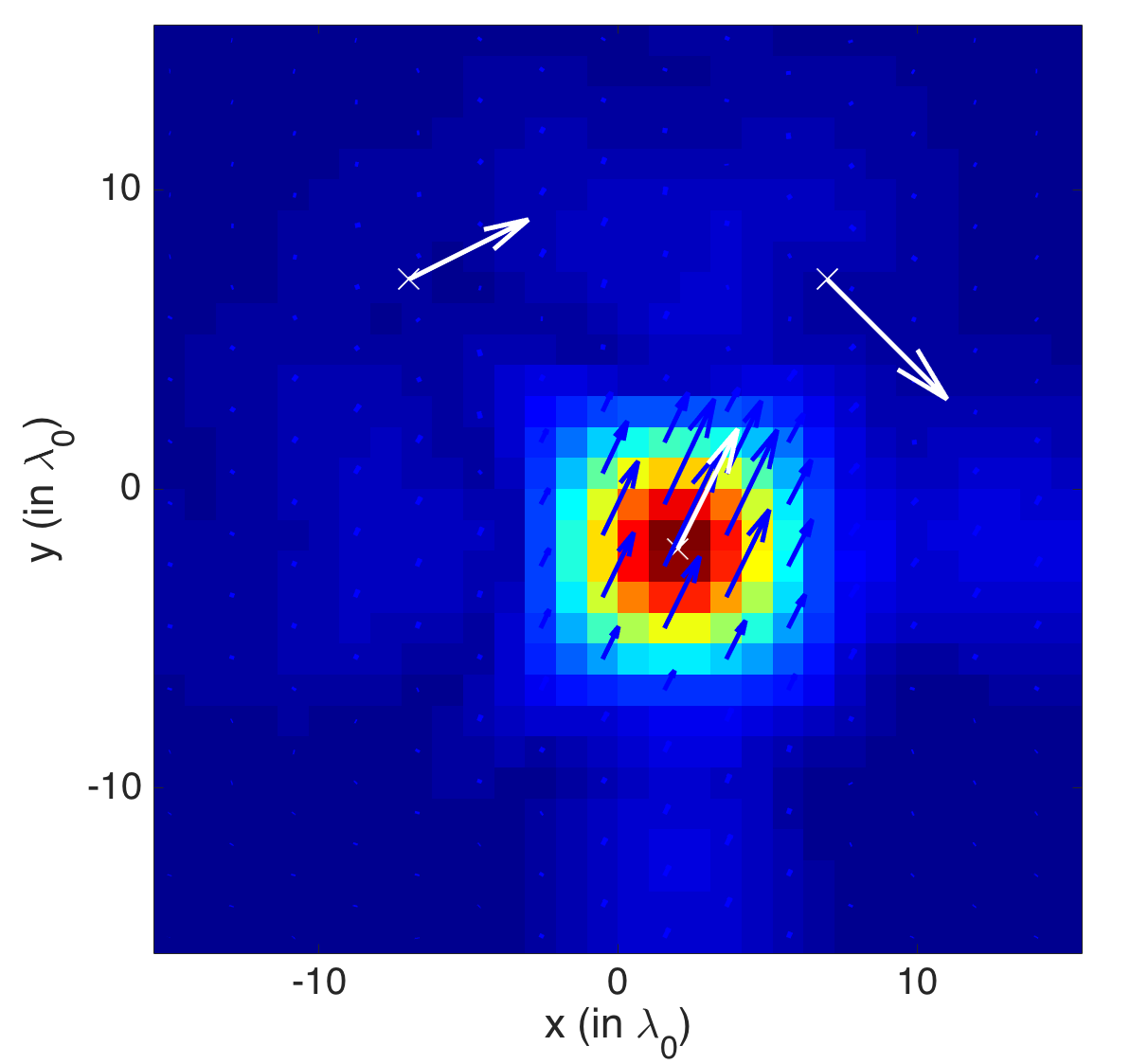}
\hspace{0.2cm}
\includegraphics[width=6.3cm]{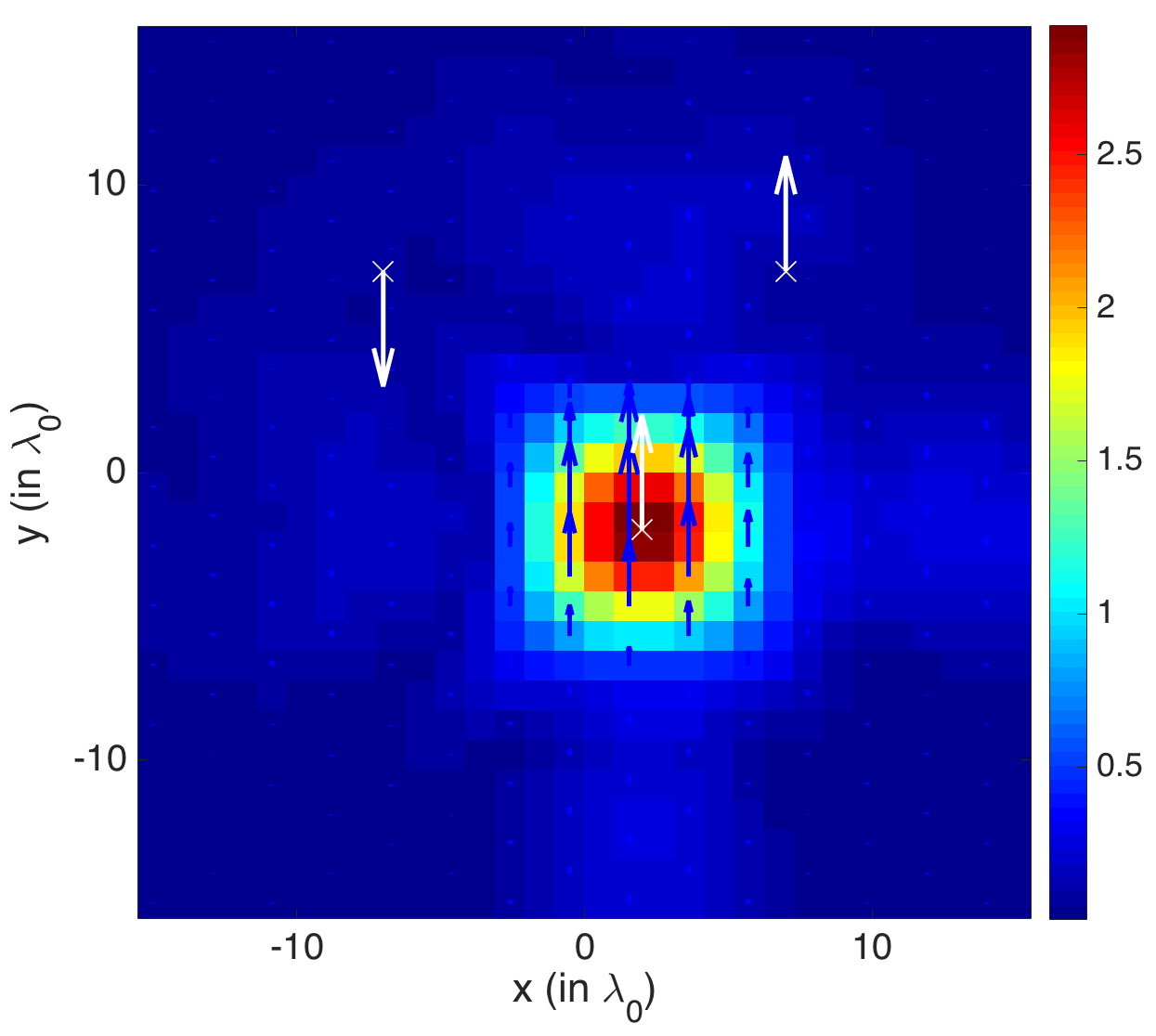}
\end{center}
\caption{Same as in the figure \ref{fig.crosrrange1} but in the plane $z=7\lambda_0$.}
\label{fig.crosrrange2}
\end{figure}

\begin{figure}[htp]
\begin{center}
\includegraphics[width=6.05cm]{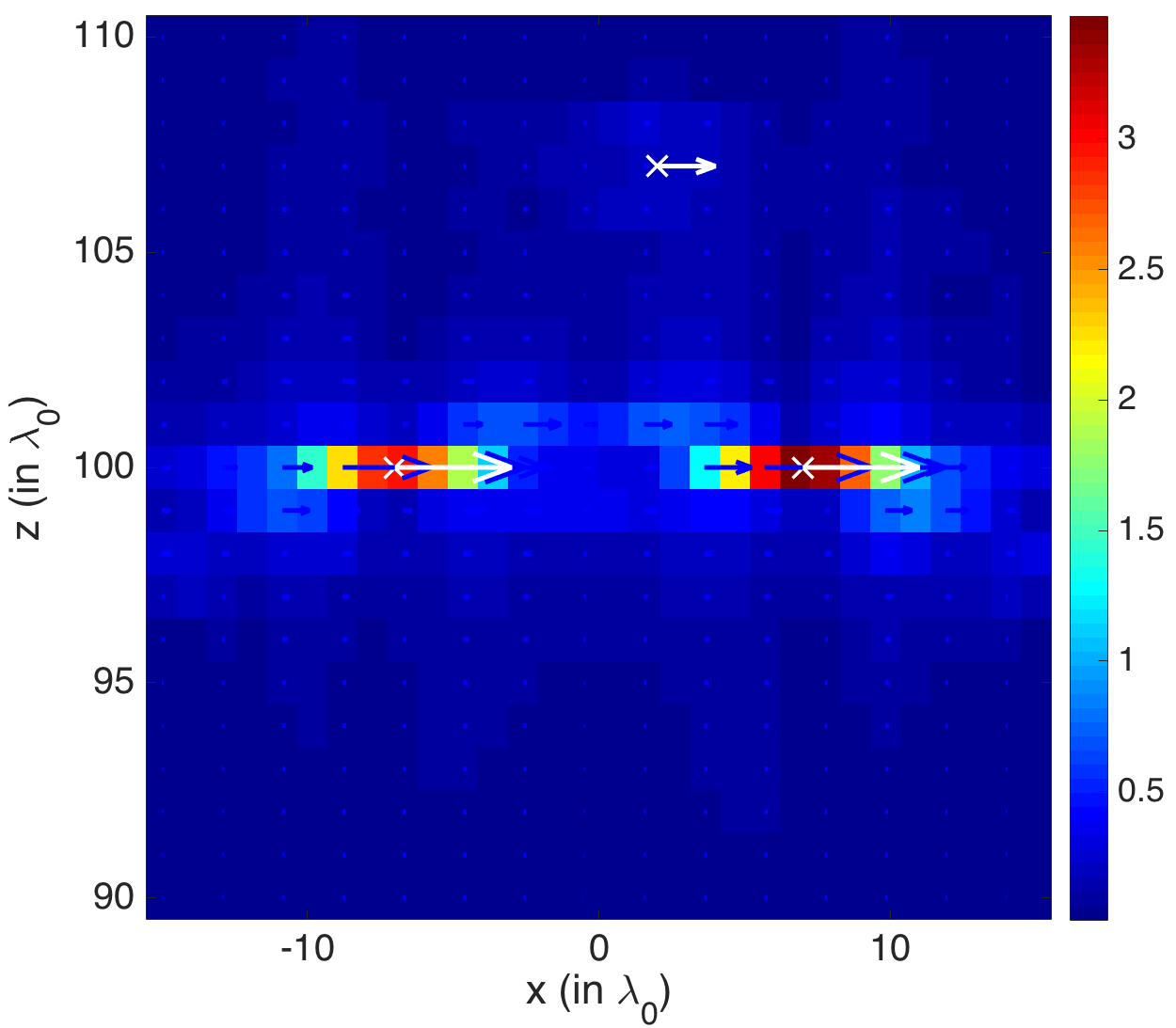}
\hspace{0.2cm}
\includegraphics[width=6.05cm]{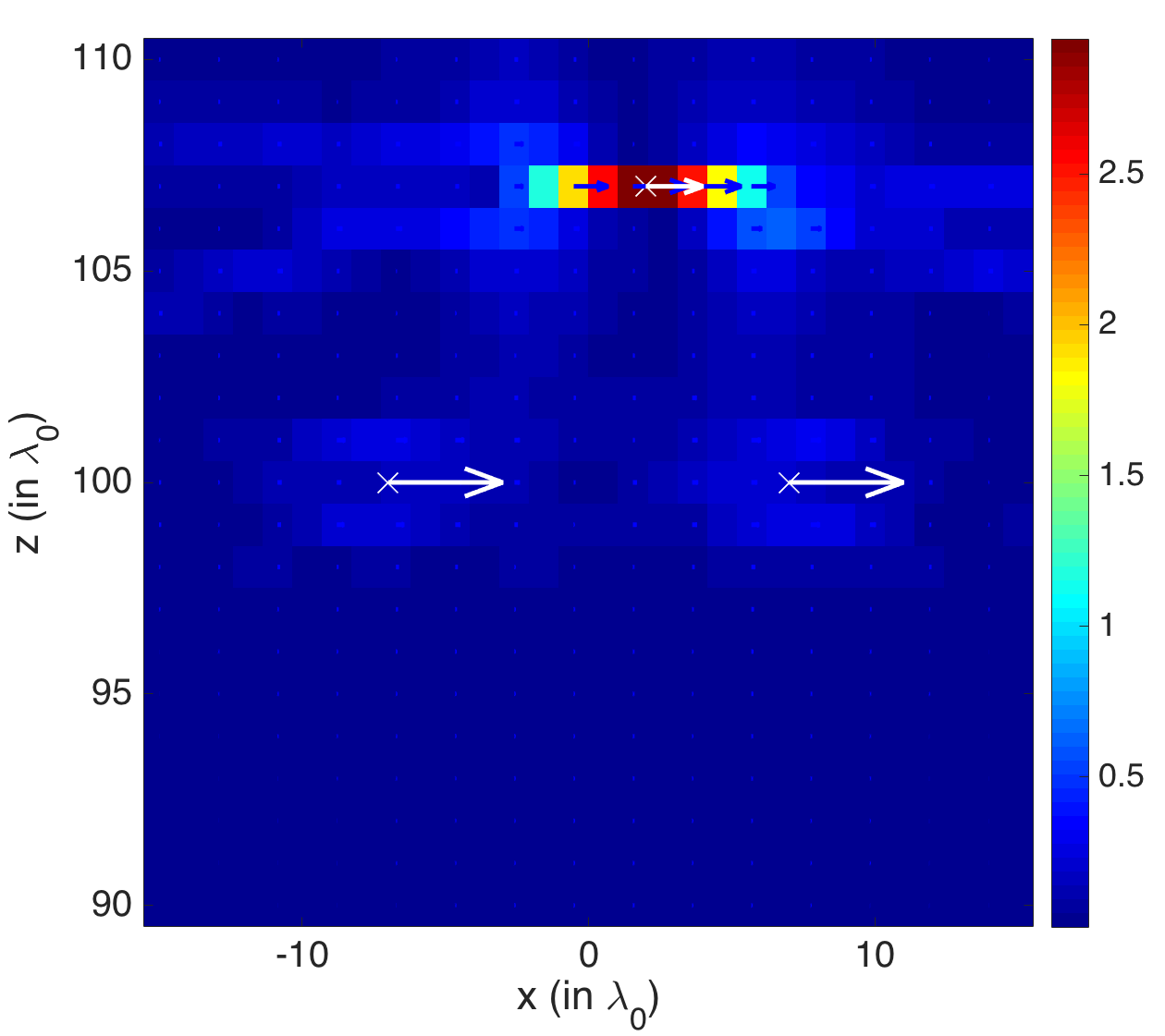}
\end{center}
\caption{Image of $\|\vp\|$ (color scale) and of
$(\overline{p}_x / |p_x|) p_x$ (blue arrows). The plane $y=7
\,\lambda_0$ is on the left and the plane $y=-2 \,\lambda_0$ is on the right.}
\label{fig.range1}
\end{figure}

\section{The active imaging problem}
\label{sec:active}
Here we focus on the imaging problem where the array is composed of
sources and receivers. The problem and the imaging function we use are
defined in section~\ref{sec:math:form:active}. The generalization of the
Kirchhoff imaging function we consider associates to each imaging point
a matrix. As we show in section~\ref{sec:cr:pola}, the cross-range
resolution estimate for the norm of this matrix coincides with the
classic estimate in acoustics. The procedure to extract the
polarizability tensor from the image is given in
section~\ref{sec:cr:polamat}. We then proceed to study the resolution in
the range direction and we obtain results that are similar to those in
acoustics (section~\ref{sec:r:pola}). The reconstructed polarizability
in range is studied in
section~\ref{sec:range:active}. As in the passive case, the image needs
to be corrected in the depth direction to suppress oscillatory
artifacts. Finally we report numerical experiments illustrating our
results in section~\ref{sec:numexp:active}.

\subsection{Mathematical formulation of the active imaging problem}
\label{sec:math:form:active}
We now consider the problem of imaging point-like scatterers using a
collocated array $\cA$ of sources and receivers within the $z=0$ plane
(see figure~\ref{fig.Frau.active}).
The field scattered by a small object located at $\vvy_*$ resulting from
an incident field $\vvE_{\mathrm{inc}}(\vvx;k)$ is
\[
 \vvE_{\mathrm{scatt}}(\vvx;k) = \mu \omega^2 \mG(\vvx,\vvy_*;k) \bGa_*
 \vvE_{\mathrm{inc}}(\vvy_*;k),
\]
where $\bGa_* \in \complex^{3\times 3}$ is the {\em polarizability
tensor} of the scatterer, satisfying $\bGa^T_* = \bGa_*$ by reciprocity,
see e.g.  \cite{Novotny:2012:PNO}. The analogue in acoustics of the
matrix $\bGa_*$ is the reflection coefficient. In general, the
polarizability tensor depends on $k$, but for our study we assume that
the dependence is weak within the frequency band we consider (this
occurs when e.g.  the scatterers are standard dielectric,
\cite{Milton}). Our goal is to image both the {\em position} and {\em
polarizability tensors} of a collection of these scatterers. We work
under a weak scattering assumption, where the scattered field from $N$
scatterers located at $\vvy_n$ with polarizability tensors $\bGa_n$,
$n=1,\ldots,N$ is given by the Born approximation (see e.g.
\cite{Born:1959:POE})
\begin{equation}
\label{eq.diffractpass}
\vvE_{\mathrm{scat}}(\vvx;k)=\mu \omega^2 \sum_{n=1}^{N}
\mG(\vvx,\vvy_n;k) \bGa_n  \vvE_{\mathrm{inc}}(\vvy_n;k).
\end{equation}
The field used to probe the medium is controlled by a distribution
$\vvp: \cA \to \complex^3$ where $\vvp(\vx_s)$ is the polarization
vector of a dipole at $\vvx_s=(\vx_s,0) \in \cA$. The electric field
generated by the array dipole distribution $\vvp(\vx_s)$ is
\begin{equation}
\label{eq:inc}
\vvE_{\mathrm{inc}}(\vvx;k)=\mu \omega^2 \int_{\cA} \md \vvx_s \mG(\vvx,\vvx_s;k)\, \vvp(\vvx_s).
\end{equation}
The data we use to image is the scattered field at the array. Since
different array dipole distributions can be used, the data that can be
collected in this setup can be thought of as the (matrix valued) array response function
$\Pi(\vx_r,\vx_s;k) \in \complex^{3\times 3}$ defined for $\vx_r,\vx_s \in \cA$ and wavenumber
$k$ such that the scattered field resulting from the array dipole
distribution $\vvp(\vvx_s)$ is
$$
\vvE_{\mathrm{scat}}(\vvx_r;k)= \mu^2 \omega^4  \int_{\cA} \Pi(
\vx_r,\vx_s;k) \vvp(\vvx_s) \, \md \vx_s.
$$ 
Combining \eqref{eq.diffractpass} and \eqref{eq:inc}, the array response
function for $N$ point-like scatterers is
\begin{equation}
\label{eq.datapassif}
 \Pi(\vx_r,\vx_s, ;k)= \sum_{n=1}^{N}   \mG(\vvx_r,\vvy_n;k) \bGa_n
 \mG(\vvy_n,\vvx_s;k),~ \text{for}~ \vx_s, \vx_r \in \cA.
\end{equation}

\begin{figure}[!ht]
\begin{center}
 \includegraphics[width=0.65\textwidth]{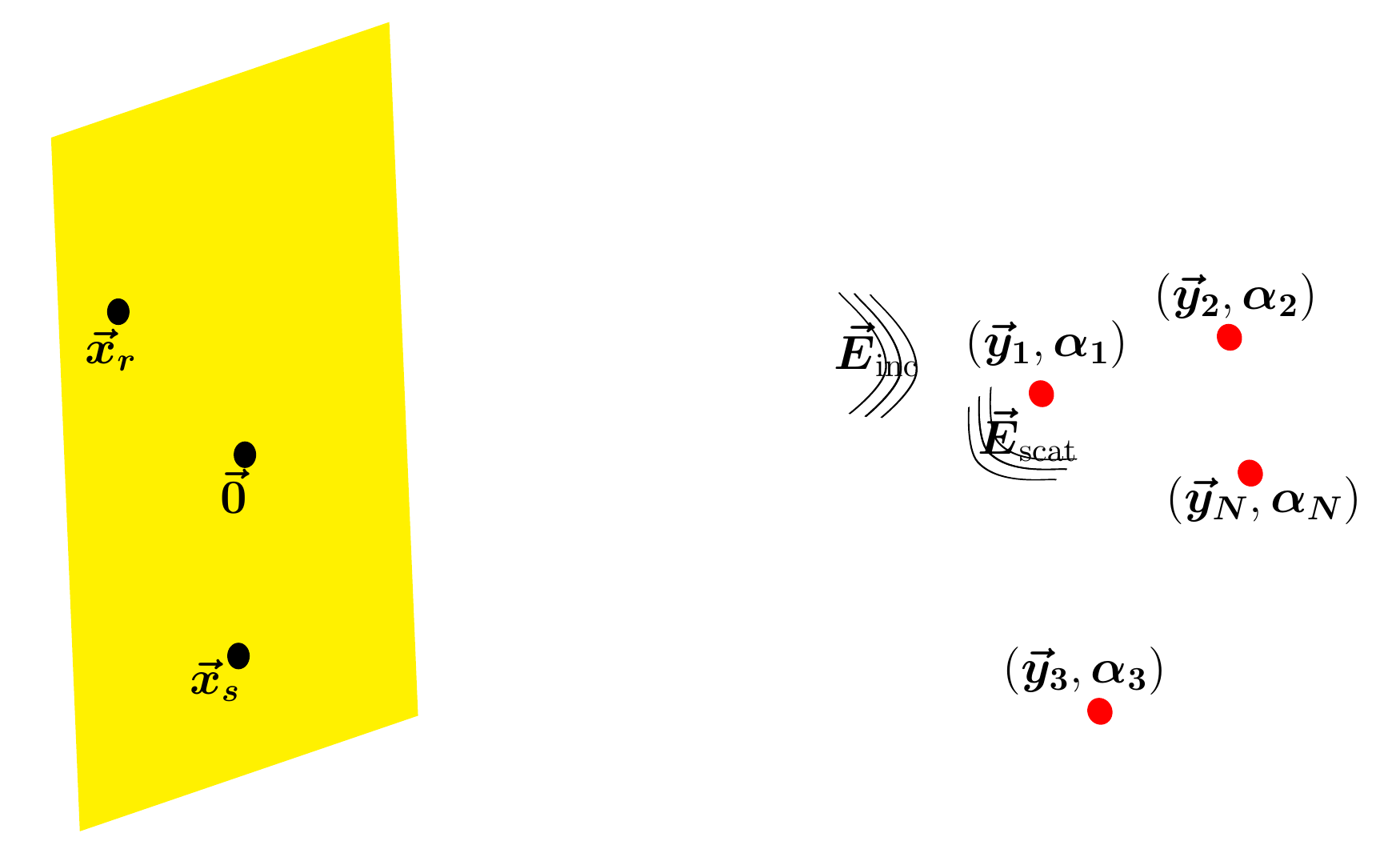}
\end{center}
 \caption{Active imaging problem description.}
 \label{fig.Frau.active}
\end{figure}

The imaging function we use is an electromagnetic version of the
Kirchhoff imaging function
\begin{equation}\label{eq.Kirchpass_general}
\bbI(\vvy;k)=\int_{\cA} \int_{\cA} \md \vx_r \, \md
\vx_s\overline{\mG(\vvx_r,\vvy;k)} \, \Pi(
\vvx_r,\vvx_s; k) \, \overline{\mG(\vvx_s,\vvy;k)},
\end{equation}
which gives a $3 \times 3$ complex matrix at each imaging point $\vvy$.
For our particular data, the image is
\begin{eqnarray}\label{eq.Kirchpass}
\bbI(\vvy;k)&=&  \sum_{n=1}^{N} \left[ \int_{\cA}   \md \vvx_r \overline{\mG(\vvx_r,\vvy;k)}  \mG(\vvx_r,\vvy_n;k) \right] \bGa_n  \left[\int_{\cA}  \md \vvx_s   \mG(\vvy_n,\vvx_s;k) \overline{\mG(\vvx_s,\vvy;k)} \right] \nonumber \\
&=& \sum_{n=1}^{N}  \mH(\vvy,\vvy_n;k) \, \bGa_n \, \mH(\vvy,\vvy_n;k)^{\top},
\end{eqnarray}
where the $3\times3$ matrix $\mH(\vvy,\vvy_n,k)$ is the ``point spread
matrix'' defined in \eqref{eq.defmatA}.

\subsection{Cross-range estimations of the position}
\label{sec:cr:pola}
To estimate the cross-range resolution of the active Kirchhoff imaging
function $\bbI(\vvy;k)$, we consider the case of one point-like
scatterer 
($N=1$) located at $\vvy_{*}=(\vy_{*}, L+\eta_*)$ with polarizability
tensor $\bGa_{*}$. We assume that the range position $L+\eta_*$ is known.
The case of multiple point-like scatterers is obtained by linearity. 
In the following we assume that $\|\bGa_n\|=\cO(1)$.  It is
convenient to introduce the block decomposition of the polarizability tensor
$$
\bGa_{*}=\begin{pmatrix} \bGa_{*;1:2,1:2}&
\bGa_{*,1:2,3}\\ \bGa_{*;3,1:2}&  \bGa_{*;3,3}\end{pmatrix},
$$
where $\bGa_{*;1:2,1:2}$, is a $2\times 2$ matrix,
$\bGa_{*;1:2,3}$ is a $2\times 1$ matrix, and so on.

The next proposition shows that the cross-range resolution of the
Kirchhoff imaging function is also given by the Rayleigh criterion
$L/(ka)$. The difference from imaging active sources is that for
scatterers, the image function decays faster. To be more precise, the
decay is in the order of  $(L/(ka \| \vy-\vy_*\|))^2$ for scatterers
versus $L/(ka \| \vy-\vy_*\|)$ for sources, see proposition
\ref{eq.propcrossrang}.

\begin{proposition}[Decreasing of the imaging function in cross-range]\label{prop.crossrangpass} 
The Kirchhoff imaging function (\ref{eq.Kirchpass}) of a point-like
scatterer located at $\vvy_{*}=(\vy_{*}, L+\eta_{*})$ and evaluated at
$\vvy = (\vy,L+\eta_*)$ satisfies
\begin{equation}\label{eq.asymppas1}
 \|\bbI(\vvy;k)\| \leq   \frac{ a^4}{L^4}  \left[  \cO  \left( \Big(
 \frac{ L}{  a \, k   \left\|\vy-\vy_* \right\|}\Big)\right)^2+o(1) \right],  \  \forall \vvy=(\vy,L+\eta_*) \ \mbox{ with } \vy\neq \vy_*,
\end{equation}
where the $o(1)$ is explicitly given by
$\cO(a^4\Theta_a^2/L^4)+\cO(\Theta_b^2)$. When the shape of the array
$\cA$ is a disk of radius $a$, one has
\begin{eqnarray}\label{eq.asymppas2}
\bbI(\vvy_*;k)=\frac{  a^4}{(16 \pi)^2 L^4} \begin{pmatrix}  \bGa_{*;1:2,1:2}& \boldsymbol {0} \\  \boldsymbol {0} & 0 \end{pmatrix} +o\left(\frac{a^4}{L^4}\right).
\end{eqnarray}
\end{proposition}
\begin{proof}
Let $\vvy=(\vy, L+\eta_{*})$ be a point in the cross-range of the dipole
$\vvy_{*}$  with  $\vy\neq \vy_*$.  As $\| \mH(\vvy,\vvy_n,k)\|=\|
\mH(\vvy,\vvy_n,k)^{\top}\|$ and $ \bGa_*=\cO(1)$, one gets immediately
from the expression (\ref{eq.Kirchpass}) of  $\bbI(\vvy;k)$ (with $N=1$)
that
$$
\| \bbI(\vvy;k) \| \leq   \cO\big(\|\mH(\vvy,\vvy_*;k) \|^2\big) .
$$
Using proposition \ref{prop.kirchfraun}, it is straightforward to derive
the following Fraunhofer asymptotic for the matrix $\mH(\vvy,\vvy_*;k)$ 
\begin{equation}\label{eq.asympA}
\mH(\vvy,\vvy_*;k)= \widetilde{\mH}(\vvy,\vvy_*;k)+  \cO\left(\frac{a^4\Theta_a}{L^4}\right)+ \cO\left(\frac{a^2\Theta_b}{L^2}\right),
\end{equation}
which leads to (\ref{eq.asymppas1}) by applying lemma \ref{lem.crossrange2} for the asymptotic of the matrix $ \widetilde{\mH}(\vvy,\vvy_*;k)$.

For the case of a disk array $\cA$ of radius $a$, the asymptotic formula
(\ref{eq.asymppas2}) of $\bbI(\vvy_{*};k)$ follows immediately from the
asymptotic (\ref{eq.asympA}) when $\vvy=\vvy_*$ and the lemma
\ref{lem.crossrange1} for the asymptotic of the matrix $
\widetilde{\mH}(\vvy_*,\vvy_*;k)$.
\end{proof}

\subsection{Cross-range estimations of the polarizability tensors}
\label{sec:cr:polamat}
We now study the estimation of the polarizability tensors of $N$
point-like scatterers from the
Kirchhoff imaging function. We derive an error estimate by assuming that
the positions $\vvy_1, \ldots, \vvy_N$ of the point-like scatterers are known. At
$\vvy_i$ we have the following identity
\begin{equation}\label{eq.identitypolapass}
  \mH(\vvy_i,\vvy_i;k) \, \bGa_i \, \mH(\vvy_i,\vvy_i;k)=\bbI(\vvy_i;k)-\sum_{i\neq j}\mH(\vvy_i,\vvy_j;k) \, \bGa_j \, \mH(\vvy_i,\vvy_j;k)^{\top}
\end{equation}
where we use that $\mH(\vvy_i,\vvy_i;k)=\mH(\vvy_i,\vvy_i;k)^{\top}$, 
from \eqref{eq.defmatA} and the fact that $\mG(\vvx_r,\vvy_i;k)$ and $\overline{\mG(\vvx_r,\vvy_i;k)}$ commute.
Hence, if the point-like scatterers are distant enough, we expect the coupling term
(i.e. the second term in the right hand side of
\eqref{eq.identitypolapass}) to be small. Neglecting the coupling terms,
$\bGa_i$ can be obtained by solving the linear system
\begin{equation}\label{eq.systpolpass}
\mH(\vvy_i,\vvy_i;k) \, \bGa_i \, \mH(\vvy_i,\vvy_i;k)=\bbI(\vvy_i;k).
\end{equation}
In the Fraunhofer regime, this system is ill-conditioned. Indeed, from
formula \eqref{eq.asymptA}, it follows that the matrix $\mH(\vvy_i,\vvy_i;k)$ is ill-conditioned.
Hence, as for the passive imaging problem, one cannot retrieve all
entries of the polarizability tensor $\bGa_*$. To be more precise, the asymptotic formula 
\eqref{eq.asymptA} shows that in this regime the blocks $\mH_{1:2,3}$,
$\mH_{3,1:2}$ and $\mH_{3,3}$ are asymptotically smaller compared to the block
$\mH_{1:2;1:2}$. The block $\mH_{1:2;1:2}(\vvy_i,\vvy_i,k)$ can be
approximated by $(a^2/L^2)\mI$, where $\mI$ is the $2\times 2$ identity.
Hence we can stably extract from \eqref{eq.identitypolapass} the block
$\bGa_{*;1:2,1:2}$ of the polarization tensor. We give an error estimate
for this procedure in the following proposition.

\begin{proposition}
In the Fraunhofer
asymptotic regime, the system
\begin{equation}\label{eq.systpblocalpha}
\mH_{1;2,1;2}(\vvy,\vvy,k) \, \bGa_{1:2,1:2} \,
\mH_{1;2,1;2}(\vvy,\vvy,k)=\mathbb{I}_{\text{KM};1:2,1;2}(\vvy;k)
\end{equation}
is invertible and its solution is
\begin{equation}\label{eq.systalpha} 
\bGa_{1:2,1:2}=
\mH_{1;2,1;2}(\vvy,\vvy,k)^{-1}\mathbb{I}_{\text{KM};1:2,1;2}(\vvy;k)
\mH_{1;2,1;2}(\vvy,\vvy;k)^{-1}.
\end{equation}
If the array $\cA$ is a disk of radius $a$,
the estimated cross-range polarizability tensor $\bGa_{1:2,1:2}$ (which
depends on the imaging point $\vvy$) approximates the
exact one $\bGa_{i;1:2,1:2}$ in the following sense
\begin{itemize}
\item If the imaging point coincides with a dipole, i.e. $\vvy=\vvy_i$,
\begin{equation}\label{eq.asympsolsystalpha}
\| \bGa_{1:2,1:2} - \bGa_{i;1:2,1:2}\| = \cO \left( \Big( \frac{L} {a\,
k \min \limits_{j\neq i}\|\vy_i-\vy_j\|} \Big)^2\right)+   \cO\left(\frac{ b}{L}\right) + \cO\left(\frac{a^2\Theta_a}{L^2}\right)+ \cO(\Theta_b).
\end{equation}
\item If the imaging point does not coincide with a dipole in the range,
i.e. $\vy\neq \vy_j$ for all $j=1,\ldots,N$,
\begin{equation}\label{eq.drecraspolamat}
\| \bGa_{1:2,1:2}  \| = \cO \left( \Big( \frac{L} {a\, k \min
\limits_{j=1,\ldots,N}\|\vy-\vy_j\|}\Big)^2 \right)+ \cO\left(\frac{\Theta_a^2}{L^4}\right)+ \cO(\Theta_b^2) .
\end{equation}
\end{itemize}
\end{proposition}

\begin{proof}
For simplicity we omit the dependence in $k$ in this proof.  In
Fraunhofer asymptotic regime, $\mH_{1;2,1;2}(\vvy,\vvy)$ is invertible
with $\|\mH_{1;2,1;2}(\vvy,\vvy)^{-1}\| = \cO(L^2/a^2)$ (see proposition
\ref{prop.polacrossactiv}). Thus the solution $\bGa_{1:2,1:2}$ of
\eqref{eq.systpblocalpha} is given by formula \eqref{eq.systalpha}.

In the case $\vvy=\vvy_i$, the identity \eqref{eq.identitypolapass}  and the
asymptotic \eqref{eq.asymptA} of the matrix
$\mH_{1:2,1:2}(\vvy_i,\vvy_i)$, gives after a short calculation 
that the first block $\bGa_{i;1:2,1:2}$ of the exact polarizability
tensor $\bGa_i$ satisfies the following system
\begin{flalign}\label{eq.relationalpha12}
& &&  \mH_{1:2,1:2}(\vvy_i,\vvy_i) \, \bGa_{i;1:2,1:2} \,
\mH_{1:2,1:2}(\vvy_i,\vvy_i) \nonumber\\ & &&  =\mathbb{I}_{\text{KM};1:2,1;2}(\vvy_i)-\Big[ \sum_{j\neq i} \mH(\vvy_i,\vvy_j) \, \bGa_j \, \mH(\vvy_i,\vvy_j)^{\top} \Big]_{1:2,1:2}  
+  \cO\Big(\frac{a^4 b}{L^5}\Big)+\Big(\frac{a^6\Theta_a}{L^6}\Big)+ \cO\Big(\frac{a^4\Theta_b}{L^4}\Big). \nonumber \\
\end{flalign}
We deduce that the difference between the
estimated $\bGa_{1:2,1:2}$ (satisfying \eqref{eq.systpblocalpha}) and
the true $\bGa_{i;1:2,1:2}$ (satisfying \eqref{eq.relationalpha12}) is
\begin{flalign}\label{eq.ineqestimalphai}
 &&&\|\bGa_{1:2,1:2} -\bGa_{i;1:2,1:2} \| \nonumber \\
  \leq && & \| \mH_{1:2,1:2}(\vvy_i,\vvy_i)^{-1} \|^2  \, \left[
\big\| 
\big[ \sum_{j\neq i} \mH(\vvy_i,\vvy_j) \, \bGa_j \,
\mH(\vvy_i,\vvy_j)^{\top}\big]_{1:2} \big\| +  \cO\Big(\frac{a^4
b}{L^5}\Big)+\cO\Big(\frac{a^6\Theta_a}{L^6}\Big)+ \cO\Big(\frac{a^4\Theta_b}{L^4}\Big) \right].  \nonumber
\end{flalign}
Finally, using the asymptotic formula \eqref{eq.asymppas1} to control
the terms in sum above we obtain
\[
 \|\bGa_{1:2,1:2} -\bGa_{i;1:2,1:2} \| =  \cO \Big( \Big(
 \frac{L} {a\, k \min \limits_{i\neq j}\|\vy_i-\vy_j\|} \Big)^2\Big) + \cO\left(\frac{ b}{L}\right) + \cO\left(\frac{a^2 \Theta_a}{L^2}\right)+ \cO(\Theta_b).
  \]
In the case $\vvy\neq \vvy_i$, we have
$$
\| \bGa_{i;1:2,1:2}  \| \leq   \| \mH_{1:2,1:2}(\vvy,\vvy)^{-1} \|^2 \,
\| \mathbb{I}_{\text{KM};1:2,1;2}(\vvy)\|.
$$
Thus, the asymptotic formula \eqref{eq.drecraspolamat} follows immediately
from the decay of the imaging function  $\mathbb{I}_{KM;1:2,1;2}(\vvy)$
given by \eqref{eq.asymppas1}.
\end{proof}
The asymptotic \eqref{eq.asympsolsystalpha} shows that one obtains a
good reconstruction of the cross-range polarizability tensor
$\bGa_{i;1:2,1:2}$ by solving the linear system
\eqref{eq.systpblocalpha}. The error is given by a coupling
term (contribution from the other scatterers) and remainders which are
small in the Fraunhofer regime. This coupling term decreases as $(L/(a\,
k \min \limits_{i\neq j}\|\vy_i-\vy_j\|))^2$ which is faster than the
decrease in $L/(a\, k \min \limits_{i\neq j}\|\vy_i-\vy_j\|)$ we found
in passive imaging (see proposition \ref{prop.polacrossactiv}).
The asymptotic \eqref{eq.drecraspolamat} shows that using $\|\bGa_{1:2,1:2}\|$ as an
image (as it does depend on the imaging point $\vvy$) also leads to a
stable reconstruction of the scatterer's positions, with cross-range
resolution given by the Rayleigh criterion. 

\subsection{Range estimation of the position}
\label{sec:r:pola}
We study the resolution in range by considering one point-like scatterer
located at $\vvy_*=(\vy_*,L+\eta_*)$ with known cross-range position
$\vy_*$. We assume the array response function
$\Pi(\vx_r,\vx_s;\omega/c)$ is known for $\vx_r,\vx_s \in \cA$ and on
the frequency band $[\omega_0-B/2, \omega_0+ B/2]$, with bandwidth $B$.
The Kirchhoff imaging function over this band and at a point $\vvy$ is
obtained from the single frequency Kirchhoff imaging function
\eqref{eq.Kirchpass_general} by integrating over the frequency band
$$
\mathbb{I}_{\text{KM}}(\vvy)=\int_{|\omega-\omega_0|<B/2} \md \omega \,
\mathbb{I}_{\text{KM}}\Big(\vvy;\frac{\omega}{c}\Big).
$$ 
As in the passive case, we assume that the Fraunhofer asymptotics
(section \ref{sub.asympfraun}) hold uniformly over the frequency band
$[\omega_0-B/2, \omega_0+ B/2]$.  The next proposition shows that as in
the passive imaging case, the range resolution is $c/ B$. 

\begin{proposition}[Imaging function decrease in the range direction]
\label{eq.proprangmat}
When the array $\cA$ is a disk of radius $a$,
the Kirchhoff imaging function \eqref{eq.image}  of the dipole $\vvy_{*}=(\vy_{*},
L+\eta_{*})$ satisfies for all $\vvy=(\vy_*,L+\eta)$  with $\eta \neq
\eta_*$,
\begin{equation}\label{eq.asymprangepass1}
 \|\mathbb{I}_{\text{KM}}(\vvy)\| =    B  \frac{ a^4}{L^4} \left[  \cO \left(
 \frac{c}{  B   \left|\eta-\eta_{*} \right|}
 \right)+o(1)\right].
\end{equation}
At the dipole location we have
\begin{eqnarray}\label{eq.asymprangepass2}
\mathbb{I}_{\text{KM}}(\vvy_*)=   B\left[ \frac{  a^4}{(16 \pi)^2 L^4} \begin{pmatrix}  \bGa_{*;1:2,1:2}& \boldsymbol {0} \\  \boldsymbol {0} & 0 \end{pmatrix} +o\Big(\frac{ a^4}{L^4}\Big) \right].
\end{eqnarray}
\end{proposition}
\begin{proof}
We first consider the case where the imaging point
$\vvy=(\vvy_*,L+\eta)$ does not coincide with the dipole, i.e. $\eta
\neq \eta_*$. Using 
the asymptotic \eqref{eq.asymtAg} and the definition of $\widetilde{\mH}$ in \eqref{eq.matIcross}, we observe that
the imaging function \eqref{eq.Kirchpass} at $\vvy$
is
\begin{equation}\label{eq.rangepassifform}
\mathbb{I}_{\text{KM}}(\vvy)=   \int_{|\omega-\omega_0|<B/2}\md \omega
\,     \exp[ 2 \mi \omega(\eta_*-\eta)/c] \,  \mQ(\vvy, \vvy_*) \,
\bGa_{*} \, \mQ(\vvy, \vvy_*)^{\top}  + B o\Big(\frac{a^4}{L^4}\Big),
\end{equation}
where we used the notation
$$
 \mQ(\vvy, \vvy_*)= \frac{1}{4\pi L^2}\int_{\cA} \md\vx_r\, \mP(\vvx_r,\vvy)  \mP(\vvx_r,\vvy_*)=\cO \Big(\frac{a^2}{L^2}\Big).
$$
Integrating in frequency we get
\begin{equation}\label{eq.sincassymp}
 \int_{|\omega-\omega_0|<B/2}\md \omega \,     \exp[ 2 \mi
 \omega(\eta_*-\eta)/c] = B \exp[2 \mi \omega_0 (\eta_*-\eta)/c ]
 \operatorname{sinc}\left(\frac{B (\eta_*-\eta)}{c}\right).
\end{equation}
Finally combining this last relation with \eqref{eq.rangepassifform}
yields
\begin{equation}\label{eq.decreaseimageactifrange}
\mathbb{I}_{\text{KM}}(\vvy_*)=B \exp[2 \mi \omega_0 (\eta_*-\eta)/c ]
\operatorname{sinc}\left(\frac{B (\eta_*-\eta)}{  c}\right) \cO
\Big(\frac{a^4}{L^4}\Big)+ B o\Big(\frac{a^4}{L^4}\Big).
\end{equation}
The asymptotic formula \eqref{eq.asymprangepass1} follows.

When $\vvy=\vvy_*$, the asymptotic formula \eqref{eq.asymprangepass2}
follows immediately by integrating over the frequency the asymptotic
\eqref{eq.asymppas2}.
\end{proof}

\subsection{Polarizability tensor recovery in the range direction}
\label{sec:range:active}
From the asymptotic analysis of section~\ref{sec:cr:pola}, at a
frequency $\omega$ we can only expect to recover the cross-range
polarizability tensor $\bGa_{i;1:2,1:2}$ by solving the linear system
\eqref{eq.systalpha}. The solution $\bGa_{1:2,1:2}$ depends on both the
imaging point and the frequency $\omega$. Since we assume that the true
polarizability tensor does not depend on $\omega$ we propose to estimate
it by averaging the single frequency estimate over the frequency band,
i.e.
\begin{equation}\label{eq.solalphamultifreq}
\bGa_{1:2,1:2}=\frac{1}{ B}\int_{|\omega-\omega_0|<  B/2}\md \, \omega\, \bGa_{1:2,1:2} (\omega).
\end{equation}

As we see next in section~\ref{sec:polarange}, if the depth of the
scatterer is well-known, this procedure estimates well the
cross-range polarizability of the scatterer. However, as in the passive
case (see section~\ref{sec:rangepol:passive}), the image oscillates in
depth. These oscillations can be characterized and suppressed as we show
in section~\ref{sec:suppression:active}.

\subsubsection{Analysis of polarizability tensor image in range}
\label{sec:polarange}

In the next proposition, we isolate the effect of depth by considering
scatterers that are aligned in range. We show that the depth resolution
of the cross-range polarizability tensor image is $c/B$, i.e. identical
to the passive case (section~\ref{sec:range:passive}). If the imaging
point coincides with the scatterer position, the error in estimating the
cross-range polarizability tensor remains small, provided the scatterers
are well separated, i.e.  $|\eta_i-\eta_j|$ is large with respect to the
range resolution $c/B$. 

\begin{proposition}\label{prop.polarange.mat}
When the array $\cA$ is a disk of radius $a$ and the dipoles are all
aligned the range direction of $\cA$, the image of the cross-range
polarizability tensor $\bGa_{1:2,1:2}$ (given by
\eqref{eq.solalphamultifreq}) satisfies the two following estimates:
\begin{itemize}
\item If the imaging point is the dipole location, i.e. $\vvy=\vvy_i$,
we have
\begin{equation}\label{eq.rangepolaretensorcover}
\|\bGa_{1:2,1:2}-\bGa_{i;1:2,1:2}\| = \cO \Big(  \frac{c} {B \min \limits_{i\neq j}|\eta_i-\eta_j|}\Big)+o(1),
\end{equation}
\item If the imaging point range is different from any of the dipole
ranges, $\vvy=(\vy_*,L+\eta)\neq \vvy_j$ (for all $j=1,\ldots,N$),
\begin{equation}\label{eq.drecraspolatensprrange}
\|\bGa_{1:2,1:2}\| =   \cO \Big(  \frac{c} {B \min \limits_{j=1,\ldots,N}|\eta-\eta_j|} \Big)+o(1).
\end{equation}
\end{itemize}
\end{proposition}
\begin{proof}
We start with the case where $\vvy = \vvy_i$. To shorten notation we use
$\mH_{1:2,1:2}(\omega/c)$ instead of
$\mH_{1:2,1:2}(\vvy_i,\vvy_i;\omega/c)$. From the definition of the
cross-range polarizability tensor image \eqref{eq.solalphamultifreq} and
our assumption that the exact cross-range polarizability is independent
of frequency we
get
\[
\bGa_{1:2,1:2}-\bGa_{i;1:2,1:2}=\frac{1}{B}\int_{|\omega-\omega_0|<
B/2}\md \, \omega (\bGa_{1:2,1:2} (\omega) -  \bGa_{i;1:2,1:2}).
\]
To estimate the difference $\bGa_{1:2,1:2} (\omega) -  \bGa_{i;1:2,1:2}$
we use \eqref{eq.systpblocalpha} and \eqref{eq.relationalpha12} to obtain
\begin{flalign}
&&&\bGa_{1:2,1:2} (\omega) -  \bGa_{i;1:2,1:2} \nonumber\\
&&& =  \mH_{1:2,1:2}\Big(\frac{\omega}{c}\Big)^{-1}\Big( \big[ -\sum_{j\neq i} \mH(\vvy_i,\vvy_j;\frac{\omega}{c}) \, \bGa_j \, \mH(\vvy_i,\vvy_j;\frac{\omega}{c})^{\top} \big]_{1:2,1:2}+o\big(\frac{a^4}{L^4}\big) \Big) \mH_{1:2,1:2}\Big(\frac{\omega}{c}\Big)^{-1} .\nonumber
\end{flalign}
Now recalling that $\mH_{1:2,1:2}(\omega/c)^{-1}=(4 \pi \, L)^2/a^2+o(L^2/a^2)$, one gets that
\begin{flalign*}\label{eq.estimaterangepola}
&&& \bGa_{1:2,1:2}-\bGa_{i;1:2,1:2} \nonumber\\ 
&&&=\frac{(4 \pi L)^4}{a^4 B}\int_{|\omega-\omega_0|<  B/2} \md \omega \big[ -\sum_{j\neq i} \mH(\vvy_i,\vvy_j;\frac{\omega}{c}) \, \bGa_j \, \mH(\vvy_i,\vvy_j;\frac{\omega}{c})^{\top} \big]_{1:2,1:2}  +o(1) .
\end{flalign*}
Then, since the dipoles are aligned, one can use the relation
\eqref{eq.decreaseimageactifrange} to rewrite the coupling term
$\sum_{j\neq i} \mH(\vvy_i,\vvy_j;\omega/c) \, \bGa_j \,
\mH(\vvy_i,\vvy_j;\omega/c)^{\top}$ in terms of a $\operatorname{sinc}$ to get
\begin{equation*}
\bGa_{1:2,1:2}-\bGa_{i;1:2,1:2}=\frac{-(4 \pi L)^4}{a^4} \sum_{j\neq i}
e^{2 \mi \omega_0 (\eta_j-\eta_i)/c } \operatorname{sinc}\left(\frac{B
(\eta_j-\eta_i)}{ c}\right) O\big(\frac{a^4}{L^4}\big)+o(1).
\end{equation*}
This gives the asymptotic \eqref{eq.rangepolaretensorcover}.
The asymptotic \eqref{eq.drecraspolatensprrange} can be proved in a
similar way.
\end{proof}

\subsubsection{Suppression of oscillatory artifact in depth}
\label{sec:suppression:active}
The image of the cross-range polarizability oscillates depth. To see
this, consider the reconstruction formula \eqref{eq.solalphamultifreq}
for a single dipole ($N=1$) located at $\vvy_*$ and with polarizability
tensor $\bGa_*$ for imaging points $\vvy=(\vvy_*,L+\eta_*)$ in the range of the dipole $\vvy_*$. The multi-frequency estimate of the cross-range
polarizability tensor is
\[
\bGa_{1:2,1:2}= \frac{1}{B}\int_{|\omega-\omega_0|< B/2} \md \, \omega \,
\mH_{1:2,1:2}\Big(\vvy, \vvy;\frac{\omega}{c}\Big)^{-1}
\big[\mathbb{I}_{\text{KM}}(\vvy,\frac{\omega}{c})]_{1:2,1:2} \mH_{1:2,1:2}\Big(\vvy, \vvy;\frac{\omega}{c}\Big)^{-1}.
\]
Using the asymptotic $\mH_{1:2,1:2}(\vvy, \vvy;\omega/c)^{-1}=(4 \pi \,
L)^2/a^2+o(L^2/a^2)$ and the asymptotic \eqref{eq.rangepassifform} and \eqref{eq.sincassymp} for  $[\mathbb{I}_{\text{KM}}(\vvy,\frac{\omega}{c})]_{1:2,1:2}$  in the range direction of $\vvy_*$, this becomes:
\begin{equation*}
\bGa_{1:2,1:2}
= \frac{(4 \pi L)^4}{a^4 } e^{2 \mi \omega_0 (\eta_*-\eta)/c }\operatorname{sinc}\left(\frac{B (\eta_*-\eta)}{c}\right) \mQ(\vvy, \vvy_*) \, \bGa_{*} \, \mQ(\vvy, \vvy_*)^{\top} +o(1) \, ,
\end{equation*}
for $\vvy=(\vvy_*,L+\eta_*)$. The presence of the complex exponential and the $\operatorname{sinc}$ causes the image
$\cI(\vy_*,L+\eta)$ to oscillate in $\eta$ (twice faster than in passive imaging see \eqref{eq:sincphase}). Thus, if one does not know exactly the cross-range position $L+\eta_*$ of the dipole, one cannot reconstruct in a stable way $\bGa_{1:2,1:2}$. To deal with this artifact, we fix the phase of one component of
$\bGa_{1:2,1:2}$. In the following, we have arbitrarily chosen to
enforce that the $1,1$ entry be real and positive, i.e.
$\arg(\bGa_{1,1}) =0$. This can be achieved by post-processing the
solution of \eqref{eq.systmultifreq} by the operation
$(\overline{\bGa_{1,1}}/|\bGa_{1,1}|) \, \bGa_{1:2,1:2}$.  If
$|\bGa_{1,1}|$ is small, this operation can be problematic and we can
instead fix the phase of another entry of $\bGa$. One can also
regularize using $(\overline{\bGa_{1,1}}/(|\bGa_{1,1}|+\delta))
\bGa_{1:2,1:2}$, for a small $\delta>0$.

\subsection{Numerical experiments}
\label{sec:numexp:active}

\subsubsection{Experiments without noise}  

Figures \ref{fig.crosrrangeactif}, \ref{fig.rangeactif},
\ref{fig.crossrangeellipse} and \ref{fig.rangeellipse} illustrate the
case of three point-like scatterers in a configuration identical to that in
figures \ref{fig.crosrrange1}, \ref{fig.crosrrange2}  and
\ref{fig.range1}. The dipoles are located at $(-7 \lambda_0,\,
7\lambda_0 ,\, L)$, $(7 \lambda_0,\, 7\lambda_0 ,\, L)$ and
$(2\lambda_0,-2\lambda_0 , \, L+7\lambda_0)$.  Their polarizability
tensors are $3\times 3$ symmetric complex matrices, with entries given in
appendix~\ref{app:matrices}. We use the
reconstruction formula \eqref{eq.solalphamultifreq} to recover the
$2\times 2$ block $\bGa_{i;1:2,1:2}$ of each dipole. The Frobenius norms
of the cross-range polarization tensors are approximately 3.9, 3.5 and
4.5, respectively.

In figures \ref{fig.crosrrangeactif} and \ref{fig.rangeactif}, we
visualize $\|\bGa_{i;1:2,1:2}\|$ in the cross-range (figure
\ref{fig.crosrrangeactif}) and range (figure \ref{fig.rangeactif}) of
the different dipoles. We use ellipses to visualize the $2\times 2$ real
symmetric matrices
$\operatorname{Re}((\overline{\bGa_{1,1}}/|\bGa_{1,1}|) \,
\bGa_{1:2,1:2})$ and
$\operatorname{Im}((\overline{\bGa_{1,1}}/|\bGa_{1,1}|) \,
\bGa_{1:2,1:2})$. The principal axis of the ellipses represent the
orientation and magnitude of the matrices. The ellipses are plotted at
the scatterer cross-range positions (figure \ref{fig.crosrrangeactif})
and range positions (figure \ref{fig.rangeactif}). The white and yellow
ellipses represent respectively the real part and imaginary part of the
exact cross-range polarizability tensors. The dashed black and purple
ellipses are those associated with the reconstructed quantities (real
and imaginary, respectively).

We observe that the Frobenius norm  $\|\bGa_{i;1:2,1:2}\|$ is well
estimated for all dipoles in both range and cross-range. The focal spot
is consistent with the Rayleigh criterion $\lambda_0 L/a=5 \lambda_0$ in
the cross-range (figure \ref{fig.crosrrangeactif}) and with $2\pi
c/B=\lambda_0$ in the range (figure \ref{fig.rangeactif}). Furthermore,
in figure \ref{fig.crosrrangeactif}, the image decays away from the
scatterer positions faster than in passive imaging (see figures
\ref{fig.crosrrange1}, \ref{fig.crosrrange2}  and \ref{fig.range1})
which confirms the asymptotics \eqref{eq.crossrangeplarecover} and
\eqref{eq.drecraspolamat}. Finally the $2\times 2$ matrices
$\operatorname{Re}((\overline{\bGa_{1,1}}/|\bGa_{1,1}|) \,
\bGa_{1:2,1:2})$ and
$\operatorname{Im}((\overline{\bGa_{1,1}}/|\bGa_{1,1}|) \,
\bGa_{1:2,1:2} )$ are well reconstructed at the scatterer positions.
Their corresponding ellipses agree in both range and cross-range.

To show the stability of the reconstructions in the focal spot we
display in figures \ref{fig.crossrangeellipse} and
\ref{fig.rangeellipse} the corresponding ellipses at locations (in range
and cross-range) around the focal spot of one particular dipole. These
images (and figures~\ref{fig.crossrangeellipsebruit},
\ref{fig.rangeellipsebruit} and \ref{fig.cubictarget2}) were obtained
with a modified version of the \verb|plot_tensor_field| function in
\cite{Peyre}. We emphasize that the image in range
(figure~\ref{fig.rangeellipse}) is stable in depth because we used the
correction explained in section~\ref{sec:suppression:active}. In figures
\ref{fig.crossrangeellipse} and \ref{fig.rangeellipse}, we use ellipses
to visualize the matrices
$\operatorname{Re}((\overline{\bGa_{1,1}}/|\bGa_{1,1}|) \,
\bGa_{1:2,1:2})$ (left) and
$\operatorname{Im}((\overline{\bGa_{1,1}}/|\bGa_{1,1}|) \,
\bGa_{1:2,1:2} )$ (right) in the vicinity of the cross-range (figure
\ref{fig.crossrangeellipse}) and range position (figure
\ref{fig.rangeellipse}) of the dipole located at  $(7 \lambda_0,\,
7\lambda_0 ,\, L)$. The color of the ellipses is proportional to the
Frobenius norm of $\|\bGa_{i;1:2,1:2}\|$. Comparing these ellipses to
the exact ones of figures \ref{fig.crosrrangeactif} and
\ref{fig.rangeactif} confirms that the correction suppresses the 
oscillation in depth and gives a stable reconstruction of $
\bGa_{1:2,1:2}$, up to a complex sign.

\begin{figure}[h!]
\begin{center}
\includegraphics[width=6.2cm]{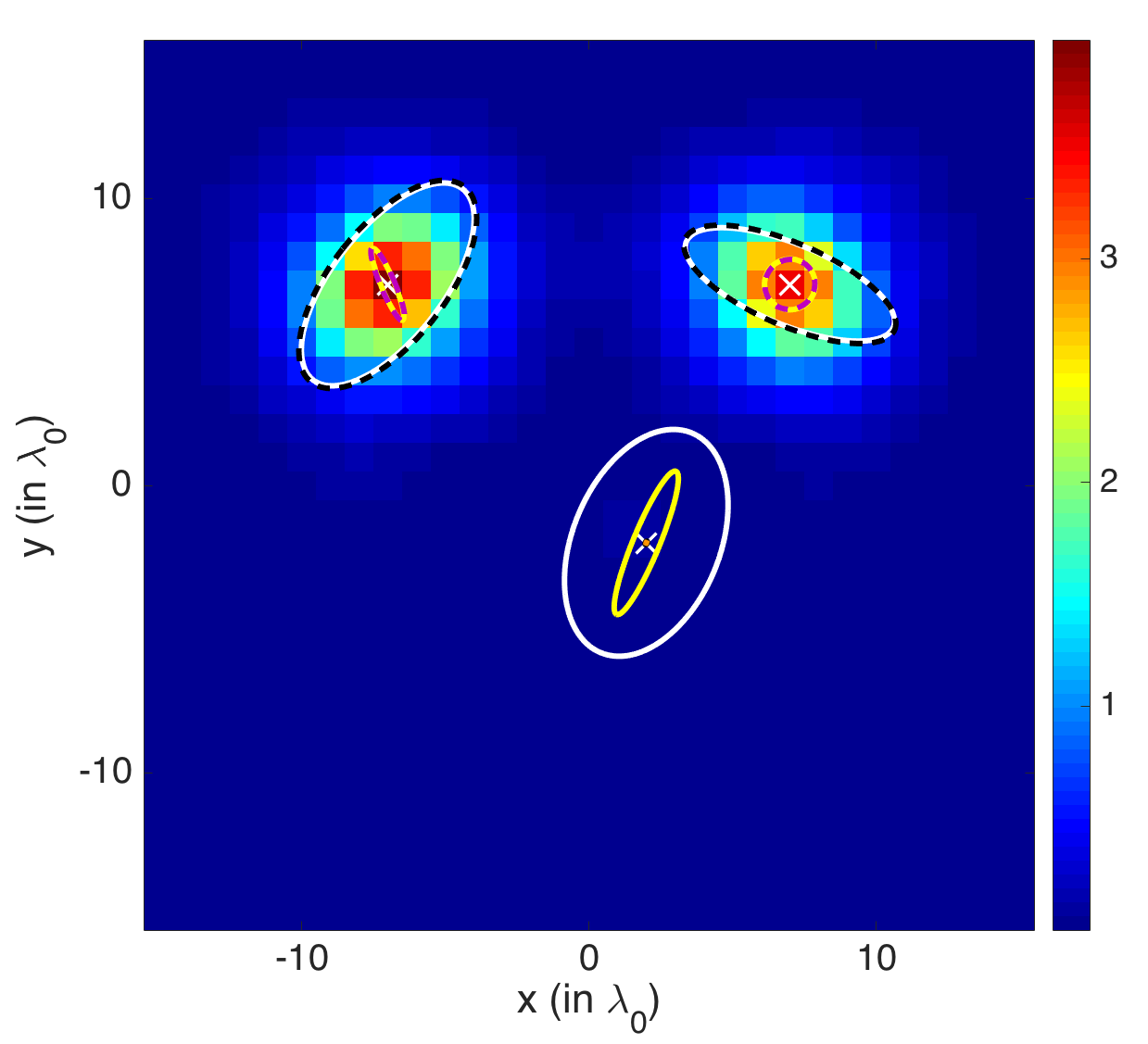}
\hspace{-0.1cm}
\includegraphics[width=6.1cm]{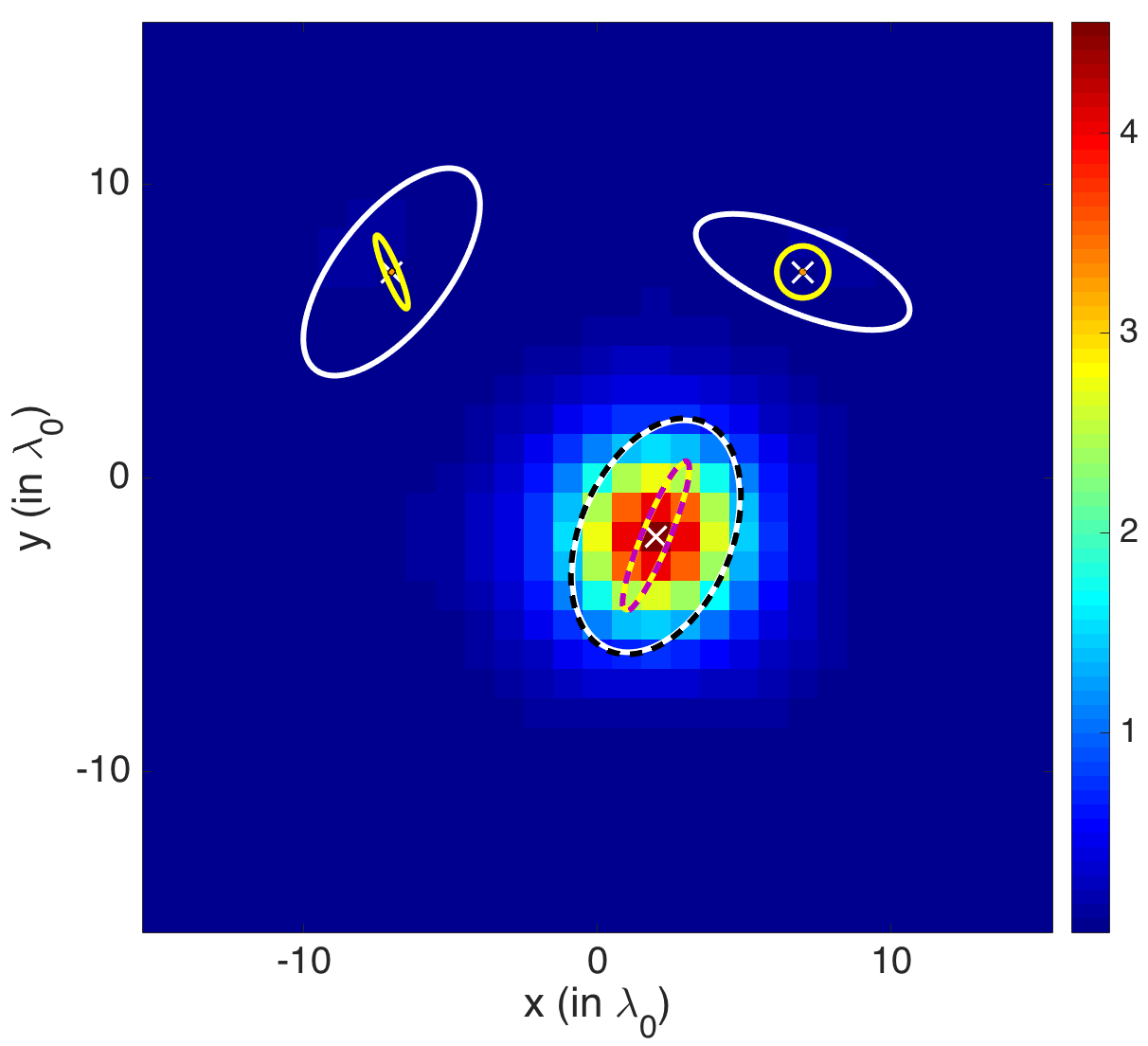}
\end{center}
\caption{Image of $|\bGa_{1:2,1:2}|$ (color scale). Visualization of
$\operatorname{Re}((\overline{\bGa_{1,1}}/|\bGa_{1,1}|) \,
\bGa_{1:2,1:2})$  with ellipses (white for the exact one and black for
the reconstructed one) and of
$\operatorname{Im}((\overline{\bGa_{1,1}}/|\bGa_{1,1}|) \,
\bGa_{1:2,1:2})$ (yellow ellipse for the exact one and purple ellipse
for the reconstructed one) at the scatterer cross-range position in the plane $z=L$ (left) and the plane $z=L+7\lambda_0$ (right).}
\label{fig.crosrrangeactif}
\end{figure}

\begin{figure}[h!]
\begin{center}
\includegraphics[width=6.25cm]{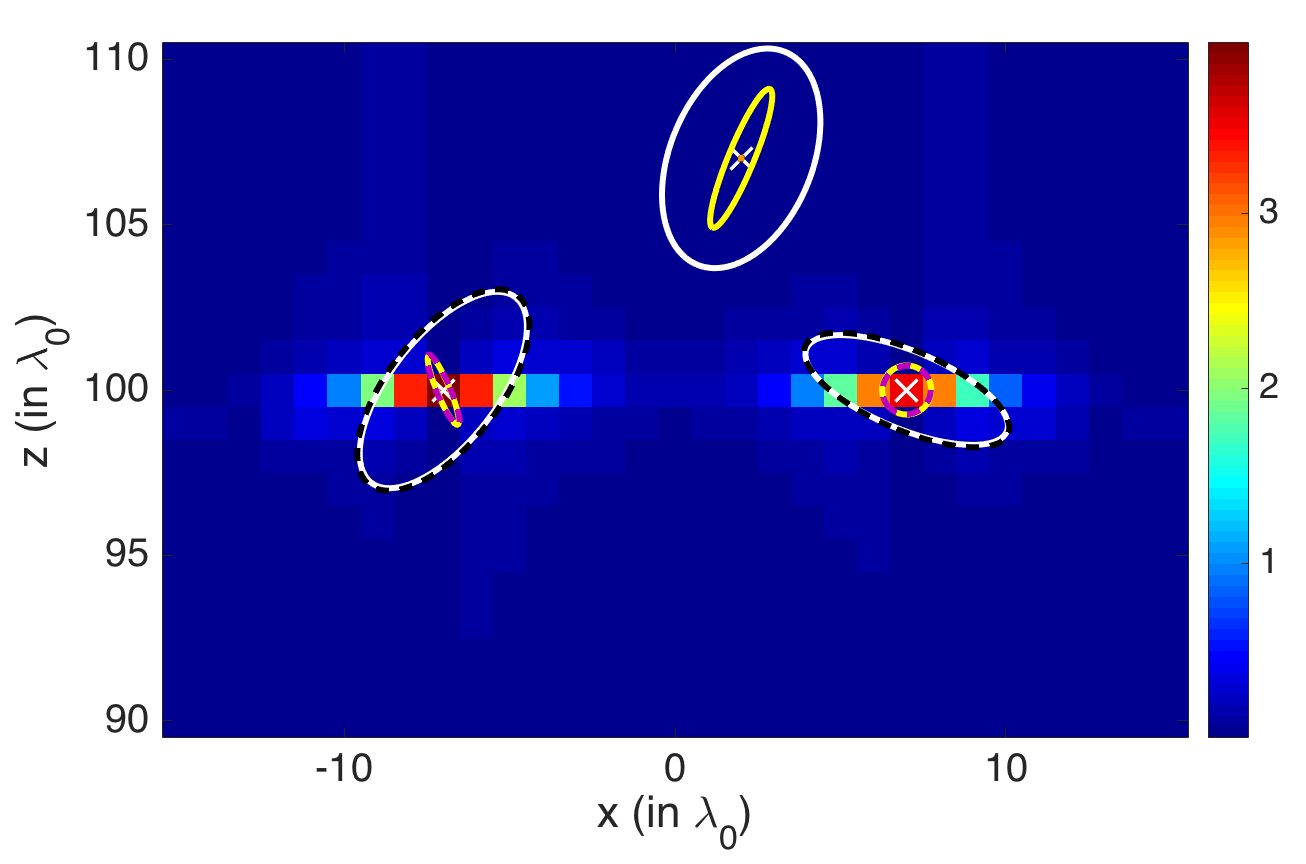}
\hspace{-0.1cm}
\includegraphics[width=6.25cm]{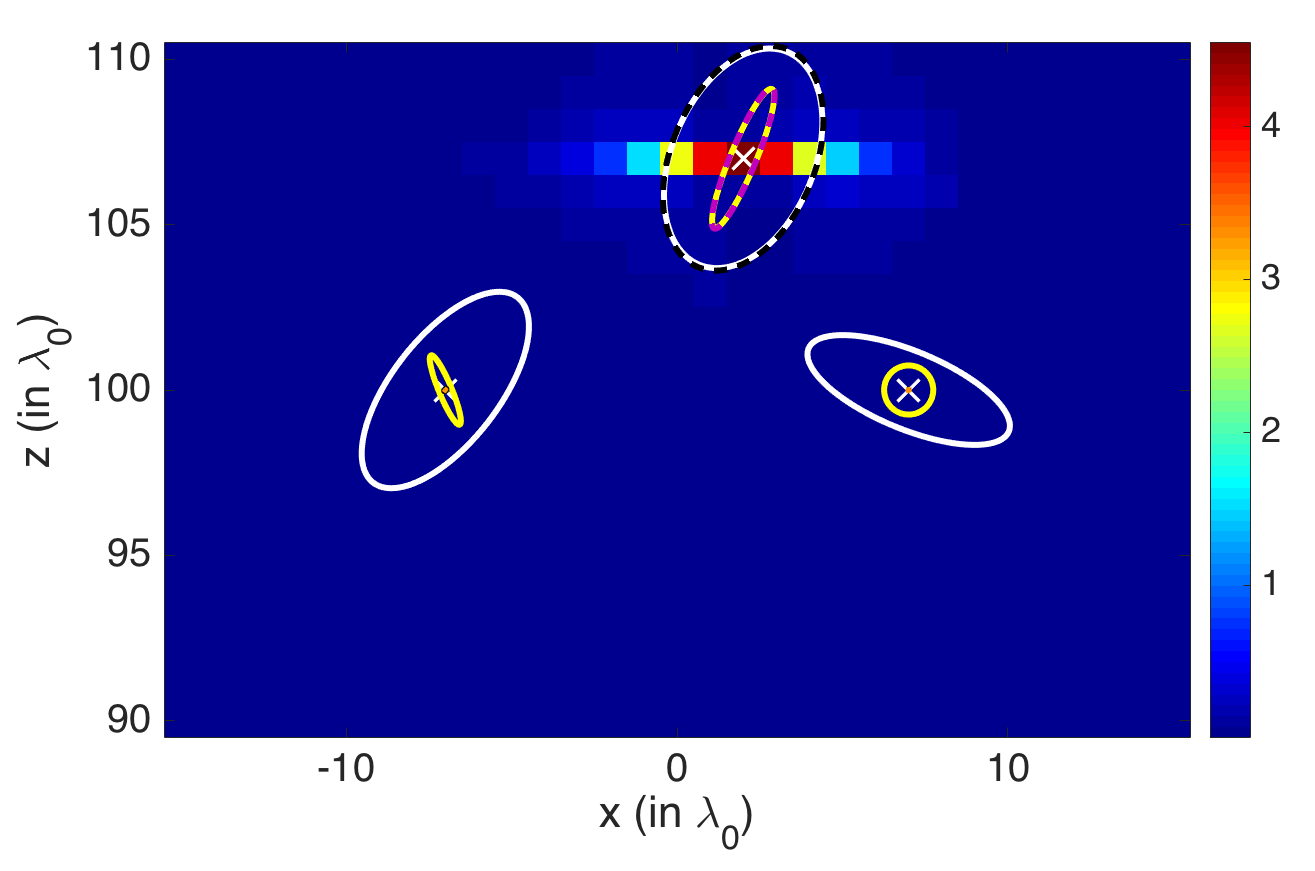}
\end{center}
\caption{Image of $|\bGa_{1:2,1:2}|$ (color scale). Visualization of
$\operatorname{Re}((\overline{\bGa_{1,1}}/|\bGa_{1,1}|) \,
\bGa_{1:2,1:2})$ with ellipses (white for the exact one and black for
the reconstructed one) and of
$\operatorname{Im}((\overline{\bGa_{1,1}}/|\bGa_{1,1}|) \,
\bGa_{1:2,1:2})$ (yellow ellipse for the exact one and purple ellipse
for the reconstructed one) at the scatterer range position in the plane
$y=0$ (left) and the plane $y=-2\lambda_0$ (right).}
\label{fig.rangeactif}
\end{figure}

\begin{figure}[h!]
\begin{center}
\includegraphics[width=6.1cm]{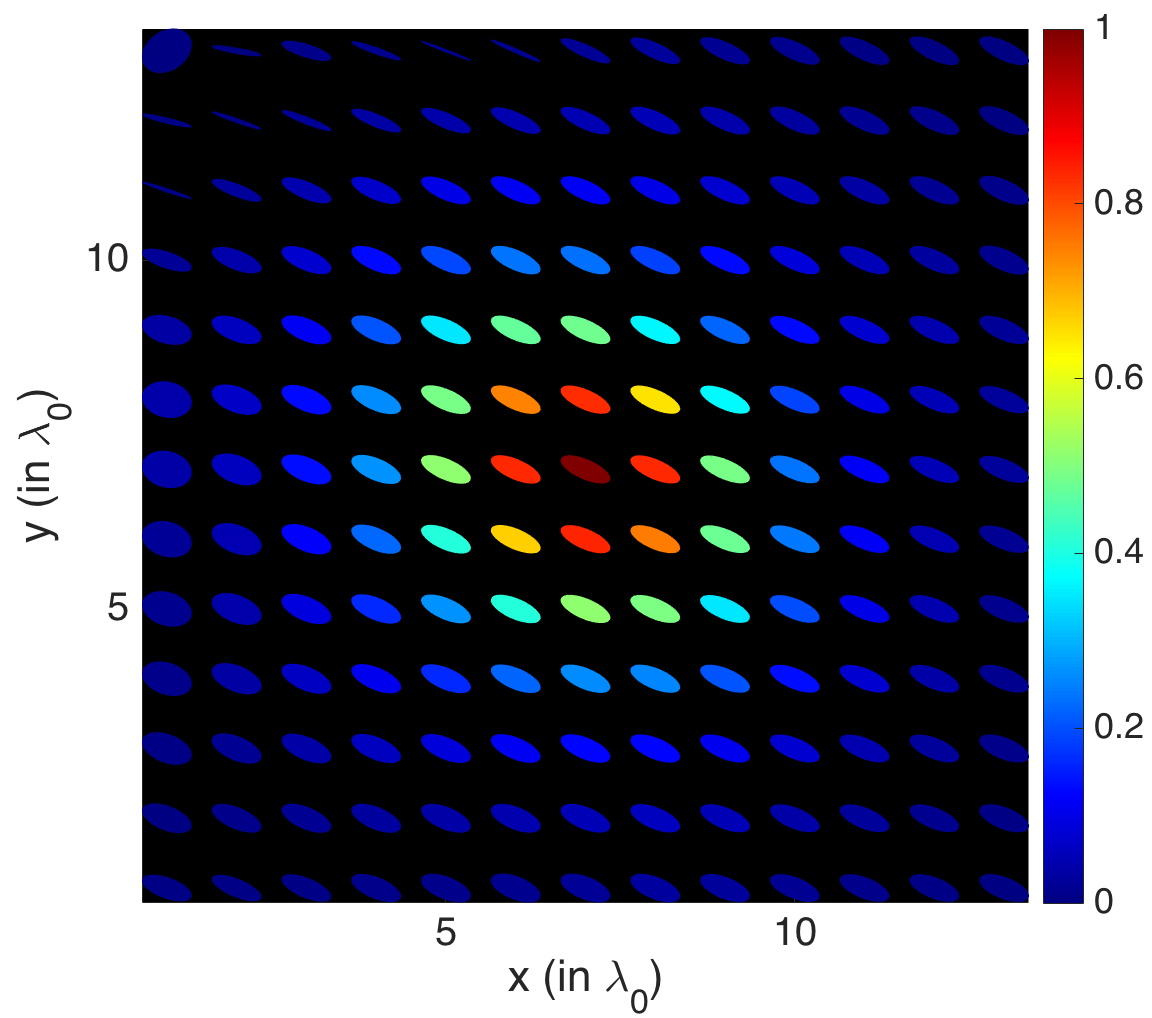}
\hspace{-0.1cm}
\includegraphics[width=6.1cm]{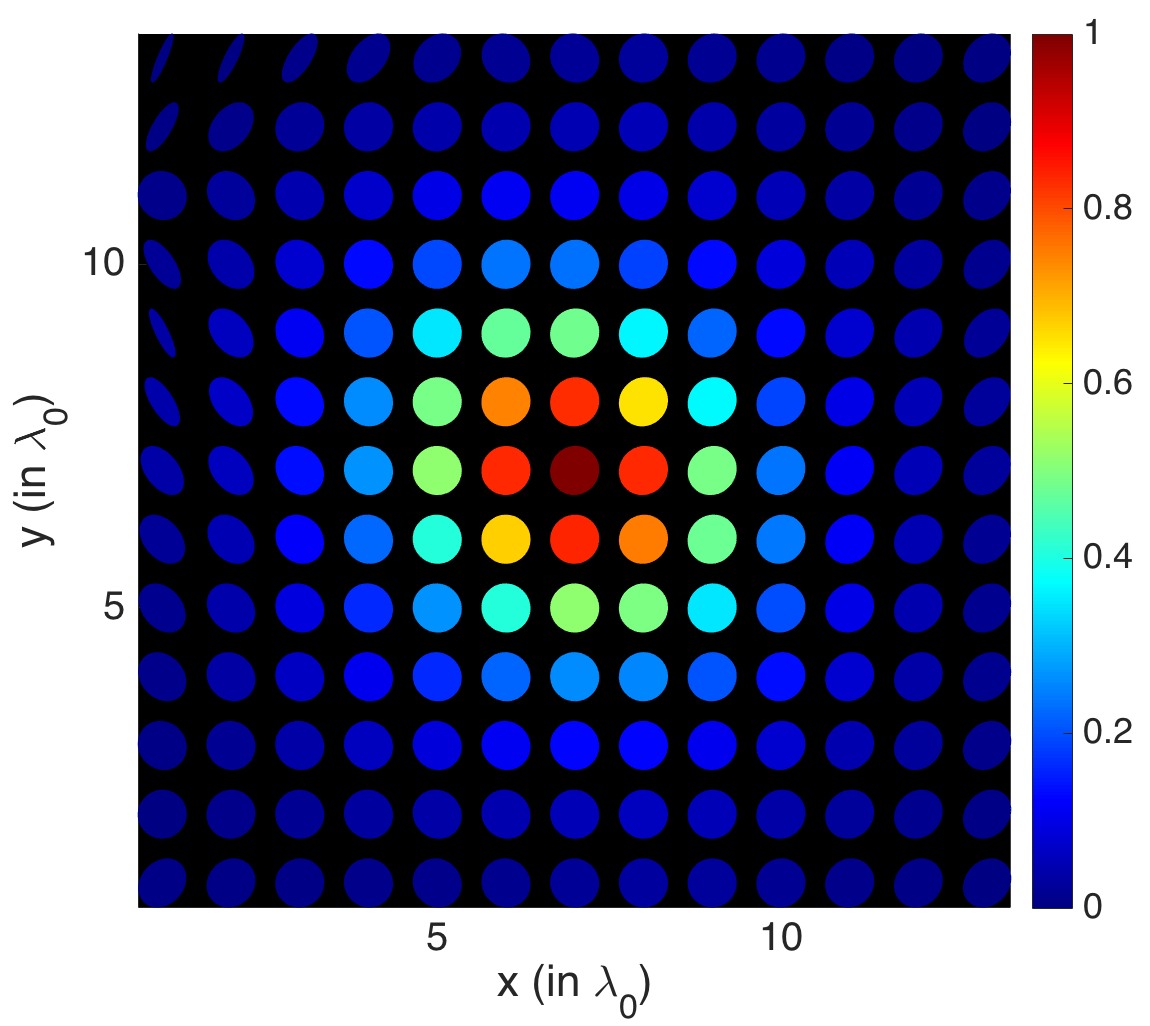}
\end{center}
\caption{Visualization of
$\operatorname{Re}((\overline{\bGa_{1,1}}/|\bGa_{1,1}|) \,
\bGa_{1:2,1:2})$ (left) and
$\operatorname{Im}((\overline{\bGa_{1,1}}/|\bGa_{1,1}|) \,
\bGa_{1:2,1:2})$ (right) in the vicinity of the dipole placed at
$(7\lambda_0,7\lambda_0,L)$ in the plane $z=L$.}
\label{fig.crossrangeellipse}
\end{figure}

\begin{figure}[h!]
\begin{center}
\includegraphics[width=6.1cm]{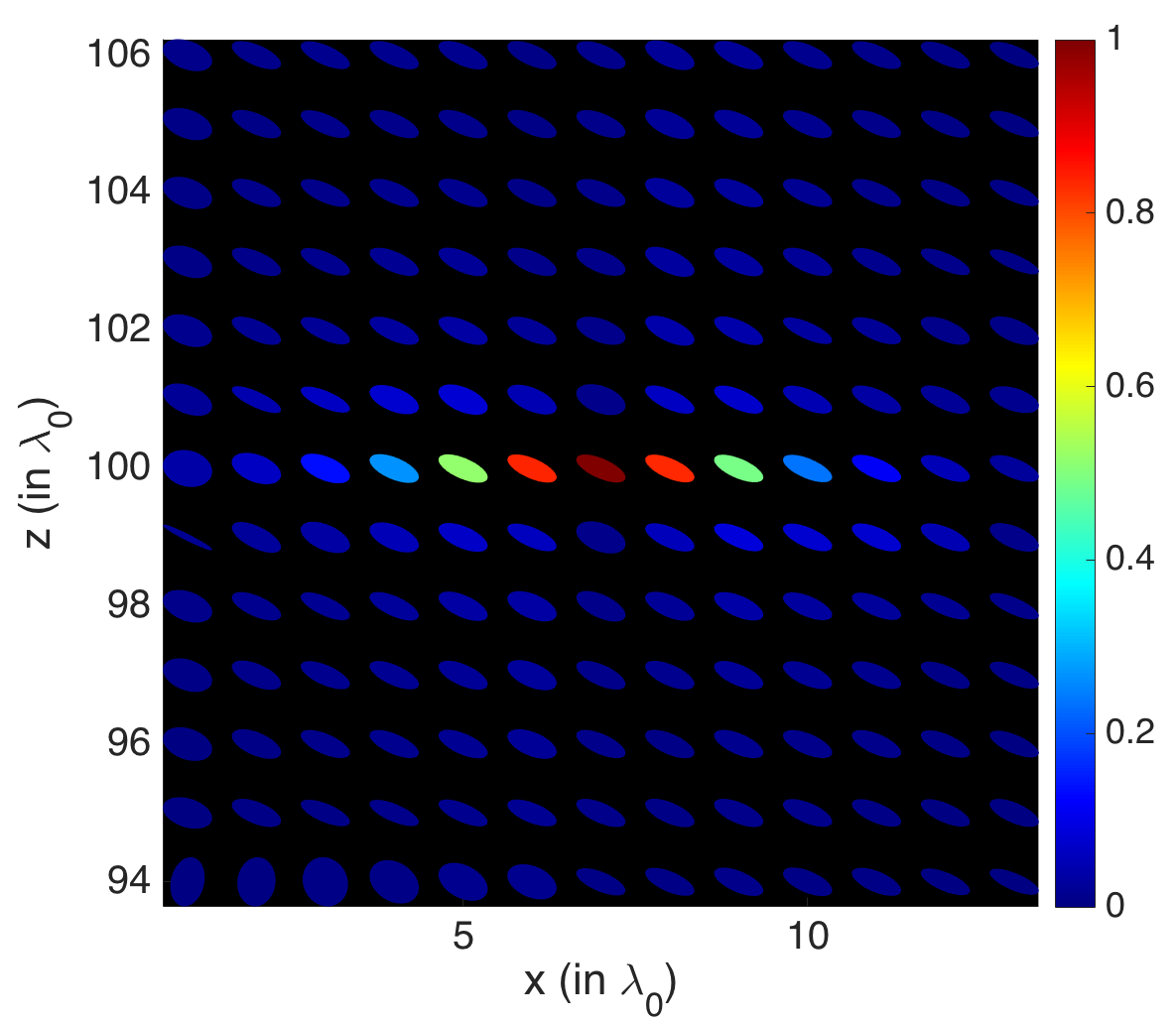}
\hspace{-0.1cm}
\includegraphics[width=6.1cm]{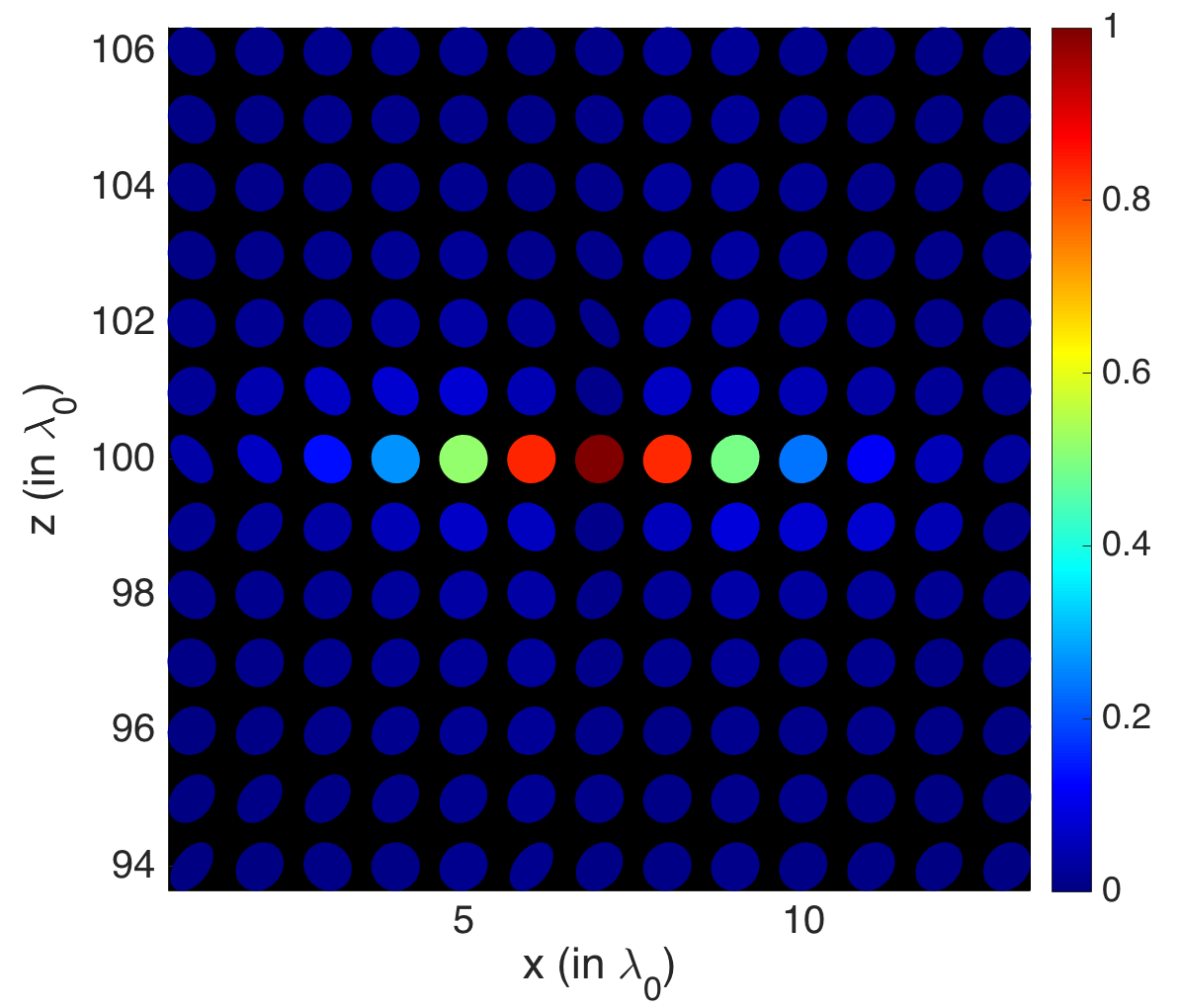}
\end{center}
\caption{Visualization of
$\operatorname{Re}((\overline{\bGa_{1,1}}/|\bGa_{1,1}|) \,
\bGa_{1:2,1:2})$ (left) and
$\operatorname{Im}((\overline{\bGa_{1,1}}/|\bGa_{1,1}|) \,
\bGa_{1:2,1:2})$ (right) in the vicinity of the dipole placed at
$(7\lambda_0,7\lambda_0,L)$ in the plane $y=7\lambda_0$.}
\label{fig.rangeellipse}
\end{figure}

\subsubsection{Experiments with noise} 
For this experiment, the array $\cA$ is a square containing $N\times N$
collocated sources and receivers and that we have $N_{\mathrm{freq}}$
equally spaced frequency samples in the band $\omega_0 + [-B/2,B/2]$.
To simulate noise in the measurements, we add a $3N\times 3N$ matrix
$\mathcal{W}(\omega_n/c)$ with zero mean uncorrelated Gaussian
distributed entries $\mathcal{N}(0,\epsilon p_{\rm{avg}})$ to the data
$\Pi(\omega_n/c)$ defined in \eqref{eq.datapassif}, where
$n=1,\ldots,N_{\text{freq}}$. The matrices for each frequency are
uncorrelated.  As in \cite{Borcea:2008:EII}, we choose to fix the level
of noise $\epsilon$ with respect to the average power $p_{\rm{avg}}$
received by the array of the signal on frequency band, namely:
$$
p_{\rm{avg}}=\frac{1}{(3 N)^2
\,N_{\mathrm{freq}}}\sum_{n=1}^{N_{\mathrm{freq}}}
\left\|\Pi\left(\frac{\omega_n}{c}\right)\right\|^2,
$$
where $\|\cdot\|$ stands for the Frobenius norm. The expected noise power is given by
$$
p_{\mathrm{noise}}=E\left[\sum_{n=1}^{N}
\left\|\mathcal{W}\left(\frac{\omega_n}{c}\right)\right\|^2\right]= (3 N)^2 \,N_{\mathrm{freq}} \,\epsilon \, p_{\rm{avg}}.
$$
The signal-to-noise ratio (SNR) in decibels (dB) is then
$$
\mathrm{SNR}=10 \operatorname{log}_{10}\left(
\frac{p_{\mathrm{noise}}}{p_{\mathrm{avg}}} \right)=10  \operatorname{log}_{10}(\epsilon).
$$

Figures \ref{fig.crossrangeellipsebruit} and
\ref{fig.rangeellipsebruit} reproduce the numerical experiments of 
figures \ref{fig.crossrangeellipse} and \ref{fig.rangeellipse} but with
a SNR of $10$dB (i.e. the noise power is 10 times larger than the signal
power). With noise present, the ellipse orientations agree with the
true ones at the focal spot, but are significantly different away from
the focal spot, where the image magnitude is also small. We conclude
that  Kirchhoff imaging is robust to relatively large measurement noise
(as in acoustics, see e.g. \cite{Borcea:2008:EII}).
\begin{figure}[h!]
\begin{center}
\includegraphics[width=6.05cm]{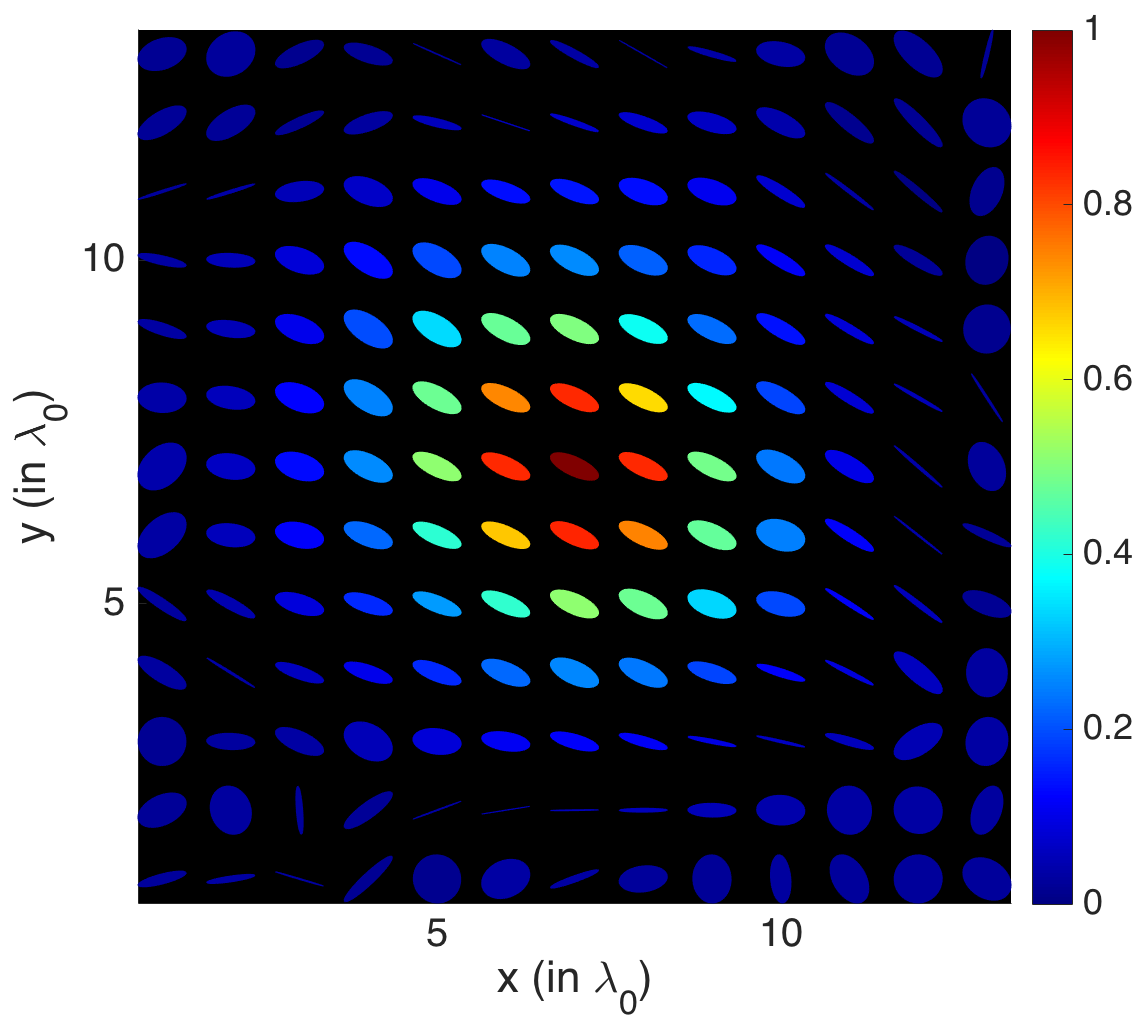}
\hspace{-0.1cm}
\includegraphics[width=6.1cm]{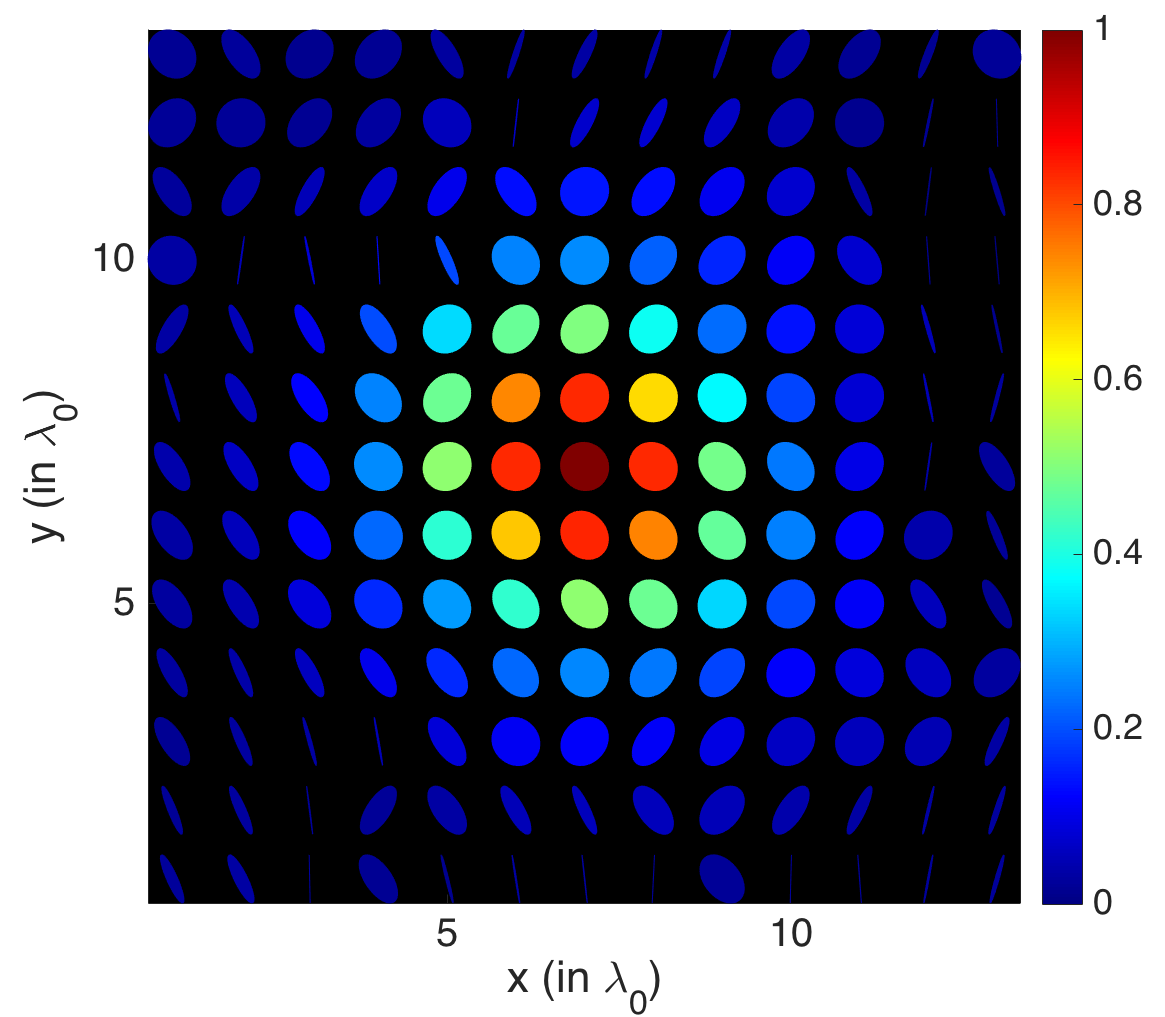}
\end{center}
\caption{Experiment identical to that in figure
\ref{fig.crossrangeellipse}, but with noise in the data and SNR of $10$ dB.}
\label{fig.crossrangeellipsebruit}
\end{figure}

\begin{figure}[h!]
\begin{center}
\includegraphics[width=6.1cm]{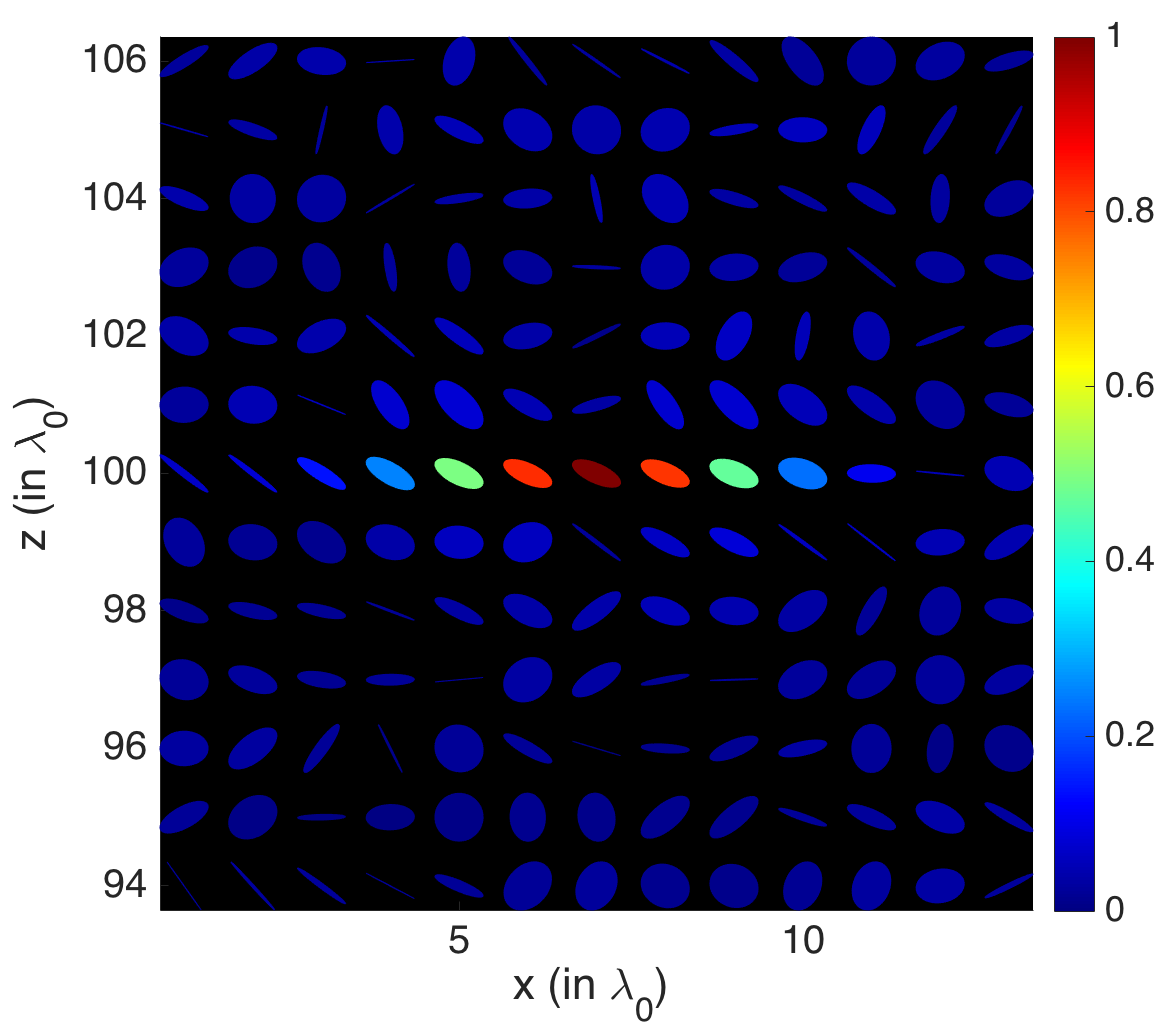}
\hspace{-0.1cm}
\includegraphics[width=6.1cm]{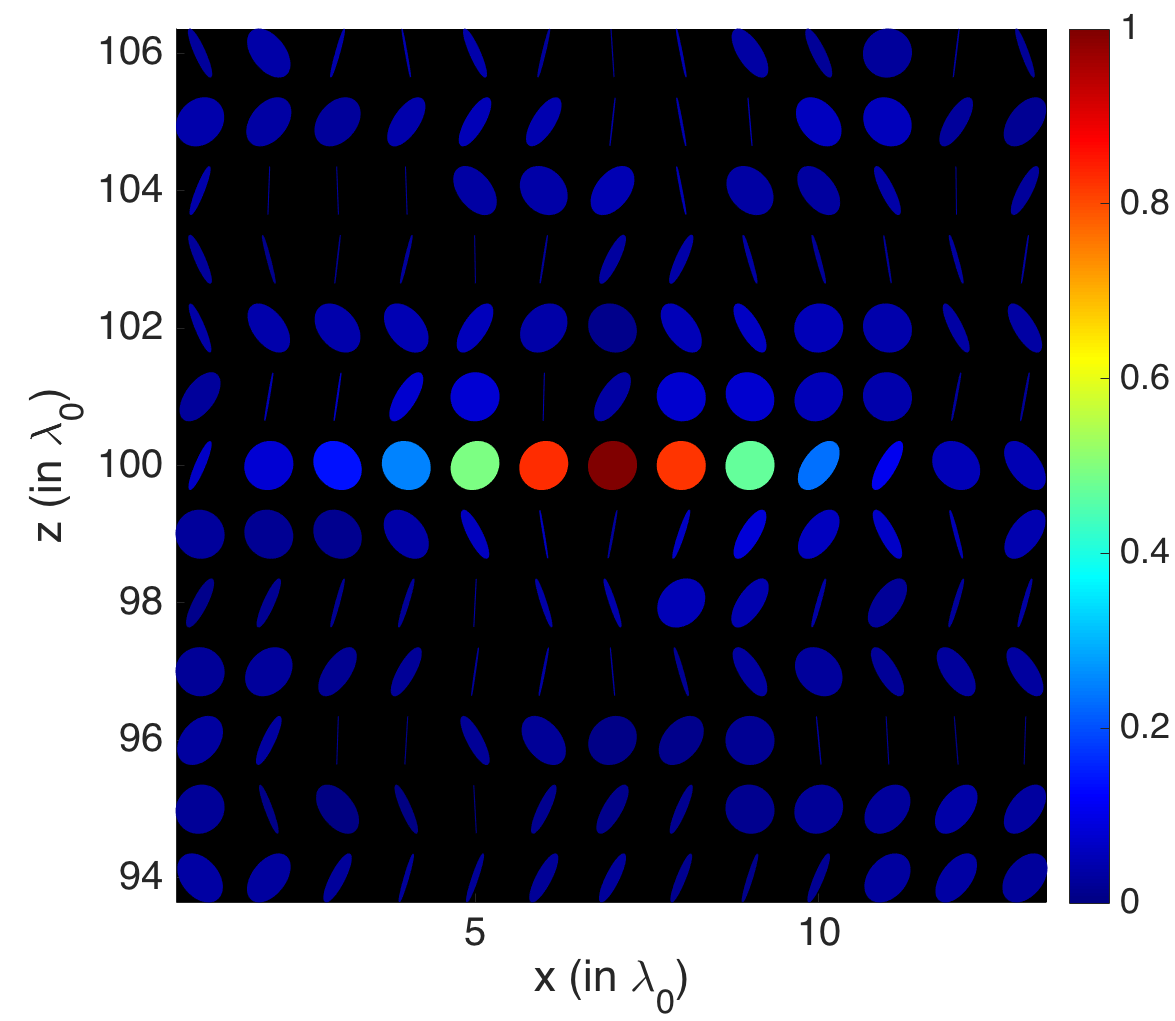}
\end{center}
\caption{Experiment identical to that in figure \ref{fig.rangeellipse},
but with noise and SNR of $10$ dB.}
\label{fig.rangeellipsebruit}
\end{figure}

\subsubsection{Experiments with an extended dipole distribution}
\label{sec:numexp:extended:active}
In the figures \ref{fig.cubictarget1}, \ref{fig.cubictarget2}  and
\ref{fig.cubictarget3} we deal  with the case of a volumetric
polarizability tensor distribution within a cube of side $5\lambda_0$,
with center $(0,0,L)$. This distribution is uniform and generated by a set of dipoles separated by a wavelength of $\lambda_0/4$. All the dipoles are assumed to have the same $3\times 3$ complex symmetric polarizability tensor 
 \[
 \bGa_* = \begin{bmatrix}
    2+1i & 1 & 0\\
    1 & 2+2i & 0\\
    0 & 0 & 1+1i
 \end{bmatrix}.
 \]
The array $\cA$, the central frequency $f_0$ and the frequency band are
identical to the ones of figure \ref{fig.phase}.

In figure \ref{fig.cubictarget1} (left), we observe a good
reconstruction of the location of the scatterer in the cross-range plane
$z=0$. The right figure shows that only the edges of the cube are imaged
in the range direction. Thus the Kirchhoff imaging function detects only
the discontinuities of an extended scattered (as in acoustics
\cite{Blei:2013:MSI}). In contrast to point-like scatterers, the
information about the norm $\|\bGa_{1:2,1:2}\|$ is lost here but one
still observes a very good estimation of the orientations after
normalization and up to a complex sign.  The quality of the
reconstruction can also be seen in figures \ref{fig.cubictarget2}  and
\ref{fig.cubictarget3} which shows the reconstructed tensors are stable
in the vicinity of the scatterer in cross-range (figure
\ref{fig.cubictarget2}) and range (figure \ref{fig.cubictarget3}).

\begin{figure}[h!]
\begin{center}
\includegraphics[width=6.1cm]{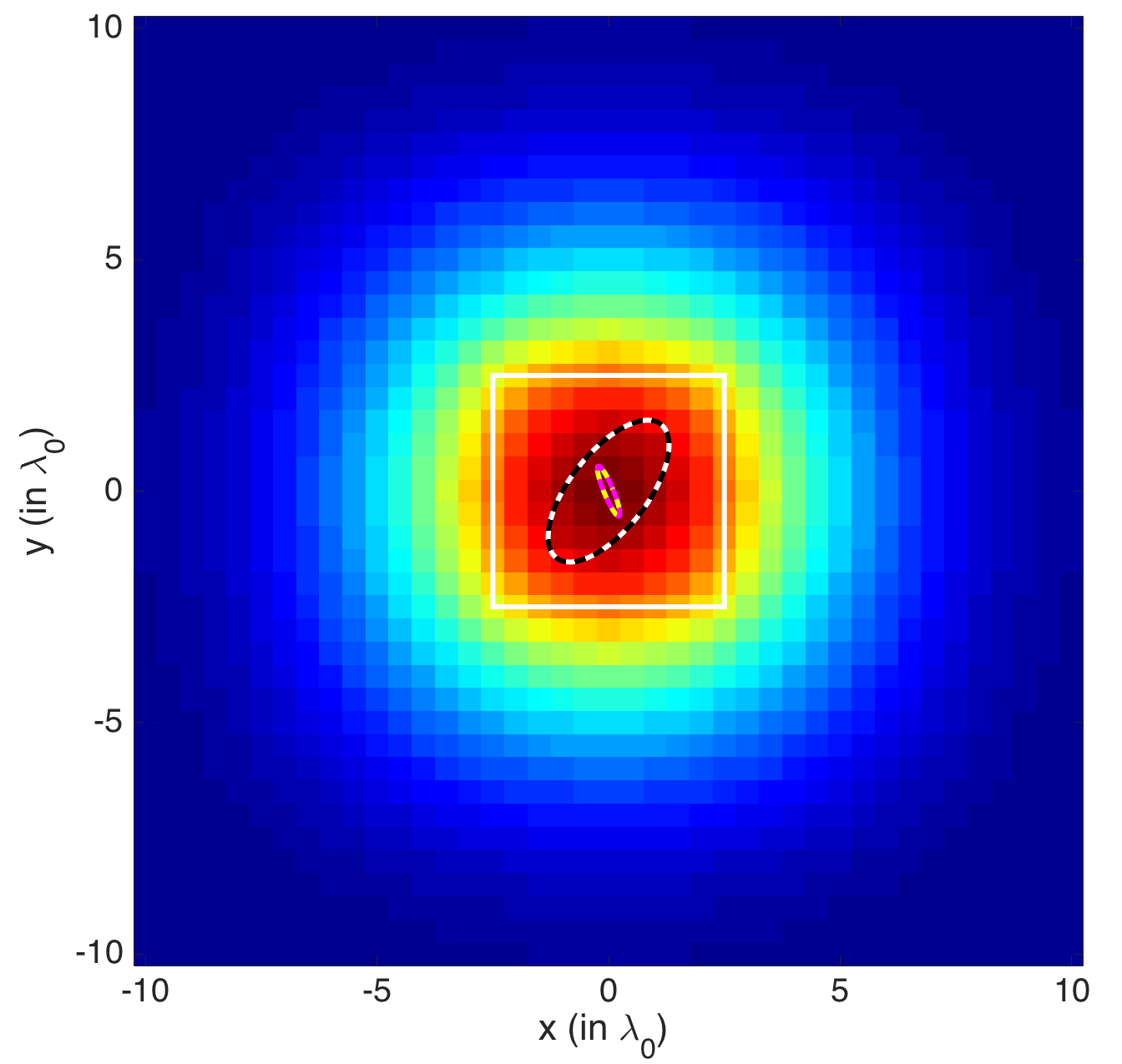}
\hspace{0.1cm}
\includegraphics[width=6.1cm]{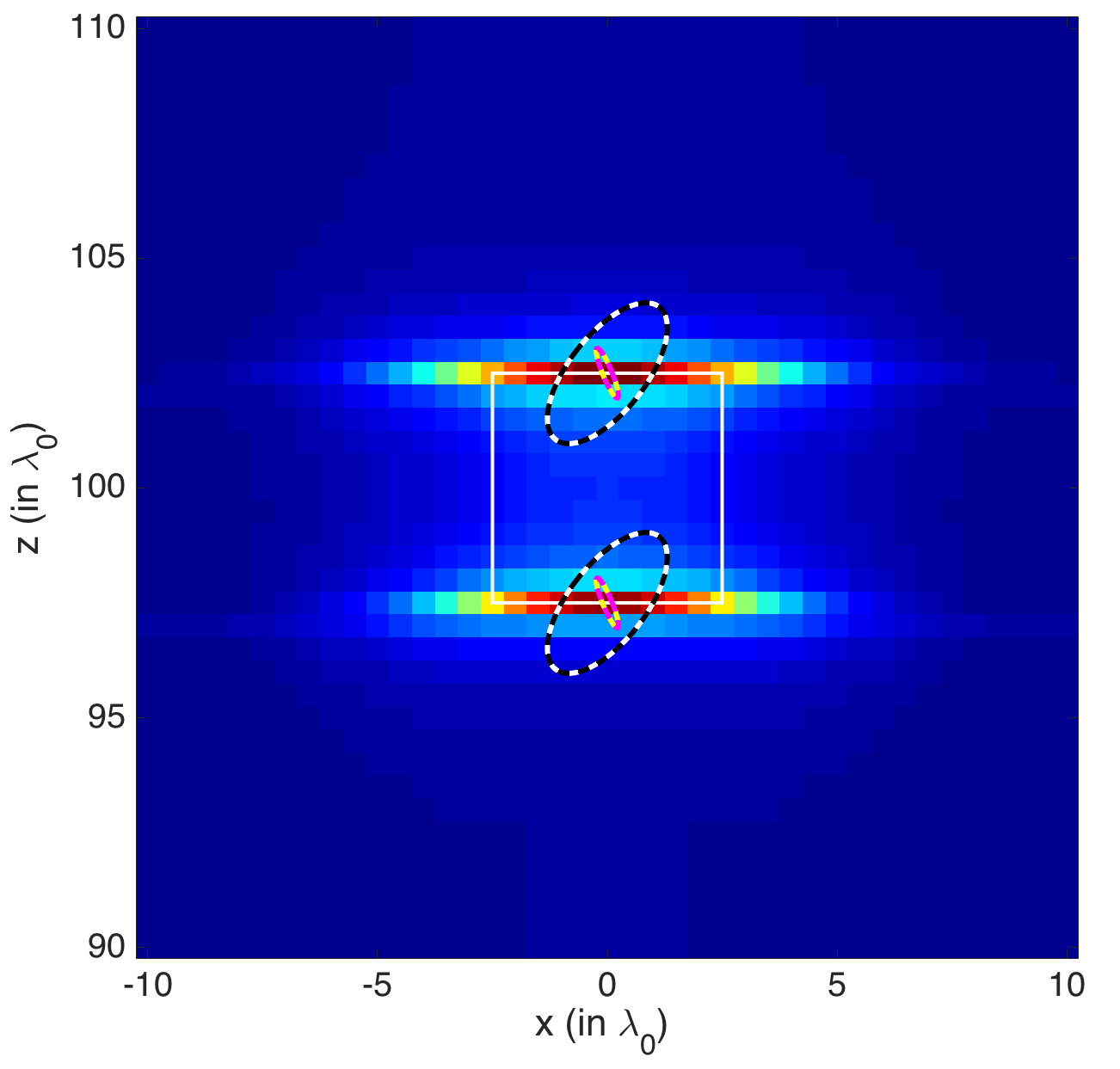}
\end{center}
\caption{Image of $\|\bGa_{1:2,1:2}\|$ (color scale) in the plane $z=L$
(left) and in the plane $y=0$ (right). Visualization of
$\operatorname{Re}(e^{\mi \theta} \bGa_{1:2,1:2})$ with ellipses (white
ones for the exact tensor and black ones for the reconstructed one) and
of $\operatorname{Im}(e^{\mi \theta} \bGa_{1:2,1:2})$ (yellow ellipses
for the exact tensor and purple ones for the reconstructed one). These
ellipses are represented at the center of the cube (left) and at the
centers of the $z=\pm2.5\lambda_0$ faces (right).}
\label{fig.cubictarget1}
\end{figure}

\begin{figure}[h!]
\begin{center}
\includegraphics[width=6.1cm]{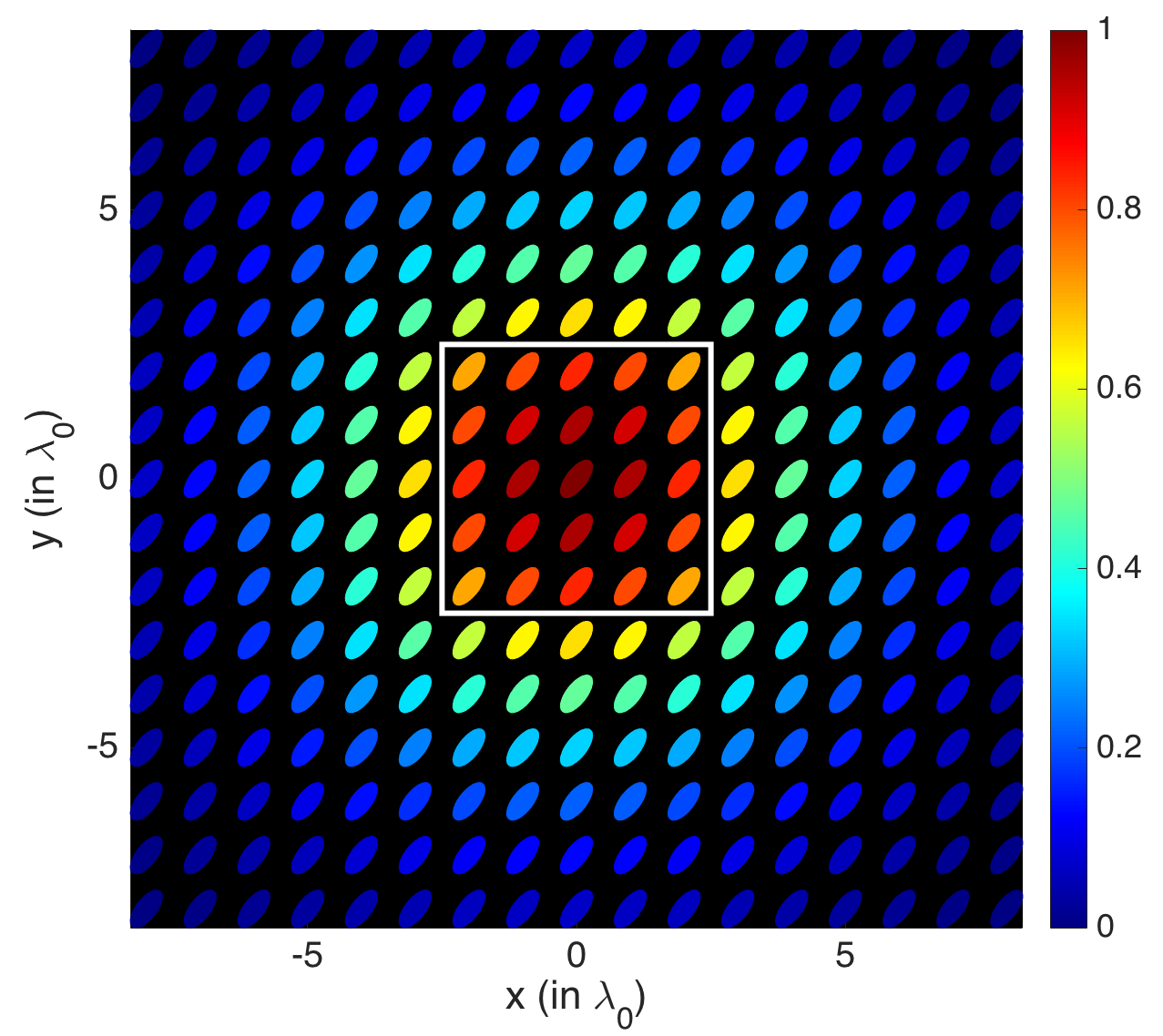}
\hspace{0.1cm}
\includegraphics[width=6.0cm]{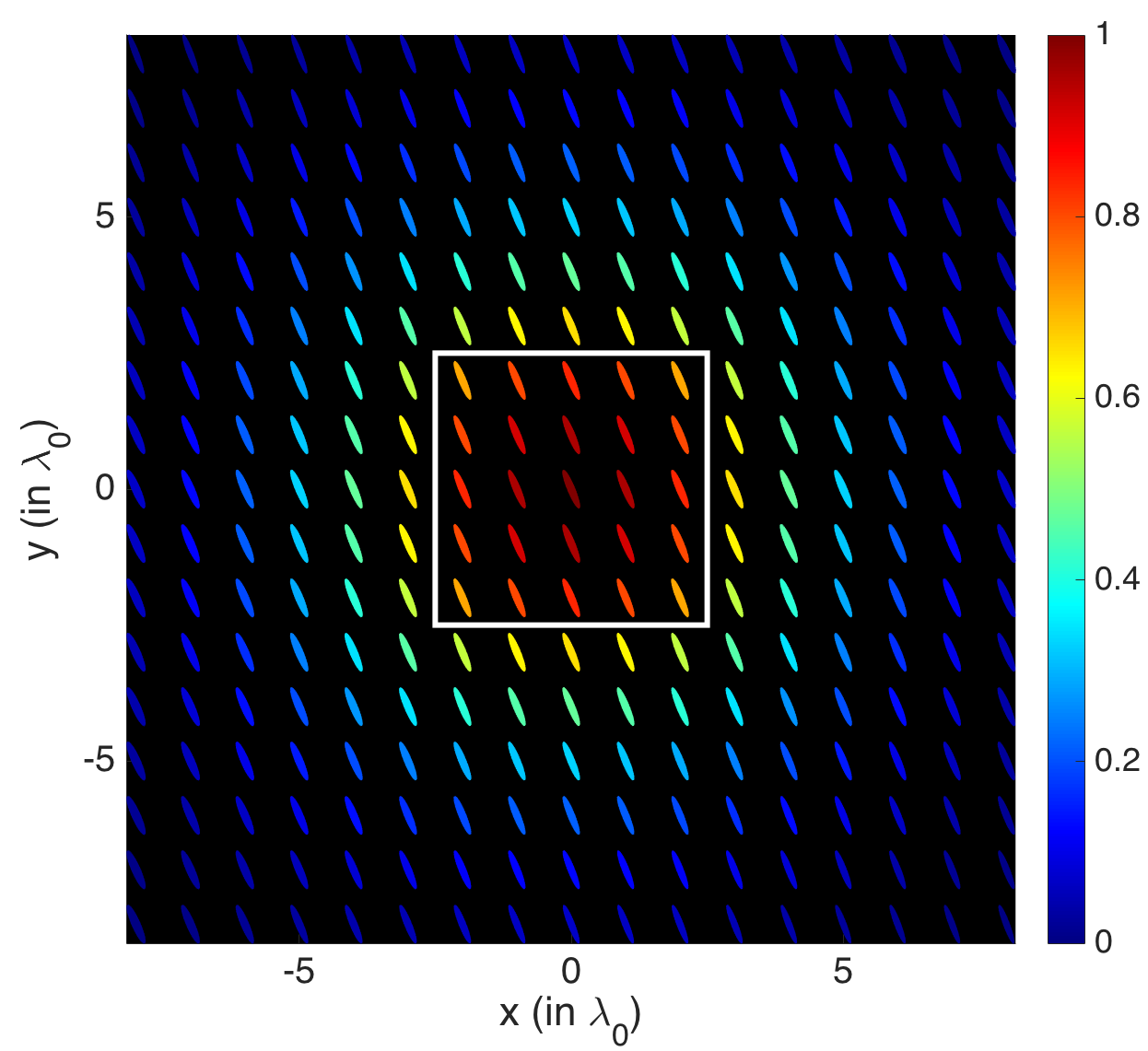}
\end{center}
\caption{Visualization of
$\operatorname{Re}((\overline{\bGa_{1,1}}/|\bGa_{1,1}|) \,
\bGa_{1:2,1:2})$ (left) and
$\operatorname{Im}((\overline{\bGa_{1,1}}/|\bGa_{1,1}|) \,
\bGa_{1:2,1:2})$ (right) in the plane $z=0$.}
\label{fig.cubictarget2}
\end{figure}

\begin{figure}[h!]
\begin{center}
\includegraphics[width=6.1cm]{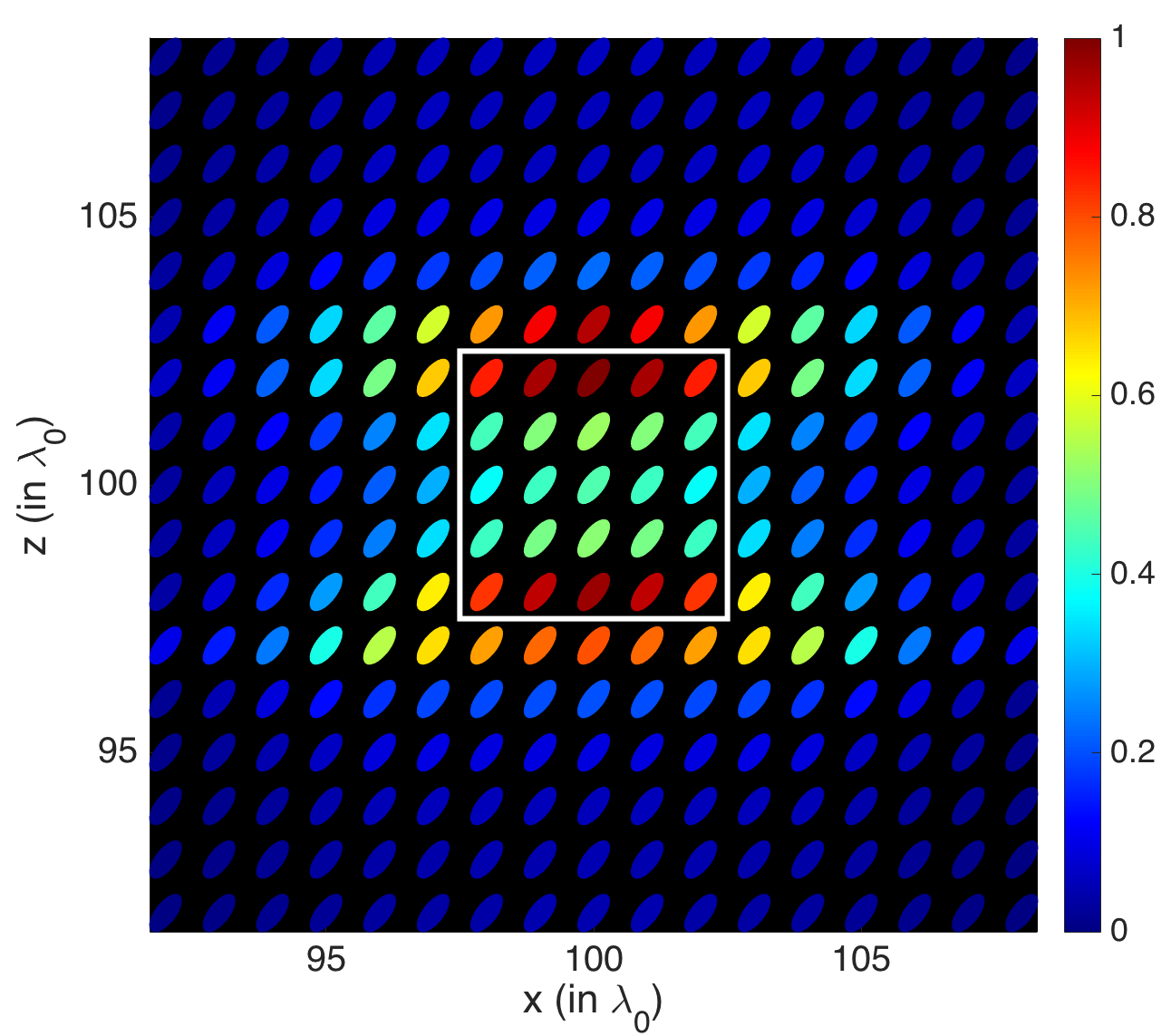}
\hspace{0.1cm}
\includegraphics[width=6.cm]{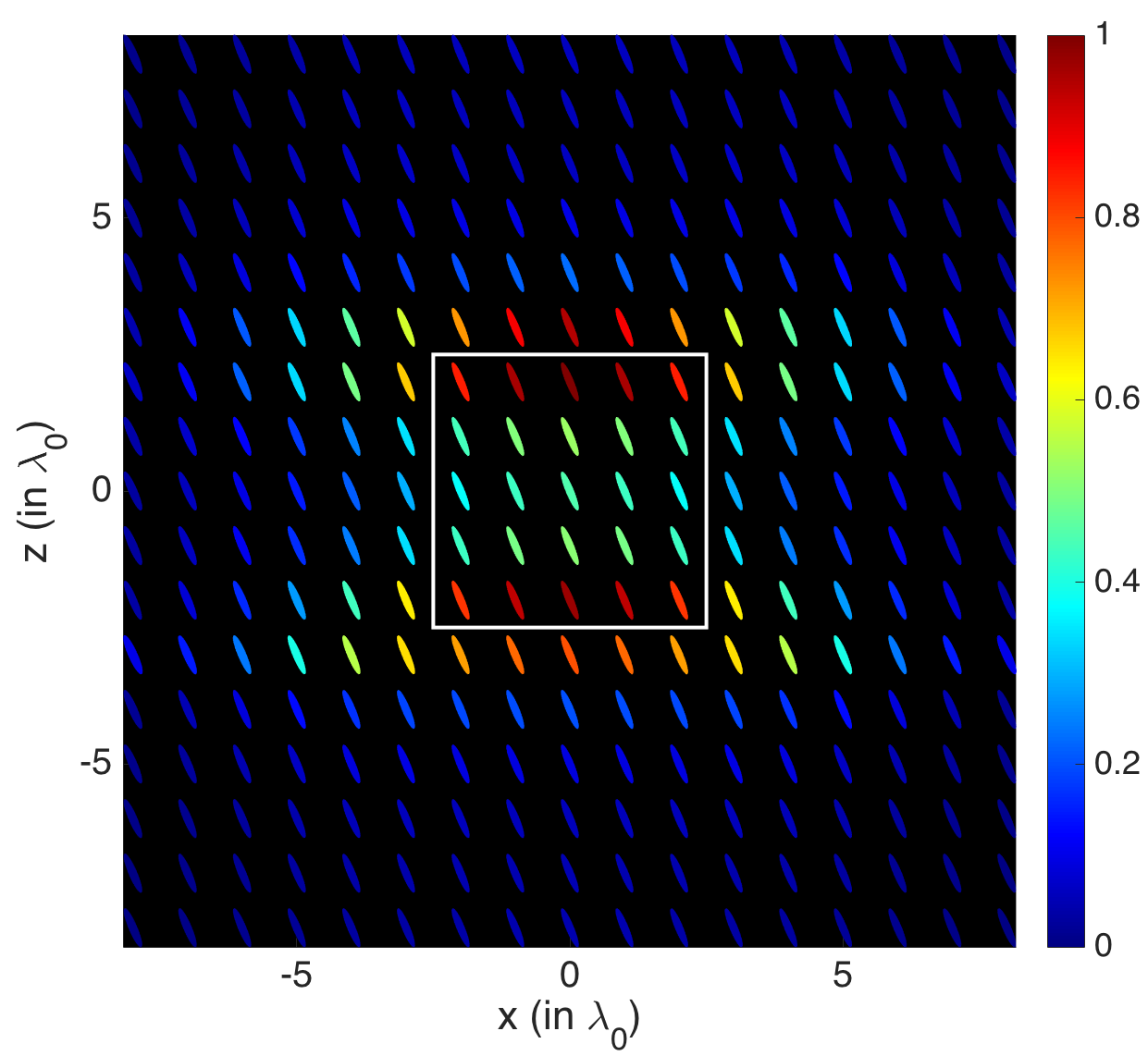}
\end{center}
\caption{Visualization of
$\operatorname{Re}((\overline{\bGa_{1,1}}/|\bGa_{1,1}|) \,
\bGa_{1:2,1:2})$ (left) and
$\operatorname{Im}((\overline{\bGa_{1,1}}/|\bGa_{1,1}|) \,
\bGa_{1:2,1:2})$ (right) in the plane $y=0$.}
\label{fig.cubictarget3}
\end{figure}

\section{Summary and future work}
\label{sec:future}

We used the Fraunhofer asymptotic to study an electromagnetic version of
the Kirchhoff imaging function for both the passive and active imaging
problems. The images we obtain are vector (resp. matrix) valued in the
passive (resp. active case). The norm of these images behaves like the
images in acoustics, meaning that we get identical resolution estimates
for the position of well-separated dipoles (or scatterers) as those we
would obtain in acoustics. The vector (or matrix) valued image contains
information about the polarization vector (or polarizability tensor) of
the point sources (or small scatterers) in the medium. We show how to
extract this information to stably image these quantities. Our
asymptotic reveals that we can only expect to stably recover the
cross-range components of these quantities. The range components are
lost.

This is a first step to understand what quantities can be imaged in an
idealized case. Indeed, the data we used may be complicated to acquire
in practice, as we need to measure at least the two cross-range
components of the electric field in the passive case. In the active case
we also need experiments where one can control the cross-range
polarization vector components of the array sources. 

We are currently adapting this imaging technique to a case where only
cross-correlations of the electric field are available at the array,
i.e. we assume we can only measure the {\em electric coherence matrix} 
\[
C ( \vx_r, \vx_s, \tau)_{i,j} = \left[ \left\langle E_i^*(\vx_r,t) E_j
(\vx_s,t+\tau) \right\rangle \right]_{i,j}, \mbox{ for } i,j=1,2
\] 
for all points $\vx_r, \vx_s$ in the array. Here
$\langle\,\cdot\,\rangle$ represents averaging over realizations of the
electromagnetic field (which is assumed stationary so there is no
dependence on $t$).  This data is equivalent to knowing the {\em Stokes
parameters} (see e.g. \cite{Mandel:OCQ:1995}), which characterize the
polarization of electromagnetic waves.  The technique we plan to use is a
generalization of the imaging with cross-correlations method for
acoustics in \cite{Bardsley:2016:KIW} to the Maxwell equations.

\appendix

\section{Proof of lemma  \ref{lem.crossrange1}}
\label{app:cross:passive}

\begin{proof}
We consider a point $\vvy=(\vy, L+\eta)$ with $\vy=(y_1,y_2)$.
In cylindrical coordinates the array is given by
\[
\cA:=\Mcb{ \vvx_r=(r,\theta, 0) \in \real^3 \mid  0\leq r\leq a \, \mbox{ and } \, \theta \in [0,2\pi)}.
\]
Thus,  it leads to the following expression of $\widetilde{\mH}(\vvy_*,\vvy_*;k)$ in cylindrical coordinates:
\begin{eqnarray} \label{eq.inttermprinc}
\widetilde{\mH}(\vvy_*,\vvy_*;k)&=&\frac{1}{(4\pi L)^2}\int_{0}^{a}r \md r
\int_{0}^{2\pi} \md \theta\, \mP(\vvx_r,\vvy_*)  \nonumber\\
&=&\frac{  a^2 }{16 \pi L^2} \mI -\frac{ 1}{(4\pi L)^2}\int_{0}^{a} r \md r \int_{0}^{2\pi}
\md \theta\,
\frac{(\vvx_r-\vvy_*)(\vvx_r-\vvy_*)^{\tr}}{\|\vvx_r-\vvy_*\|^2}
\end{eqnarray}
with
$$\vvx_r-\vvy_*=(r \cos(\theta)-y_{1,*}, r\sin(\theta)-y_{2,*},-L-\eta_*).$$
Integrating in $\theta$ and using the Fraunhofer asymptotic to simplify
the resulting expression and  finally integrating in $r$ leads to
\begin{eqnarray} \displaystyle \label{eq.projpolar}
&&\frac{1}{(4\pi L)^2}\int_{0}^{a} r \md r \int_{0}^{2\pi}\md \theta \,
\frac{(\vvx_r-\vvy_*)(\vvx_r-\vvy_*)^{\tr}}{\|\vvx_r-\vvy_*\|^2} \\ [5pt]
 &&=   \frac{a^2}{L^2}\begin{pmatrix}
\displaystyle  \cO\left(\frac{a^2}{L^2} \right) &\displaystyle
\cO\left(\frac{b^2}{L^2} \right)+\cO\left(\frac{a^4}{L^4} \right)&
\displaystyle \cO\left(\frac{b}{L} \right)+\cO\left(\frac{a^3}{L^3} \right)\\[10pt]
 \displaystyle  \cO\left(\frac{b^2}{L^2} \right)+\cO\left(\frac{a^4}{L^4}
 \right)  &\displaystyle  \cO\left(\frac{a^2}{L^2} \right)  &
 \displaystyle \cO\left(\frac{b}{L} \right)+\cO\left(\frac{a^3}{L^3}\right)  \\[10pt]
 \displaystyle \cO\left(\frac{b}{L} \right)+\cO\left(\frac{a^3}{L^3}
 \right)&  \displaystyle \cO\left(\frac{b}{L}
 \right)+\cO\left(\frac{a^3}{L^3} \right) & \displaystyle
 \frac{1}{16\pi}+ \cO\left(\frac{h}{L} \right)+ \cO\left(\frac{a^2}{L^2} \right)\nonumber 
\end{pmatrix}.
\end{eqnarray}
The formula (\ref{eq.asymp}) follows by combining
(\ref{eq.inttermprinc}), (\ref{eq.projpolar}) and using the
relation: $\cO(ha^2/L^3)=\cO( (k h a^4)/(L^4 \Theta_a))=o( a^4/L^4)$.
\end{proof}

\section{Polarizability tensors used in numerical experiments}
\label{app:matrices}

 In section~\ref{sec:numexp:active}, the polarizability tensors that we
 used are the $3 \times 3$ complex symmetric matrices given by
 \[
 \begin{aligned}
  \bGa_1 &= \begin{bmatrix}
   2 + \mi &  1 & 0\\
   1 &  2 + 2\mi & 0\\
   0 & 0 &(1+\mi)/2
   \end{bmatrix},
   ~
   \bGa_2 = \begin{bmatrix}
   2 + 2\mi & - 1\mi & 1/2\\
   -1\mi & 1+\mi & 0\\
   1/2 & 0 & 1
   \end{bmatrix}\\
   \text{and}
   ~
   \bGa_3 &= \begin{bmatrix}
   1 + 2\mi & 1 & \mi/2\\
   1 & 3+ 2\mi & 0\\
   \mi/2 & 0 &  \mi/2
   \end{bmatrix}.
 \end{aligned}
\] 

\section*{Acknowledgements}
The work of M. Cassier and F. Guevara Vasquez was partially supported by
the National Science Foundation grant DMS-1411577. FGV would like to
thank Arnold D. Kim for inspiring conversations about this problem.

\bibliographystyle{abbrvnat}
\bibliography{polbib}

\end{document}